\documentclass[11pt, reqno]{amsart}

\usepackage{pstricks} 

\usepackage[all,arc]{xy}
\usepackage{enumerate}
\usepackage{mathrsfs}

\usepackage{amsmath}
\usepackage{amsthm}
\usepackage{amsfonts}
\usepackage{amssymb}
\usepackage{amsbsy}
\usepackage{graphicx}
\usepackage{enumitem}
\usepackage{physics}
\usepackage{hyperref}
\usepackage{bbm}
\usepackage{longtable}

\newcommand{\ms}[1]{\mathscr{#1}}
\newcommand{\mc}[1]{\mathcal{#1}}

\newcommand{\floor}[1]{ \lfloor #1 \rfloor }

\newcommand{\varep}{ \varepsilon }

\newcommand{\sse} {\subseteq}

\newcommand{\Z}{\mathbb{Z}}
\newcommand{\R}{\mathbb{R}}

\newcommand{\C}{\mathbb{C}}
\newcommand{\E}{\mathbb{E}}

\newcommand{\ra}{\rightarrow}
\newcommand{\toinf}{\ra \infty}

\newcommand{\beq}{\begin{equation}}
\newcommand{\eeq}{\end{equation}}
\newcommand{\mbf}[1]{\mathbf{#1}}
\newtheorem{theorem}{Theorem}
\newtheorem{prop}[theorem]{Proposition}
\newtheorem{lemma}[theorem]{Lemma}
\newtheorem{cor}[theorem]{Corollary}
\theoremstyle{definition}
\newtheorem{definition}[theorem]{Definition}
\newtheorem{remark}[theorem]{Remark}

\newcommand{\p}{\mathbb{P}}

\numberwithin{equation}{section}
\numberwithin{theorem}{subsection}

\newcommand{\largebox}{\Lambda}
\newcommand{\lbox}{\largebox}
\newcommand{\vertices}{\lbox_0}
\newcommand{\edges}{{\lbox_1}}
\newcommand{\plaquettes}{{\lbox_2}}
\newcommand{\wloop}{\gamma}
\newcommand{\supp}{\mathrm{supp}}
\newcommand{\mbbm}[1]{\mathbbm{#1}}
\newcommand{\ind}{\mbbm{1}}
\newcommand{\oneskel}{S_1(\lbox)}
\newcommand{\twoskel}{S_2(\lbox)}
\newcommand{\twoskelminus}[1]{S_2(\plaquettes \backslash #1)}
\newcommand{\Hom}{\mathrm{Hom}}
\newcommand{\homsym}{\psi}
\newcommand{\twoskelhom}{\zeta}
\newcommand{\groupid}{1}
\newcommand{\upath}{w}
\newcommand{\randomv}{\mathbf{\Gamma}}
\newcommand{\elemmatrix}{A_\beta}
\newcommand{\knotbound}{10^{24}}

\newcommand{\betathreshnonab}{126 \frac{(\knotbound \alpha_\beta)^7}{1 - \knotbound \alpha_\beta}}

\newcommand{\plaqset}{P}
\newcommand{\vortex}{V}
\newcommand{\scale}{s}
\newcommand{\vbdconst}{20e}
\newcommand{\betaexprone}[1]{20e^{#1+1} \alpha_\beta}
\newcommand{\betaexprtwo}[1]{\frac{420e^{#1 +1} \alpha_\beta}{1 - \betaexprone{#1}}}

\begin{document}

\title[Wilson loop expectations for finite gauge groups]{Wilson loop expectations in lattice gauge theories with finite gauge groups}
\author{Sky Cao}
\address{\newline Department of Statistics \newline Stanford University\newline Sequoia Hall, 390 Jane Stanford Way \newline Stanford, CA 94305\newline \newline \textup{\tt skycao@stanford.edu}}
\thanks{Research was supported by NSF grant DMS-1501767}
\keywords{Lattice gauge theory, Wilson loop, cluster expansion, Stein's method, Poisson approximation.}
\subjclass[2010]{70S15, 81T13, 81T25, 82B20}

\begin{abstract}
Wilson loop expectations at weak coupling are computed to first order, for four dimensional lattice gauge theories with finite gauge groups which satisfy some mild additional conditions. This continues recent work of Chatterjee, which considered the case of gauge group $\Z_2$. The main steps are (1) reducing the first order computation to a problem of Poisson approximation, and (2) using Stein's method to carry out the Poisson approximation. 
\end{abstract}


\maketitle

\tableofcontents

\section{Introduction}\label{section:intro}

\subsection{Motivation}

Lattice gauge theories are models from physics which are obtained by discretizing continuous spacetime using a lattice. The effect of this discretization is that path integrals which were originally over infinite dimensional spaces become ordinary integrals over finite dimensional spaces. This ensures that lattice gauge theories are mathematically well-defined, unlike their continuum counterparts, Euclidean Yang-Mills theories. However, ultimately the goal is to rigorously construct a Euclidean Yang-Mills theory (and then show that it can be used to construct a Quantum Yang-Mills theory). Naturally, one might hope to do so by discretizing spacetime by a lattice, then sending the lattice mesh size to zero (so that the lattice ``converges" to continuous spacetime), and finally taking the limit of the corresponding lattice gauge theories. This is analogous to constructing Brownian motion by taking a limit of random walks. In order for this approach to work, various properties of lattice gauge theories must be very well understood. Building on recent work of Chatterjee \cite{Ch2019}, this paper seeks to improve our understanding of one particular property of lattice gauge theories, to be described shortly. For more background and a review of the existing results involving lattice gauge theories, see Chatterjee's survey \cite{Ch2018} and the references therein, in particular Seiler's monograph \cite{SEI1982} and the book by Glimm and Jaffe \cite{GJ1987}.

\subsection{Main result}

We first define lattice gauge theories. Let $G$ \label{notation:G} be a compact group, with the identity denoted by $\groupid$\label{notation:group-id}. We will commonly refer to $G$ as the gauge group. Let $\rho$\label{notation:rho} be a unitary representation of $G$, with dimension $d$, and let $\chi$\label{notation:chi} be the character of $\rho$. Take a finite lattice 
\[ \lbox\label{notation:Lambda} := ([a_1, b_1] \times \cdots \times [a_4, b_4])  \cap \Z^4, \]
where $b_i - a_i$ is the same for all $1 \leq i \leq 4$. Let $E(\lbox)$ be the set of directed nearest-neighor edges in $\lbox$. Let $\Sigma(\lbox)$ be the set of functions $\sigma : E(\lbox) \ra G$, with the constraint that for any edge $(x, y) \in E(\lbox)$, we have $\sigma_{(x, y)} = \sigma_{(y, x)}^{-1}$. We will commonly refer to $\sigma$ as an edge configuration. By a ``plaquette" $p$\label{notation:plaquette} in $\lbox$, we mean a unit square whose four boundary edges are in $\lbox$. Let $P(\lbox)$ be the set of plaquettes in $\lbox$. For $p \in P(\lbox)$, suppose the vertices of $p$ are $x_1, x_2, x_3, x_4$, in (say) counter-clockwise order. In an abuse of notation, for $\sigma \in \Sigma(\lbox)$, define
\beq\label{eq:sigma-p-def} \sigma_p := \sigma_{(x_1, x_2)} \sigma_{(x_2, x_3)} \sigma_{(x_3, x_4)} \sigma_{(x_4, x_1)}. \eeq
Define
\[ S_\lbox(\sigma) := \sum_{p \in P(\lbox)} \Re (\chi(1) - \chi(\sigma_p)). \]
Let $\mu_\lbox$ be the product Haar measure on $\Sigma(\lbox)$. For $\beta \geq 0$\label{notation:beta}, let $\mu_{\lbox, \beta}$\label{notation:mu-lbox-beta} be the probability measure defined by
\beq\label{eq:def-lgt} d\mu_{\lbox, \beta}(\sigma) := Z_{\lbox, \beta}^{-1} ~e^{-\beta S_\lbox(\sigma)} d\mu_\lbox(\sigma), \eeq
where $Z_{\lbox, \beta}$ is the normalizing constant.
We say that $\mu_{\lbox, \beta}$ is the lattice gauge theory with gauge group $G$, on $\lbox$, with inverse coupling constant $\beta$. In this paper we will work in the large $\beta$ regime, which is also known as the weak coupling regime.

The choices of the gauge group $G$ which are most relevant to physics include $U(1)$, $SU(3)$, and $SU(3) \times SU(2) \times U(1)$. However, in this paper, we will be forced to make a mathematical simplification, and assume that $G$ is finite. This simplification may lessen the direct relevance to physics, but finite gauge groups have in fact been previously studied in both the mathematics and physics literature; see e.g. \cite{AF1984, Borgs1984, Ch2019, CJR1979, LMR1989, MP1979, SV1989, WEG1971} for an incomplete list.
It remains to be seen whether the methods of this paper may be extended to continuous groups. The difficulty that continuous groups pose will be described at the end of this subsection, after a brief proof sketch of the main result is given.

Hereafter, assume $G$ is finite. In this case, we can specify that the Haar measure $\mu_\lbox$ is simply counting measure. Having defined lattice gauge theories, we now turn to the key objects of interest. Let $\wloop$\label{notation:wloop} be a closed loop in $\lbox$, denoted by its sequence of directed edges $e_1, \ldots, e_n$. The length of $\wloop$ is the number of edges in $\wloop$. We say that $\wloop$ is a self avoiding loop if no edge is repeated (ignoring orientation, so that if $(x, y)$ is in $\wloop$, then $(y, x)$ cannot also be in $\wloop$). Given an edge configuration $\sigma \in \Sigma(\lbox)$, define the Wilson loop variable $W_\wloop$\label{notation:W-wloop} by
\[ W_\wloop = W_\wloop(\sigma) := \chi(\sigma_{e_1} \cdots \sigma_{e_n}). \]
We will also commonly write
\beq\label{eq:ordered-product} \prod_{e \in \wloop} \sigma_e := \sigma_{e_1} \cdots \sigma_{e_n},\eeq
i.e. the product over the edges in $\wloop$ is understood to be an ordered product.

Now given $\lbox, \beta, \wloop$, let $\langle W_\wloop \rangle_{\lbox, \beta}$ be the expectation of $W_\wloop$ under the lattice gauge theory $\mu_{\lbox, \beta}$. Ideally, we want to take $\lbox \uparrow \Z^4$, and work with an infinite volume limit of the lattice gauge theories. However, as mentioned in \cite{Ch2018}, in general it is not known whether such infinite volume limits are unique. Failing this, we will work with subsequential limits, which always exist, by a compactness argument (given in Section 6 of \cite{Ch2019} for the case $G = \Z_2$, and which is easily generalized to finite groups). Hereafter, let $\langle W_\wloop \rangle_\beta$ denote the expectation of $W_\wloop$ under some subsequential limit of the finite volume lattice gauge theories $\mu_{\lbox, \beta}$, as $\lbox \uparrow \Z^4$.
Actually, in the case that the gauge group is Abelian, infinite volume limits do always exist, at least if the unitary representation $\rho$ is one dimensional. This seems to be well known, and was e.g. pointed out by Seiler \cite{SEI1982}, as well as by Fr\"{o}hlich and Spencer \cite{FS1982}. One proceeds by using the Ginibre inequality \cite{GIN1970} to show that $\langle W_\wloop \rangle_{\lbox, \beta}$ is monotone in $\lbox$.

Chatterjee \cite{Ch2019} recently obtained a first order expression for the Wilson loop expectation $\langle W_\wloop\rangle_\beta$ in the case $G = \Z_2 = \{\pm 1\}$, at large $\beta$. We extend this result to any finite group $G$, under certain mild conditions, to be given shortly. The motivation for performing this calculation is that one approach to defining a continuum limit is to define Wilson loop expectations in the continuum by taking a limit of Wilson loop expectations in lattice gauge theories (see Section 5 of \cite{Ch2018}). In order to do so, we must have a precise understanding of the latter. Also, the constraint that $\beta$ be large might not be too restrictive, since it is believed that for physically relevant lattice gauge theories such as $G = SU(3)$, as the lattice mesh size is taken to zero, we must simultaneously take $\beta \ra \infty$ (see Section 3 of \cite{Ch2018}).

We are almost ready to state the main result. First, define
\[ \Delta_G \label{notation:Delta-G} := \min_{g \neq 1} \Re (\chi(1) - \chi(g)). \]
Here the minimum is also implicitly over $g \in G$. Note the condition $\Delta_G > 0$ is equivalent to the condition that $\rho$ is faithful. Now for $\beta \geq 0$, $g \in G$, define
\beq\label{eq:phi-beta-def}  \varphi_\beta(g) := \exp(-\beta\Re (\chi(1) - \chi(g))), \eeq\label{notation:varphi-beta} 
\beq\label{eq:r-beta-def}  r_\beta := \sum_{g \neq 1} \varphi_\beta(g)^6, \eeq \label{notation:r-beta}
\beq\label{eq:elemmatrix-def} \elemmatrix := r_\beta^{-1} \sum_{g \neq 1} \rho(g) \varphi_\beta(g)^6. \eeq \label{notation:elemmatrix}
Note since $\rho$ is a unitary representation, we have that $\elemmatrix$ is Hermitian, and moreover $\norm{\elemmatrix}_{op} \leq 1$ for all $\beta \geq 0$. Let $\lambda_1(\beta), \ldots, \lambda_d(\beta)$ be the eigenvalues of $\elemmatrix$. Let $G_0 := \{g \in G : g \neq 1, \Re (\chi(1) - \chi(g)) = \Delta_G \}$, and let
\[ A  := \frac{1}{|G_0|} \sum_{g \in G_0} \rho(g). \]
Observe $\lim_{\beta \toinf} \elemmatrix = A$ entrywise, and consequently $A$ is also Hermitian.

\begin{theorem}\label{thm:main-result}
Let $G$ be a finite group, and let $\rho$ be a unitary representation of $G$ of dimension $d$. Suppose $\rho$ is faithful, so that $\Delta_G > 0$. Suppose
\beq\label{eq:beta-thresh-main-result} \beta \geq \frac{1}{\Delta_G} (1000 +  14\log \abs{G}), \eeq
and also suppose $\norm{A_\beta}_{op} < 1$. Let
\beq\label{eq:main-result-c-beta-def} c_\beta := \min(0.15, 2^{-19} \log \norm{\elemmatrix}_{op}^{-1}, 1 - \norm{\elemmatrix}_{op}). \eeq
For any self avoiding loop $\wloop$ in $\Z^4$ of length $\ell$\label{notation:ell}, we have
\beq\label{eq:main-result-wilson-loop-bound} \bigg|\langle W_\wloop \rangle_\beta - \sum_{i=1}^d e^{-\ell r_\beta (1 - \lambda_i(\beta))}\bigg| \leq (2e +2) d e^{-\beta \Delta_G c_\beta / (3 + 2c_\beta)}. \eeq
\end{theorem}

\begin{remark}
To be clear, the bound \eqref{eq:main-result-wilson-loop-bound} holds for any subsequential limit of the finite volume lattice gauge theories $\mu_{\lbox, \beta}$, as $\lbox \uparrow \Z^4$. 
Also, the $\beta$ threshold \eqref{eq:beta-thresh-main-result} is very loose. It should be possible to tighten this somewhat by a more careful argument.
\end{remark}

\begin{remark}
If $\norm{A}_{op} < 1$, then $\norm{A_\beta}_{op} < 1$ for all large enough $\beta$, and also $c_\beta$ will stay bounded away from zero as $\beta$ goes to infinity. Somewhat ironically, this theorem does not directly cover the case $G = \Z_2$, since this group has only a single non-identity element, which leads to $\|A\|_{op} = 1$. However, if in the error bound \eqref{eq:main-result-wilson-loop-bound} we allow ourselves an extra error term that accounts for the ``roughness" of the loop $\wloop$, then by a small additional argument, we can handle this case (note Theorem 1.1 of \cite{Ch2019} also includes this roughness penalty).

For groups with more than one non-identity element, i.e. groups of order at least 3, there always exists a faithful unitary representation for which $\norm{A}_{op} < 1$; see Lemma \ref{lemma:exist-unitary-rep}. Additionally, the condition $\norm{A}_{op} < 1$ is satisfied for any pair $(G, \rho)$ such that $G$ is of order at least 3, and $\rho$ is an irreducible faithful unitary representation of $G$; see Lemma \ref{lemma:faithful-unitary-irrep}.

Furthermore, in the case that $\rho$ is irreducible, $A$ becomes a very simple matrix. This is because $G_0$ must be a union of conjugacy classes of $G$, which implies that for any $h \in G$, we must have $\rho(h) A \rho(h^{-1}) = A$, i.e. $A$ is an intertwiner. Thus by Schur's lemma, $A$ must be a multiple $\lambda$ of the identity matrix. This value $\lambda$ may then be calculated from the character table of $\rho$. The matrix $A_\beta$ may be similarly calculated.
\end{remark}

\begin{remark}
I expect that the methods of this paper can be adapted to other dimensions besides four, and possibly also other types of lattices, though the specific formulas for Wilson loop expectations will likely change. 
\end{remark}

\begin{remark}
There is also a recent article by Forsstr\"{o}m, Lenells, and Viklund \cite{FLV2020}, which handles the case of finite Abelian gauge groups. They do obtain a much better $\beta$ threshold in this setting.
\end{remark}


To interpret Theorem \ref{thm:main-result}, suppose we are taking $\beta \toinf$. Then in order for $\langle W_\wloop \rangle_\beta$ to have nontrivial behavior (i.e. not tending to 0 or $d$), we must take $\ell \sim \alpha r_\beta^{-1}$ for some $\alpha > 0$, which gives
\[ \langle W_\wloop \rangle_\beta \sim \sum_{i=1}^d e^{- \alpha (1 - \lambda_i(\beta))}. \]
Letting $\lambda_1, \ldots, \lambda_d$ be the eigenvalues of $A$, and recalling $\elemmatrix \ra A$ entrywise, we then have
\[ \langle W_\wloop \rangle_\beta  \sim \sum_{i=1}^d e^{-\alpha (1 - \lambda_i)}.\]

It should be noted that Theorem 1.1 of \cite{Ch2019} additionally covers the case of multiple disjoint self avoiding loops. In this paper, for simplicity, we will just focus on a single self avoiding loop. When the gauge group is Abelian, the proof of the more general case is essentially the same, but when the gauge group is non-Abelian, there are some additional complications that affect both the statement and the proof of the main result. I do however anticipate that these complications are secondary to the main arguments behind the first order computation.

Instead of immediately proving Theorem \ref{thm:main-result} for general finite groups, I have decided to first consider the case where the gauge group is Abelian, at the expense of making the paper longer. This is because it turns out that the main probabilistic insights needed are already all present in the Abelian case. The extension from Abelian to non-Abelian then mainly consists of verifying that the main insights still hold, even if the gauge group is non-Abelian.

I will now try to give a quick preview of the proof, starting in the Abelian case. First, at large $\beta$, it turns out that it is more natural to study the finite volume lattice gauge theory $\mu_{\lbox, \beta}$ not in terms of the edge configuration $\sigma$, but rather in terms of a certain random collection of surfaces. This is analogous to studying low temperature Ising models on $\Z^2$ in terms of a random collection of loops. Now here is the key insight, which is already present in Chatterjee's work \cite{Ch2019}: to first order, the Wilson loop expectation is the moment generating function of a Poisson random variable, and this we know how to exactly evaluate. The key step in showing this claim is showing that a certain random variable defined in terms of the random collection of surfaces is approximately Poisson. To do so, we use the dependency graph approach to Stein's method for Poisson approximation \cite{CR2013}, and to verify the conditions of Stein's method, we use cluster expansion \cite{BFP2010}. Indeed, here is why it is so useful to work in terms of the random collection of surfaces: the cluster expansion allows us to show that the random collection of surfaces has ``a lot of independence", in a sense to be made precise in the proof. It is this independence that makes it possible to perform a first order calculation.

If we now consider general gauge groups, we can still express the Wilson loop expectation in terms of a random collection of surfaces. Except now, to first order, Wilson loop expectations are traces of expectations of a fixed matrix raised to a Poisson random variable, but this is still exactly evaluable. As before, the key to the Poisson approximation is to take advantage of the fact that the random collection of surfaces has a lot of independence. However, the non-Abelianness of the situation makes verifying this independence more difficult, due to the presence of additional topological considerations. Central to the handling of these additional topological considerations is the observation by Szlach\`{a}nyi and Vecserny\`{e}s \cite{SV1989} that the difficulties which arise can be understood in terms of algebraic topology.

Finally, the reason that we are forced to consider finite gauge groups is as follows. When the gauge group is continuous, the random collection of surfaces is actually not random at all: with probability 1, it will consist of every plaquette in the lattice $\lbox$. Thus in this case there is no information contained about Wilson loop expectations. 

As we begin the rest of the paper, let me note that there is an index of notation at the end, which contains much of the notation introduced in the course of the paper. 


\section{Discrete exterior calculus}\label{section:dec}

This section collects the basic statements about discrete exterior calculus which will be needed throughout this paper. We mostly follow Section 2 of \cite{Ch2019}. 

\subsection{The lattice cell complex in four dimensions}\label{section:lattice-cell-complex}

Let $x \in \Z^4$. There are four edges coming out of $x$ in a positive direction. Denote these edges $dx_1, dx_2, dx_3, dx_4$. For $1 \leq k \leq 4$ and $1 \leq i_1 < \cdots < i_k \leq 4$, the edges $dx_{i_1}, \ldots, dx_{i_k}$ determine a positively oriented $k$-cell of $\Z^4$ (which can be visualized as a $k$ dimensional unit cube). Denote this $k$-cell by
\[ dx_{i_1} \wedge \cdots \wedge dx_{i_k}.\]
Denote the negatively oriented version of this $k$-cell by $-dx_{i_1} \wedge \cdots \wedge dx_{i_k}$. More generally, for $1 \leq i_1, \ldots, i_k \leq 4$, define $dx_{i_1} \wedge \cdots \wedge dx_{i_k}$ to be zero if the $i_1, \ldots, i_k$ are not all distinct, and otherwise define it to be $(-1)^m dx_{j_1} \wedge \cdots \wedge dx_{j_k}$, where $j_1 < \cdots < j_k$ is the sorted version of $i_1, \ldots, i_k$, and $m$ is the sign of the permutation that takes $i_1, \ldots, i_k$ to $j_1, \ldots, j_k$. In the case $k = 0$, the positively oriented 0-cell associated to $x$ will be denoted $x$, while the negatively oriented 0-cell will be denoted $-x$ (not to be confused with the vertex obtained by multiplying all coordinates by $-1$). To help visualize, a 0-cell is a vertex, a 1-cell is an edge, and a 2-cell is a plaquette.

A rectangle $B$ in $\Z^4$ is a set of the form
\[ ([a_1, b_1] \times \cdots \times [a_4, b_4]) \cap \Z^4, \]
where the $a_i$'s and $b_i$'s are all integers. If the side lengths $b_i - a_i$ are the same for all $i$, then $B$ is said to be a cube. Now let $c$ be a $k$-cell. We say that $c$ is in $B$ if every vertex of $c$ is a vertex of $B$. We say that $c$ is on the boundary of $B$ if every vertex of $c$ is a boundary vertex of $B$. We say that $c$ is in the interior of $B$ if $c$ is in $B$, but is not on the boundary of $B$. We say that $c$ is outside $B$ if there is at least one vertex of $c$ that is not a vertex of $B$.
For rectangles $B, B'$ in $\Z^4$, we say that $B$ is contained in $B'$, or $B$ is in $B'$, or $B'$ contains $B$, if every vertex of $B$ is in $B'$.

\subsection{Discrete differential forms, derivative, and coderivative}

Let $G$ be an additive Abelian group. A $G$-valued $k$-form $f$ is an odd $G$-valued function on the set of oriented $k$-cells of $\Z^4$. Here, ``odd" means that for any oriented $k$-cell $c$, we have $f(-c) = -f(c)$. 
We may also define $k$-forms on a cube $B$. I.e., a $G$-valued $k$-form on $B$ is an odd $G$-valued function on the oriented $k$-cells of $B$. 


Now let $1 \leq k \leq 4$. We begin to define the exterior derivative of a $k$-form. First, given a positively oriented $k$-cell $c$, we want to express the collection of oriented $(k-1)$-cells which are ``contained" in $c$. Denote $c = dx_{i_1} \wedge \cdots \wedge dx_{i_k}$, where $i_1 < \cdots < i_4$. For $1 \leq j \leq k$, let
\[ c^{(j)} := (-1)^j ~ dx_{i_1} \wedge \cdots \wedge \widehat{dx_{i_j}} \wedge \cdots \wedge dx_{i_k}\]
denote the oriented $(k-1)$-cell obtained from $c$ by omitting the edge $dx_{i_j}$, and adjusting the orientation by the factor $(-1)^{j}$. For $1 \leq \ell \leq k$, let $e_\ell$ be the $\ell$th standard basis vector, i.e. the vector with 1 in the $\ell$th coordinate, and 0 in all other coordinates. Let $x^{(\ell)} := x + e_\ell$, and define $c_{i_j} := dx^{(i_j)}_{i_1} \wedge \cdots \wedge dx^{(i_j)}_{i_k}$.

For $2 \leq k \leq 4$, the oriented $(k-1)$-cells 
\[ c^{(i_1)}, \ldots, c^{(i_k)}, -c^{(i_1)}_{i_1}, \ldots, -c^{(i_k)}_{i_k} \]
are said to be contained in $c$, respecting orientation. If $k = 1$, then $c = dx_i$ for some $1 \leq i \leq 4$. The oriented $0$-cells which are contained in $dx_i$, respecting orientation, are then defined to be $x + e_i$, and $-x$. I.e., the incoming vertex of $dx_i$ is positively oriented, while the outgoing vertex of $dx_i$ is negatively oriented.

As an example, given a positively oriented plaquette, there are two positively oriented edges and two negatively oriented edges contained in this plaquette. These four edges naturally form a closed loop on the boundary of the plaquette. Now given a positively oriented $k$-cell $c$, and an oriented $(k-1)$-cell $c'$, define
\[ I(c', c) := \begin{cases} 1 & \text{$c'$ is contained in $c$, respecting orientation} \\ -1 & \text{$-c'$ is contained in $c$, respecting orientation} \\ 0 & \text{otherwise}. \end{cases}\]
If $c$ is negatively oriented, define $I(c', c) := I(-c', -c)$. Observe also that $I(c', c) = -I(-c', c)$, and thus also $I(c', c) = -I(c', -c)$. Throughout this paper, we will often not worry about orientation, and we will say that $c'$ is contained in $c$, or $c$ contains $c'$, if $I(c', c) \neq 0$.

Fix a cube $B$. For $0 \leq k \leq 3$, and a $G$-valued $k$-form $f$ on $B$, define the exterior derivative of $f$, denoted by $df$, as the following $G$-valued $(k+1)$-form on $B$. Given an oriented $(k+1)$-cell $c$, define
\[\label{notation:df} (df)(c) := \sum_{\substack{c' \text{ a pos. orient.} \\ \text{$k$-cell in $B$}}} I(c', c) f(c'). \]
One may verify that $df(-c) = -df(c)$, so that $df$ is in fact odd. Observe that this definition of the exterior derivative coincides with the definition given in Section 2.3 of \cite{Ch2019} (in the special case $n = 4$, and $0 \leq k \leq 3$), although the formulas are stated with different notation. This different notation will be useful in Section \ref{section:stokes} when we state and prove a discrete Stokes' theorem. We thus have the following results from \cite{Ch2019}.

\begin{lemma}[Lemma 2.1 of \cite{Ch2019}]\label{lemma:exact-implies-closed}
Let $G$ be an additive Abelian group, $B$ a cube, and $0 \leq k \leq 3$. Let $f$ be a $G$-valued $k$-form on $B$. Then $ddf = 0$.
\end{lemma}

\begin{lemma}[Poincar\'{e} lemma, Lemma 2.2 of \cite{Ch2019}]
Let $G$ be an additive Abelian group, $B$ a cube, and $1 \leq k \leq 3$. For any $G$-valued $k$-form $f$ on $B$ such that $df = 0$, there exists a $G$-valued $(k-1)$-form $h$ on $B$ such that $f = dh$. Moreover, if $G$ is finite, then the number of such $h$ is the same for any such $f$. Finally, if $f$ vanishes on the boundary of $B$, then $h$ may be taken to also vanish on the boundary of $B$.
\end{lemma}

We now begin to define the coderivative. Let $1 \leq k \leq 4$. Let $f$ be a $G$-valued $k$-form. Define the coderivative of $f$ to be the $G$-valued $(k-1)$-form, denoted $\delta f$, defined as follows. For an oriented $(k-1)$-cell $c$, define
\[ \label{notation:delta-f} (\delta f)(c) := \sum_{\substack{c' \text{ a pos.} \\ \text{orient. $k$-cell}}} I(c, c') f(c'). \]
Again, this definition of $\delta f$ coincides with the definition given in Section 2.5 of \cite{Ch2019}. We thus have the following result from \cite{Ch2019}.

\begin{lemma}[Poincar\'{e} lemma for the coderivative, Lemma 2.7 of \cite{Ch2019}]
Let $G$ be an additive Abelian group. Let $1 \leq k \leq 3$. Let $f$ be a $G$-valued $k$-form which is zero outside of a finite region. Suppose $\delta f = 0$. Then there is a $G$-valued $(k+1)$-form $h$ such that $f = \delta h$. Moreover, if $f$ is zero outside a cube $B$, then $h$ may be taken to be zero outside $B$ as well.
\end{lemma}

\subsection{Discrete Stokes' theorem}\label{section:stokes}

Let $G$ be an additive Abelian group. Let $B$ be a cube. Let $f$ be a $G$-valued 1-form on $B$. Let $g$ be a $\Z$-valued 1-form on $\Z^4$, which is zero outside of $B$, and such that $\delta g = 0$. By the Poincar\'{e} lemma for the coderivative, there is a $\Z$-valued 2-form $h$ such that $\delta h = g$, and $h$ is zero outside of $B$. Define
\[ \label{notation:inner-product-f-g} \langle f, g \rangle = \langle f, \delta h \rangle := \sum_{\substack{e \text{ a pos. orient.} \\ \text{1-cell of $B$}}} \delta h(e) f(e)\]
Similarly, define
\[ \langle df, h \rangle := \sum_{\substack{p \text{ a pos. orient.} \\ \text{2-cell of $B$}}} h(p) df(p). \]
The following result is a discrete version of Stokes' theorem.

\begin{lemma}[Stokes' theorem]
We have $\langle f, \delta h\rangle = \langle df, h\rangle$.
\end{lemma}
\begin{proof}
Since $h$ is zero outside of $B$, for any oriented 1-cell $e$ of $B$, we have
\[ \delta h (e) = \sum_{\substack{p \text{ a pos. orient.} \\ \text{$2$-cell of $B$}}} I(e, p) h(p).\]
Now proceed by exchanging the order of summation. 
\end{proof}

\section{The Abelian case}\label{section:abelian-case}

\subsection{Setup and notation}

We first introduce notation which is slightly different from the notation of Section \ref{section:intro}. Let $G$ be an Abelian group, now with addition as the group operation, and the identity denoted by $0$. Since we are considering a special case anyway, let us also assume that $\rho$ is a one-dimensional unitary representation of $G$. In this case $\chi = \rho$, $\chi(0) = 1$, and $\chi$ is multiplicative, i.e. $\chi(g + g') = \chi(g) \chi(g')$ for all $g, g' \in G$. For an example to keep in mind, take $G$ to be the integers modulo $n$, and take $\rho$ to be the representation given by $\rho(k) = e^{2\pi i k / n}$. Observe that we may write
\beq\label{eq:delta-g-def-abelian} \Delta_G = \min_{g \neq 0} \Re (1 - \chi(g)),\eeq
\[ \elemmatrix = r_\beta^{-1} \sum_{g \neq 0} \chi(g) \varphi_\beta(g)^6,\]
\[ A = \frac{1}{|G_0|} \sum_{g \in G_0} \chi(g). \]
Moreover, $A \in [-1, 1]$, and $A_\beta \in [-1, 1]$ for all $\beta \geq 0$. 

Let $\largebox$ be a cube. Let $\edges$ be the set of positively oriented edges of $\largebox$, and let $\plaquettes$ be the set of positively oriented plaquettes of $\largebox$. Note any $\sigma \in G^\edges$ may be naturally extended to a $G$-valued 1-form on $\lbox$, by setting $\sigma_{-e} := - \sigma_e$ for $e \in \edges$. Thus in a slight abuse of notation, we will often think of $\sigma$ as an actual 1-form. Similarly, we will think of elements $q \in G^\plaquettes$ as actual 2-forms.

For $\sigma \in G^\edges$, $p \in \plaquettes$, observe that $\sigma_p$ (recall equation \eqref{eq:sigma-p-def}) is the same as $(d\sigma)_p$. Thus for $\beta \geq 0$, we have that $\mu_{\largebox, \beta}$ may be viewed as a probability measure on $G^\edges$, given by (recall equation \eqref{eq:phi-beta-def} for the definition of $\varphi_\beta$, and also recall that the base measure $\mu_\lbox$ is specified to be counting measure)
\[ \mu_{\largebox, \beta} (\sigma) = Z_{\largebox, \beta}^{-1} \prod_{p \in \plaquettes} \varphi_\beta((d\sigma)_p), \]
where 
\[ Z_{\lbox, \beta} = \sum_{\sigma \in G^\edges}\prod_{p \in \plaquettes} \varphi_\beta((d\sigma)_p). \]

Now let $\wloop = e_1 \cdots e_n$ be a self avoiding loop in $\lbox$ of length $\ell$. In an abuse of notation, we will identify $\wloop$ with its induced $\Z$-valued 1-form: for any oriented 1-cell $e$ in $\Z^4$, we have $\wloop_e = 1$ if $e$ appears in $\wloop$, $\wloop_e = -1$ if $-e$ appears in $\wloop$, and $\wloop_e =0$ otherwise (note as $\wloop$ is self avoiding, we have $\wloop_e \in \{-1, 0, 1\}$ for all $e$). By the expression ``$e \in \wloop$", we mean that $\wloop_e \neq 0$, i.e. the edge $e$ (in some orientation) is in the loop $\wloop$. Similarly, ``$e \notin \wloop$" denotes $\wloop_e = 0$.
Now because $G$ is an additive group, we may express the Wilson loop variable $W_\wloop : G^\edges \ra \C$ by
\[ W_\wloop(\sigma) = \chi(\langle \sigma, \wloop\rangle) . \]
One may think of the expression $\langle \sigma, \wloop \rangle$ as an integral of $\sigma$ over the closed loop $\wloop$. Now as $\wloop$ has length $\ell$, there is a cube $B_\wloop$ contained in $\lbox$ of side length $\ell$, which contains $\wloop$. Consequently, the associated 1-form is zero outside $B_\wloop$. Moreover, since $\wloop$ starts and ends at the same point, we have $\delta \wloop = 0$. Thus the Poincar\'{e} lemma for the coderivative implies that there is a $\Z$-valued 2-form $S$ (``$S$" for surface), which is zero outside $B_\wloop$ (and thus also outside $\lbox$), such that $\delta S = \wloop$. Stokes' theorem then implies
\[ W_\wloop(\sigma) = \chi(\langle d \sigma, S \rangle) = \prod_{p \in \plaquettes} \chi(S_p (d\sigma)_p).  \]
Now to calculate the infinite volume expectation $\langle W_\wloop \rangle_\beta$, it suffices to calculate the finite volume expectation $\langle W_\wloop \rangle_{\lbox, \beta}$, for all large enough $\lbox$. The following theorem does just that. Proving this is the main focus of Section \ref{section:abelian-case}.

\begin{theorem}\label{thm:main-abelian}
Let $G$ be a finite Abelian group, and let $\rho$ be a faithful one-dimensional unitary representation of $G$. (Note since $\rho$ is faithful we have $\Delta_G > 0$.) Let $\lbox$ be a cube in $\Z^4$ with side length $N$\label{notation:N}. Let $\wloop$ be a self avoiding loop of length $\ell$ in $\lbox$, and let $B_\wloop$\label{notation:B-wloop} be a cube in $\lbox$ of side length $\ell$, which contains $\wloop$. Let $L$\label{notation:L} be the $\ell^\infty$ distance between the boundary of $\lbox$ and the boundary of $B_\wloop$, and suppose $L \geq 50$ (say). Suppose
\[ \beta \geq \frac{1}{\Delta_G} (60 +  14\log (\abs{G} - 1)), \]
and also suppose that $|\elemmatrix|< 1$. 
\beq\label{eq:c-beta-def} c_\beta :=  \min(0.15, 0.5 \log |\elemmatrix|^{-1}, 1 - \elemmatrix). \eeq
Then
\[|\langle W_\wloop \rangle_{\lbox, \beta} - e^{-\ell r_\beta (1 - \elemmatrix)}| \leq (2e + 2) (e^{-\beta \Delta_G / 2}+N^4 e^{-\beta L \Delta_G / 2})^{c_\beta / (1.5 + c_\beta)}.\]
\end{theorem}

\begin{remark}
Note if $|A| < 1$, then $|\elemmatrix| < 1$ for all $\beta$ large enough, and also $c_\beta$ will stay bounded away from zero as $\beta \toinf$. Additionally, as we take $\lbox \uparrow \Z^4$, we will have that $L$ is on the order of $N$, so that the term $N^4 e^{-\beta L \Delta_G / 2}$ will disappear. Actually, this term can be done away with entirely, as we will see when we consider the general case in Section \ref{section:non-abelian-case}. However, in the present special case, getting rid of the term requires an additional argument which is not central to the overall picture, and thus I have decided against including this additional argument.
\end{remark}

\begin{remark}
Actually, in order for a finite Abelian group to have a faithful one dimensional representation, it must be cyclic. However, the proof of Theorem \ref{thm:main-abelian} does not directly use this fact, and thus I have decided against explicitly putting this in the theorem assumptions.
\end{remark}

\begin{remark}
The assumption that $\rho$ is one-dimensional is primarily a simplifying assumption, and it can be removed by slight generalizations of the arguments of the present section. However, since consideration of the Abelian case before the fully general case is mainly intended to illuminate the needed probabilistic arguments, I decided to just assume that $\rho$ is one-dimensional to remove any added complexity. In any case, the main result of the paper (Theorem \ref{thm:main-result}) applies even when the dimension of $\rho$ is greater than one.
\end{remark}

\subsection{Setting up the cluster expansion}

We begin by rewriting the partition function so that we can later perform a cluster expansion. Full proofs will be given for completeness, but I should note that this is more or less standard and the ideas can be found e.g. in Section 3 of \cite{SEI1982}. For $q \in G^\plaquettes$, define $\supp(q) := \{p \in \plaquettes: q_p \neq 0\}$. For plaquette sets $\plaqset \sse \plaquettes$\label{notation:plaqset}, define
\[ \Phi(\plaqset)\label{notation:Phi-abelian} := \sum_{\substack{q \in G^\plaquettes \\ dq = 0\\ \supp(q) = \plaqset}} \prod_{p \in \plaqset} \varphi_\beta(q_p).\]
It will also be useful later on to define
\[ \Phi_S(\plaqset) := \sum_{\substack{q \in G^\plaquettes \\ dq = 0\\ \supp(q) = \plaqset}} \prod_{p \in \plaqset} \chi(S_p q_p) \prod_{p \in \plaqset} \varphi_\beta(q_p). \]

Note $\Phi, \Phi_S$ depend on $\lbox$, but we will hide this dependence. For any $q \in G^\plaquettes$ such that $dq = 0$, let $N_1$ be the number of $\sigma \in G^\edges$ such that $q = d \sigma$. Recall by the Poincar\'{e} lemma, $N_1$ does not depend on $q$. 

\begin{lemma}\label{lemma:abelian-cluster-expansion-setup}
We have
\[ Z_{\lbox, \beta} = N_1 \sum_{\plaqset \sse \plaquettes} \Phi(\plaqset). \]
\end{lemma}
\begin{proof}
Recall by Lemma \ref{lemma:exact-implies-closed} that for any $\sigma \in G^\edges$, we have $d d\sigma = 0$. Thus
\begin{align*}
Z_{\lbox, \beta} &= \sum_{\sigma \in G^\edges} \prod_{p \in \plaquettes} \varphi_\beta((d\sigma)_p) = \sum_{\substack{q \in G^\plaquettes \\ dq = 0}} \prod_{p \in \plaquettes} \varphi_\beta(q_p) \abs{\{\sigma : d\sigma = q\}} \\
&= N_1 \sum_{\substack{q \in G^\plaquettes \\ dq = 0}} \prod_{p \in \plaquettes} \varphi_\beta(q_p) 
= N_1 \sum_{\plaqset \sse \plaquettes} \sum_{\substack{q \in G^\plaquettes \\ dq = 0\\ \supp(q) = \plaqset}} \prod_{p \in \plaqset} \varphi_\beta(q_p),
\end{align*}
as desired. Note in the final equality, we used the fact that $\varphi_\beta(0) = 1$.
\end{proof}


The objects defined in the following definition will play a fundamental role throughout the rest of the paper, even in the general setting.

\begin{definition}
Given a plaquette set $\plaqset \sse \plaquettes$, we may obtain an undirected graph $G(\plaqset)$ as follows. The vertices of the graph are the plaquettes of $\plaqset$. Place an edge between any two plaquettes $p_1, p_2 \in \plaqset$ such that there is a 3-cell $c$ in $\lbox$ which contains both $p_1, p_2$. 

A vortex is a set $\vortex \sse \plaquettes$\label{notation:vortex} such that $G(\vortex)$ is connected. For general plaquette sets $\plaqset \sse \plaquettes$, we may partition $G(\plaqset)$ into connected components $G_1, \ldots, G_k$, which corresponds to a partition of $\plaqset$ into vortices $\vortex_1, \ldots, \vortex_k$, such that $G_i = G(\vortex_i)$ for all $1 \leq i \leq k$. Observe that as the partition of an undirected graph into connected components is unique, the resulting partition of $\plaqset$ into vortices is also unique. The name ``vortex" is inspired by similar usage in \cite{Ch2019}.

For plaquette sets $\plaqset_1, \plaqset_2 \sse \plaquettes$, denote $\plaqset_1 \sim \plaqset_2$, and call $\plaqset_1, \plaqset_2$ compatible, if there do not exist edges in $G(\plaqset_1 \cup \plaqset_2)$ between the subgraphs $G(\plaqset_1)$ and $G(\plaqset_2)$. Otherwise, denote $\plaqset_1 \nsim \plaqset_2$, and call $\plaqset_1, \plaqset_2$ incompatible. For $n \geq 3$, we will say that $\plaqset_1, \ldots, \plaqset_n$ are compatible if any pair $\plaqset_i, \plaqset_j$ is compatible, $1 \leq i \neq j \leq n$. Observe that if $\plaqset$ is partitioned into vortices $\vortex_1, \ldots, \vortex_k$ as described in the previous paragraph, then $\vortex_1, \ldots, \vortex_k$ are compatible.

Given a plaquette set $\plaqset \sse \plaquettes$, let the unique partition of $\plaqset$ into compatible vortices $\vortex_1, \ldots, \vortex_k$ as previously described be called the vortex decomposition of $\plaqset$. For each $1 \leq i \leq k$, we say that $\vortex_i$ is a vortex of $\plaqset$.
\end{definition}

The following lemma states that $\Phi$ and $\Phi_S$ have factorization properties that respect the notion of compatibility which we just defined. This is the crucial fact which will allow us to use general cluster expansion results later.

\begin{lemma}\label{lemma:activity_decomp}
Suppose $\plaqset_1, \ldots, \plaqset_n \sse \plaquettes$ are compatible. Then
\[\Phi(\plaqset_1 \cup \cdots \cup \plaqset_n) = \Phi(\plaqset_1) \cdots \Phi(\plaqset_n), \]
and
\[ \Phi_S(\plaqset_1 \cup \cdots \cup \plaqset_n) = \Phi_S(\plaqset_1) \cdots \Phi_S(\plaqset_n). \]
\end{lemma}

The key to proving this lemma is the following result.

\begin{lemma}\label{lemma:gauge_eq_decomp}
Suppose $\plaqset_1, \plaqset_2 \sse \plaquettes$ are compatible. Then any 2-form $q \in G^\plaquettes$ such that $dq = 0$, $\supp(q) = \plaqset_1 \cup \plaqset_2$ has a unique decomposition $q = q^1 + q^2$, where $q^i \in G^\plaquettes$, $dq^i = 0$, $\supp(q^i) = \plaqset_i$, $i = 1, 2$.
\end{lemma}
\begin{proof}
We first show existence. The natural candidate for the decomposition is 
\[ q^i := q \mbbm{1}_{\plaqset_i}, ~~ i = 1, 2.\]
Then clearly $q =q^1 + q^2$, and $\supp(q^i) = \plaqset_i$, $i = 1, 2$. It remains to show $dq^i = 0$. Take $i = 1$, say. We need to show that for every 3-cell $c$ in $\lbox$, we have
$(dq^1)_c = 0$.
We may split into two cases. In the first case, suppose $c$ contains a plaquette of $\plaqset_1$. Then since $\plaqset_1, \plaqset_2$ are compatible, we must have that $c$ cannot contain a plaquette of $\plaqset_2$. Thus for every plaquette $p$ in $c$, we have $q^1_p = q_p$, and thus $(dq^1)_c = (dq)_c = 0$. On the other hand, if the cube $c$ contains no plaquette of $\plaqset_1$, then $q^1_p = 0$ for every plaquette $p$ in $c$, and thus $(dq^1)_c = 0$.

To show uniqueness, suppose $q = q^1 + q^2 = \tilde{q}^1 + \tilde{q}^2$. Then $q^1 - \tilde{q}^1 = \tilde{q}^2 - q^2$.
We know that $\supp(q^1 - \tilde{q}^1) \sse \plaqset_1$, $\supp(\tilde{q}^2 - q^2) \sse \plaqset_2$, and $\plaqset_1, \plaqset_2$ are disjoint. Thus $q_1 - \tilde{q}^1 = 0$, and $\tilde{q}^2 - q^2 = 0$.
\end{proof}

\begin{proof}[Proof of Lemma \ref{lemma:activity_decomp}]
It suffices to take $n = 2$. I will give the proof for $\Phi$; the proof for $\Phi_S$ is essentially the same. Observe $\Phi(\plaqset_1 \cup \plaqset_2)$ involves a sum over $q$ such that $dq = 0, \supp(q) = \plaqset_1 \cup \plaqset_2$. By Lemma \ref{lemma:gauge_eq_decomp}, we have that any such $q$ has a unique decomposition $q = q^1 + q^2$, with $dq^i = 0$, $\supp(q^i) = \plaqset_i$, $i = 1, 2$. Moreover, for such a decomposition, we have
\[ \prod_{p \in \plaqset_1 \cup \plaqset_2} \varphi_\beta(q_p) = \prod_{p \in \plaqset_1} \varphi_\beta(q^1_p) \prod_{p \in \plaqset_2} \varphi_\beta(q^2_p).\]
From this, we see that the sum over $q$ factors into the product of sums over $q^1, q^2$, and the desired result follows.
\end{proof}

\subsection{Wilson loop expectation in terms of random collections of vortices}\label{section:wilson-loop-exp-abelian}

We now begin to write the Wilson loop expectation in terms of the vortices introduced in the previous subsection. First, we need to make the following key definition, which will play an important role throughout the rest of this paper, even in the general setting.

\begin{definition}
For a positively oriented edge $e$ in $\Z^4$, define $P(e)$ to be the set of positively oriented plaquettes $p$ in $\Z^4$ which contain $e$. We will call $P(e)$ a {minimal vortex}.
\end{definition}

The sense in which $P(e)$ is ``minimal" will become clear later. With $\lbox, \beta$ implicit, let $\Sigma$\label{notation:Sigma} be a random variable which has law $\mu_{\lbox, \beta}$. Define $\plaqset(\Sigma) := \supp(d\Sigma)$\label{notation:P-Sigma}. Given a realization of $\plaqset(\Sigma)$, we may consider the decomposition of $\plaqset(\Sigma)$ into compatible vortices. The set of vortices of $\plaqset(\Sigma)$ is the random collection of surfaces which was alluded to in the introduction. The following lemma expresses the distribution of $\plaqset(\Sigma)$ in terms of the function $\Phi$ which we defined earlier.

\begin{lemma}\label{lemma:v-sigma-law}
For any plaquette set $\plaqset \sse \plaquettes$, we have
\[\p(\plaqset(\Sigma) = \plaqset) = \frac{N_1 \Phi(\plaqset)}{Z_{\lbox, \beta}} = \frac{\Phi(\plaqset)}{\sum_{\plaqset' \sse \plaquettes} \Phi(\plaqset')}. \]
\end{lemma}
\begin{proof}
Essentially the same steps as in the proof of Lemma \ref{lemma:abelian-cluster-expansion-setup}.
\end{proof}


The following lemmas allow us to express $\E W_\wloop(\Sigma) = \langle W_\wloop \rangle_{\lbox, \beta}$ in terms of the random collection of surfaces $\plaqset(\Sigma)$.

\begin{lemma}\label{lemma:wilson-loop-cond-exp}
For any plaquette set $\plaqset \sse \plaquettes$ such that $\p(\plaqset(\Sigma) = \plaqset) > 0$, we have
\[ \E[W_\wloop(\Sigma) ~|~ \plaqset(\Sigma) = \plaqset] = \frac{\Phi_S(\plaqset)}{\Phi(\plaqset)}.\]
\end{lemma}
\begin{proof}
We have
\begin{align*}
\E [W_\wloop(\Sigma) \ind(\plaqset(\Sigma) = \plaqset)] &= \frac{1}{Z_{\lbox, \beta}}\sum_{\substack{\sigma \in G^\edges \\ \supp(d\sigma) = \plaqset}} \prod_{p \in \plaquettes} \chi(S_p(d\sigma)_p) \prod_{p \in \plaquettes} \varphi_\beta((d\sigma)_p) \\
&= \frac{N_1}{Z_{\lbox, \beta}} \sum_{\substack{q \in G^\plaquettes \\ dq = 0 \\ \supp(q) = \plaqset}} \prod_{p \in \plaqset} \chi(S_p q_p) \prod_{p \in \plaqset} \varphi_\beta(q_p) \\
&= \frac{N_1}{Z_{\lbox, \beta}} \Phi_S(\plaqset). 
\end{align*}
Now conclude by Lemma \ref{lemma:v-sigma-law} and Bayes' rule.
\end{proof}

\begin{lemma}\label{lemma:cond-exp-factor}
For compatible plaquette sets $\plaqset_1, \ldots, \plaqset_n \sse \plaquettes$ such that $\p(\plaqset(\Sigma) = \plaqset_1 \cup \cdots \cup \plaqset_n) > 0$, we have
\[ \E [W_\wloop(\Sigma) ~|~ \plaqset(\Sigma) = \plaqset_1 \cup \cdots \cup \plaqset_n] = \prod_{i=1}^n \E [W_\wloop (\Sigma) ~|~ \plaqset(\Sigma) = \plaqset_i]. \]
\end{lemma}
\begin{proof}
This follows by Lemmas \ref{lemma:wilson-loop-cond-exp} and \ref{lemma:activity_decomp}.
\end{proof}

Since any plaquette set $\plaqset \sse \plaquettes$ may be decomposed into a compatible collection of vortices, we now want to understand $\E [W_\wloop(\Sigma) ~|~ \plaqset(\Sigma) = \vortex]$ for vortices $\vortex$.

\begin{definition}
Given $\lbox, \wloop$, we say that a vortex $\vortex \sse \plaquettes$ {does not contribute} if either: (1) no plaquette of $\vortex$ is on the boundary of $\lbox$, and there exists a cube $B$ in $\lbox$ which contains $\vortex$ in its interior, such that no edge of $\wloop$ is in $B$, or (2) there is a plaquette of $\vortex$ on the boundary of $\lbox$, but no plaquette of $\vortex$ is in $B_\wloop$. Let $\mc{C}_{\lbox, \wloop}$ be the collection of such vortices.
\end{definition}

\begin{lemma}\label{lemma:non-contributing-vortices}
Let $\vortex \sse \plaquettes$ be a vortex which does not contribute, i.e. $\vortex \in \mc{C}_{\lbox, \wloop}$.  If $\p(\plaqset(\Sigma) = \vortex) > 0$, then
\[ \E [W_\wloop(\Sigma) ~|~ \plaqset(\Sigma) = \vortex] = 1. \]
\end{lemma}
\begin{proof}
By Lemma \ref{lemma:wilson-loop-cond-exp}, we know
\[ \E [W_\wloop(\Sigma) ~|~ \plaqset(\Sigma) = \vortex] = \frac{\Phi_S(\vortex)}{\Phi(\vortex)}.  \]
Recalling the definitions of $\Phi$ and $\Phi_S$, it suffices to show that for any $q \in G^\plaquettes$ such that $dq = 0, \supp(q) = \vortex$, we have
\[ \prod_{p \in \vortex} \chi(S_p q_p) = 1. \]
To show this, it suffices to show
\[ \sum_{p \in \vortex} S_p q_p = \sum_{p \in \plaquettes} S_p q_p = \langle q, S \rangle = 0.\]
First, assume that no plaquette of $\vortex$ is on the boundary of $\lbox$. Then I claim that there is a 1-form $\sigma$ such that $q = d\sigma$, and $\sigma_e = 0$ for all edges $e$ in $\wloop$. We would then have by Stokes' theorem $\langle q, S \rangle = \langle \sigma, \wloop \rangle = 0$.

To see the claim, first observe that by assumption, there is a cube $B$ in $\lbox$ which contains $\vortex$ in its interior, such that no edge of $\wloop$ is in $B$. Let $\tilde{q}$ be the 2-form obtained by restricting $q$ to $B$. We have that $d \tilde{q} = 0$, and also as $\supp(q) = \vortex$, we have that $\tilde{q}$ vanishes on the boundary of $B$. Thus by the Poincar\'{e} lemma, there exists a 1-form $\tilde{\sigma}$ on $B$, such that $\tilde{q} = d \tilde{\sigma}$, and $\tilde{\sigma}$ vanishes on the boundary of $B$. Now extend $\tilde{\sigma}$ to a 1-form $\sigma$ on $\lbox$, by setting $\sigma_e := 0$ for all edges $e$ not in $B$. For plaquettes $p$ in $B$, we have $(d \sigma)_p = \tilde{q}_p = q_p$. For plaquettes $p$ not in $B$, observe that the edges of $p$ must either be outside of $B$, or on the boundary of $B$, and thus $(d \sigma)_p = 0 = q_p$. Thus $q = d \sigma$, and moreover $\sigma = 0$ outside $B$, so that in particular for all edges $e$ in $\wloop$, we have $\sigma_e = 0$.

Now suppose that $\vortex$ has a plaquette on the boundary of $\lbox$. By assumption, no plaquettes of $\vortex$ are in $B_\wloop$. Thus as $S$ is zero outside of $B_\wloop$, we have $\langle q, S \rangle = 0$. 
\end{proof}

\begin{lemma}\label{lemma:vortex-on-boundary-must-be-large-to-contribute}
Let $\vortex \sse \plaquettes$ be a vortex which has a plaquette on the boundary of $\lbox$. Recall $L$ is the $\ell^\infty$ distance between the boundary of $\lbox$ and the boundary of $B_\wloop$. If $|\vortex| < L$, then no plaquettes of $\vortex$ are in $B_\wloop$, and so $\vortex$ does not contribute.
\end{lemma}
\begin{proof}
This follows because $\vortex$ is a vortex, so that the graph $G(\vortex)$ is connected.
\end{proof}

We now consider the case where $\vortex$ is a minimal vortex.


\begin{lemma}\label{lemma:minimal-vortices-contribution}
Suppose $\vortex \sse \plaquettes$ is a minimal vortex. Let $e_0$ be such that $\vortex = P(e_0)$.
If $e_0 \notin \wloop$, then
\[ \E [W_\wloop(\Sigma) ~|~ \plaqset(\Sigma) = \vortex] = 1.\]
If $e_0 \in \wloop$, then
\[ \E [W_\wloop (\Sigma) ~|~ \plaqset(\Sigma) = \vortex] = \elemmatrix. \]
\end{lemma}
\begin{proof}
We have
\[\E [W_\wloop(\Sigma) ~|~ \plaqset(\Sigma) = \vortex] = \frac{\Phi_S(\vortex)}{\Phi(\vortex)}.\]
By direct verification, any $q \in G^\plaquettes$ such that $dq = 0, \supp(q) = P(e_0)$ must be the exterior derivative of a very simple 1-form -- we may express $q = d\sigma^{g, e_0}$, where $g \in G$, $g \neq 0$, and
\[ \sigma^{g, e_0}_e := \begin{cases} g & e = e_0 \\ 0 & e \neq e_0 \end{cases}.\]
We thus have
\[ \Phi(\vortex) = \sum_{g \neq 0} \varphi_\beta(g)^6 = r_\beta,\]
where the 6 comes from the fact that $P(e_0)$ has 6 plaquettes. Also, by Stokes' theorem, we have
\[ \langle q, S \rangle = \langle \sigma^{g, e_0}, \wloop \rangle. \]
If $e_0 \notin \wloop$, then the above is 0 for any $\sigma^{g, e_0}$, and thus $\Phi_S(\vortex) = \Phi(\vortex)$. If $e_0 \in \wloop$, then $\langle \sigma^{g, e_0}, \wloop \rangle$ is $g$ if $\wloop_{e_0} = 1$ and $-g$ if otherwise $\wloop_{e_0} = -1$ (here we have used the assumption that $\wloop$ is  self avoiding). Now since $\chi$ is the character of a unitary representation of $G$, we have $\varphi_\beta(g) = \varphi_\beta(-g)$. Thus no matter the orientation of $e_0$ in $\wloop$, we have 
\[\Phi_S(\vortex) = \sum_{g \neq 0} \chi(g) \varphi_\beta(g)^6.  \]
Now finish by dividing $\Phi_S(\vortex)$ by $\Phi(\vortex) = r_\beta$.
\end{proof}

\begin{remark}\label{remark:minimal-vortex-constant-activity}
In the course of the proof, we showed that for any minimal vortex $\vortex$ contained in $\lbox$, we have $\Phi(\vortex) = r_\beta$. We will use this fact later.
\end{remark}

Now given a realization of $\plaqset(\Sigma)$, consider its vortex decomposition. Let $E$\label{notation:E-abelian} be the event that any vortex of $\plaqset(\Sigma)$ which is not a minimal vortex does not contribute. Let $N_\wloop$\label{notation:N-wloop-abelian} be the number of edges $e \in \wloop$ such that $P(e)$ is a vortex of $\plaqset(\Sigma)$.

\begin{lemma}\label{lemma:wilson-loop-on-event-E}
On the event $E$, we have
\[ \E[W_\wloop(\Sigma) ~|~ \plaqset(\Sigma)] = \elemmatrix^{N_\wloop}. \]
\end{lemma}
\begin{proof}
This follows directly from Lemmas \ref{lemma:cond-exp-factor}, \ref{lemma:non-contributing-vortices}, and \ref{lemma:minimal-vortices-contribution}.
\end{proof}


In the remainder of the subsection, assume that the condition $L \geq 50$ from Theorem \ref{thm:main-abelian} is satisfied.
The proofs of many of the results to be stated shortly will be deferred to Section \ref{section:abelian-proofs}.

\begin{lemma}\label{lemma:only-minimal-vortices-contribute-probability-bound}
Suppose
\[ \beta \geq \frac{1}{\Delta_G} (60 +  14\log (\abs{G} - 1)). \]
Then
\[ \p(E^c) \leq (\ell r_\beta) e^{-\beta \Delta_G / 2} +  (1/2) N^4 e^{-\beta L \Delta_G / 2}.\]
\end{lemma}

Upon combining Lemmas \ref{lemma:wilson-loop-on-event-E} and \ref{lemma:only-minimal-vortices-contribute-probability-bound}, we obtain the following corollary.

\begin{cor}\label{cor:only-minimal-vortices-approximation}
Suppose
\[ \beta \geq \frac{1}{\Delta_G} (60 +  14\log (\abs{G} - 1)). \]
Then
\[ |\E W_\wloop(\Sigma) - \E \elemmatrix^{N_\wloop}| \leq 2(\ell r_\beta) e^{-\beta \Delta_G / 2} + N^4 e^{-\beta L \Delta_G / 2}. \]
\end{cor}

The following proposition allows us to approximately evaluate $\E \elemmatrix^{N_\wloop}$, by showing that $N_\wloop$ is approximately Poisson.

\begin{prop}\label{prop:poisson-approximation-abelian}
Suppose
\[ \beta \geq \frac{1}{\Delta_G} (30 + 7 \log(\abs{G} - 1)). \]
Let $\ms{L}(N_\wloop)$ be the law of $N_\wloop$. Then
\[ d_{TV}(\ms{L}(N_{\wloop}), \mathrm{Poisson}(\ell r_\beta)) \leq 1290 (e^{1.5 \ell r_\beta}) r_\beta. \]
\end{prop}

Upon combining Corollary \ref{cor:only-minimal-vortices-approximation} and Proposition \ref{prop:poisson-approximation-abelian}, we obtain the following result.

\begin{lemma}\label{lemma:small-loop-abelian}
Suppose  
\[ \beta \geq \frac{1}{\Delta_G} (60 +  14\log (\abs{G} - 1)). \]
Then
\[ |\E W_\wloop(\Sigma) - e^{-\ell r_\beta(1 - \elemmatrix)}| \leq  3 e^{1.5 \ell r_\beta} e^{-\beta \Delta_G / 2} + N^4 e^{-\beta L \Delta_G / 2}. \]
\end{lemma}
\begin{proof}
There exists a coupling $(\tilde{N}_\wloop, X)$ such that $\tilde{N}_\wloop \stackrel{d}{=} N_\wloop$, $X \sim \mathrm{Poisson}(\ell r_\beta)$, and $\p(\tilde{N}_\wloop \neq X) = d_{TV}(\ms{L}(N_\wloop), \mathrm{Poisson}(\ell r_\beta))$. By Corollary \ref{cor:only-minimal-vortices-approximation} and Proposition \ref{prop:poisson-approximation-abelian}, and also the fact $|A_\beta| \leq 1$, we obtain
\[ |\E W_\wloop(\Sigma) - \E \elemmatrix^X| \leq 2(\ell r_\beta) e^{-\beta \Delta_G / 2} + N^4 e^{-\beta L \Delta_G / 2} + 2 \cdot 1290 (e^{1.5 \ell r_\beta}) r_\beta. \]
Observe
\[ \E \elemmatrix^X = e^{-\ell r_\beta (1 - \elemmatrix)}, \]
\[ r_\beta \leq (\abs{G} - 1) e^{-6 \beta \Delta_G}. \]
The assumption on $\beta$ implies
\[ 2\cdot 1290 (\abs{G} - 1) \leq e^{\beta \Delta_G}.  \]
Finish by combining the above bounds, and noting $\ell r_\beta \leq e^{1.5 \ell r_\beta}$.
\end{proof}

The following proposition and lemma allow us to handle the case of very long loops (i.e. large $\ell$). This is analogous to Lemma 7.12 of \cite{Ch2019}.

\begin{prop}\label{prop:left-tail-bound-abelian}
Suppose 
\[ \beta \geq \frac{1}{\Delta_G} (35 + 7 \log (\abs{G} - 1)).\]
Then we have the left tail bound
\[ \p\bigg(N_\wloop \leq \frac{1}{2} \ell r_\beta\bigg) \leq (2e) e^{- 0.15 \ell r_\beta}. \]
\end{prop}

\begin{lemma}\label{lemma:large-loop-abelian}
Suppose
\[ \beta \geq \frac{1}{\Delta_G} (35 + 7 \log (\abs{G} - 1)), \]
and also suppose $|\elemmatrix| < 1$. Let $c_\beta$ be as in the statement of Theorem \ref{thm:main-abelian}. Then
\[ |\E W_\wloop(\Sigma) - e^{-\ell r_\beta (1 - \elemmatrix)}| \leq (2e + 2) e^{-c_\beta \ell r_\beta}.\]
\end{lemma}
\begin{proof}
By Lemmas \ref{lemma:cond-exp-factor} and \ref{lemma:minimal-vortices-contribution}, we always have
\[ |\E[W_\wloop(\Sigma) ~|~ \plaqset(\Sigma)]| \leq \abs{\elemmatrix}^{N_\wloop},\]
and so
\[ |\E W_\wloop(\Sigma)| \leq \E |\elemmatrix|^{N_\wloop}. \]
Applying Proposition \ref{prop:left-tail-bound-abelian}, we obtain
\[ \E |\elemmatrix|^{N_\wloop} \leq (2e) e^{-0.15 \ell r_\beta} + |\elemmatrix|^{0.5 \ell r_\beta}, \]
and thus
\[ |\E W_\wloop(\Sigma) - e^{-\ell r_\beta (1 - \elemmatrix)}| \leq (2e) e^{-0.15 \ell r_\beta} + |\elemmatrix|^{0.5 \ell r_\beta} + e^{-\ell r_\beta (1 - \elemmatrix)}.\]
Now bound the right hand side by using the definition of $c_\beta$.
\end{proof}

\begin{proof}[Proof of Theorem \ref{thm:main-abelian}]
By Lemmas \ref{lemma:small-loop-abelian} and \ref{lemma:large-loop-abelian}, we obtain
\[\begin{split}
| \E W_\wloop(\Sigma) - &e^{-\ell r_\beta(1 - \elemmatrix)}|^{1 + 1.5 / c_\beta} \leq \\
&(3 e^{1.5 \ell r_\beta} e^{-\beta \Delta_G / 2} + N^4 e^{-\beta L \Delta_G / 2})(2e+2)^{1.5 / c_\beta} e^{- c_\beta \ell r_\beta (1.5 / c_\beta)}. 
\end{split}\]
Now finish by distributing the above product, and then taking both sides to the power $c_\beta / (1.5 + c_\beta)$.
\end{proof}

\subsection{Proofs}\label{section:abelian-proofs}

We start by proving some combinatorial bounds on the number of vortices satisfying certain constraints.

\begin{lemma}\label{lemma:minimal-vortex-incompatible-bound}
Let $\vortex$ be a minimal vortex. The number of minimal vortices $\vortex'$ for which $\vortex' \nsim \vortex$ is at most 144.
\end{lemma}
\begin{proof}
If $\vortex' \nsim \vortex$, then there must be a 3-cell $c$ such that both $e, e'$ belong to $c$. The number of 3-cells which contain $e$ is 12, and each 3-cell has 12 edges.
\end{proof}


\begin{lemma}\label{lemma:vortex-combinatorial-bound}
Fix a plaquette $p \in \plaquettes$. For any $m \geq 1$, the number of vortices $\vortex$ of size $m$ which contain $p$ is at most $(\vbdconst)^m$.
\end{lemma}
\begin{proof}
It suffices to assume $m \geq 2$. We may express the number in question by the following expression:
\[ \frac{1}{m!} \sum_{\substack{p_1, \ldots, p_m \in \plaquettes \\ \exists i \, : \, p_i = p}} \ind(G(\{p_1, \ldots, p_m\}) \text{ is connected}).\]
Here the sum is also over distinct plaquettes $p_1, \ldots, p_m$. Due to space constraints, we will not write this restriction explicitly. As every connected graph contains a spanning tree, we may upper bound the above by
\[ \frac{1}{m!} \sum_{k=1}^m \sum_{T \text{ on } [m]} \sum_{\substack{p_1, \ldots, p_m \in \plaquettes\\ p_k = p}} \ind(G(\{p_1, \ldots, p_m\}) \supseteq T). \]
Here the sum over $T \text{ on } [m]$ is a sum over trees $T$ with vertex set $[m]$, and the expression $G(\{p_1, \ldots, p_m\}) \supseteq T$ means that $T$ is a subgraph of $G(\{p_1, \ldots, p_m\})$. Fixing $k=1$ say, and a tree $T$, I claim that the number of sequences of distinct plaquettes $p_1, \ldots, p_m \in \plaquettes$ such that $p_1 = p$, and $G(\{p_1, \ldots, p_m\}) \supseteq T$, is at most $20^{m-1}$. Given this claim, we can obtain the further upper bound
\[ \frac{1}{m!} m 20^{m-1} \sum_{T \text{ on } [m]} 1.  \]
By Cayley's formula (see e.g. \cite{BOL1998} Chapter VIII, Theorem 20), we have that the total number of trees $T$ on $[m]$ is $m^{m-2}$. Combining this with the bound $m! \geq m^{m} e^{-m}$, we obtain the desired upper bound of $(\vbdconst)^m$.

To see the claim, fix a tree $T$, and let $i_1, \ldots, i_{d_1}$ be the neighbors of $1$ in $T$. Note $p_{i_1}$ must share a 3-cell with $p_1 = p$, and thus the number of possible choices of $p_{i_1}$ is at most 20, since there are 4 3-cells which contain $p$ and each 3-cell has 5 other plaquettes besides $p$. Similarly, the number of possible choices for each $p_{i_j}$, $2 \leq j \leq d_1$, is at most 20. Continuing in this manner, we may traverse $T$ in breadth-first manner until we have chosen $p_i$ for all $1 \leq i \leq n$. 
\end{proof}

\begin{lemma}\label{lemma:neighbor-vortex-combinatorial-bound}
Let $\vortex$ be a vortex. For any $m \geq 1$, the number of vortices $\vortex'$ of size $m$ which are incompatible with $\vortex$ is at most $21 \abs{\vortex} (\vbdconst)^m$.
\end{lemma}
\begin{proof}
In order for a vortex $\vortex'$ to be incompatible with $\vortex$, it must contain a plaquette $p$ which shares a 3-cell with a plaquette of $\vortex$. The number of such plaquettes is at most $21 \abs{\vortex}$, since each plaquette $p'$ of $\vortex$ is contained in 4 3-cells, and every such 3-cell has 5 other plaquettes besides $p'$. Now finish by Lemma \ref{lemma:vortex-combinatorial-bound}.
\end{proof}

\begin{lemma}\label{lemma:vortex-contained-in-cube}
Any vortex of size $m$ is contained in a cube of side length $m$.
\end{lemma}
\begin{proof}
We proceed by induction. If $m = 1$, then the vortex is a single plaquette, and the claim is true. Now suppose the claim is true for some $m$. Given a vortex $\vortex$ of size $m+1$, there exists a plaquette $p \in \vortex$ such that $V \backslash \{p\}$ is still a vortex. This follows by the following graph theory fact: any connected graph has a vertex whose removal leaves the graph connected. This fact may be seen by e.g. removing a leaf of a spanning tree. Now by assumption, there is a cube $B$ of side length $m$ such that $\vortex \backslash \{p\}$ is contained in $B$. There exists $p' \in \vortex$ such that there is a 3-cell $c$ which contains both $p, p'$. If $c$ is in $B$, then $\vortex$ is contained in $B$, and we are done. Otherwise, since $c$ contains a plaquette of $B$, we may extend $B$ in some direction to ensure that $c$ is in this extended cube, of side length $m+1$.
\end{proof}

\begin{lemma}\label{lemma:contributing-plaquette-bound}
For $m \geq 1$, let $S_\wloop^m$ be the set of plaquettes $p \in \plaquettes$ such that any cube of side length $m+2$ containing $p$ does not contain an edge of $\wloop$. Then
\[ |\plaquettes \backslash S_\wloop^m| \leq  48 \cdot 7^4 \ell m^4.\]
\end{lemma}
\begin{proof}
Suppose $p \in \plaquettes \backslash S_\wloop^m$. Then there exists a cube $B$ of side length $m+2$, such that $p$ is in $B$, and there exists an edge $e \in \wloop$ such that $e$ is in $B$. Take a vertex $x$ of $e$. Then $p$ is contained in the $\ell^\infty$ ball of radius $m+2$ centered at $x$. This $\ell^\infty$ ball, when intersected with $\Z^4$, is a cube of side length $2m + 4$. The number of plaquettes in such a cube is at most $6 \cdot8 (2m + 5)^4 = 48 (2m + 5)^4$, since the number of vertices is $(2m+5)^4$, and each vertex is incident to at most $8$ edges, and each edge is incident to at most $6$ plaquettes. As $m \geq 1$, we have $2m +5 \leq 7m$. To finish, observe that the number of vertices in $\wloop$ is at most the number of edges $\ell$.
\end{proof}

The following lemma shows that if a plaquette set $\plaqset$ is not near the boundary of $\lbox$, then in order for $\plaqset$ to appear in the random collection of surfaces $\plaqset(\Sigma)$ with positive probability, it must either be a minimal vortex, or have size strictly greater than 6 (which is the size of a minimal vortex, and hence the name ``minimal vortex").

\begin{lemma}\label{lemma:smaller-than-minimal-vortices-prob-0}
Suppose the plaquette set $\plaqset \sse \plaquettes$ is such that any 3-cell which contains a plaquette of $\plaqset$ is contained in $\lbox$. If $|\plaqset| \leq 5$, then $\Phi(\plaqset) = 0$. If $|\plaqset| = 6$, then $\Phi(\plaqset) \neq 0$ if and only if $\plaqset = P(e)$, for some $e \in \edges$.
\end{lemma}
\begin{proof}
I claim that if $|\plaqset| \leq 5$, then there cannot exist $q \in G^\plaquettes$, such that $dq = 0$ and $\supp(q) = \plaqset$. Consequently, $\Phi(\plaqset)$ is an empty sum, and thus is zero. The claim follows because if $|\plaqset| \leq 5$, then there must exist a 3-cell $c$ in $\Z^4$ which contains exactly one element of $\plaqset$, call it $p$. By assumption, $c$ is contained in $\lbox$. Note $(dq)_c = \pm q_p$. But if $\supp(q) = \plaqset$, then $q_p \neq 0$, and thus $dq \neq 0$.

Similarly, if $|\plaqset| = 6$, and if $\plaqset$ is not a minimal vortex, then there must exist a 3-cell $c$ which contains exactly one element of $\plaqset$. Then proceed as before.
\end{proof}

Recall now equation \eqref{eq:delta-g-def-abelian}, which gives the formula for $\Delta_G$ in the Abelian case. Define
\[ \alpha_\beta \label{notation:alpha-beta-abelian} := (\abs{G} - 1)e^{-\beta \Delta_G}. \]

\begin{lemma}\label{lemma:activity-bound}
For any plaquette set $\plaqset \sse \plaquettes$, we have
\[ \Phi(\plaqset) \leq \alpha_\beta^{\abs{\plaqset}}. \]
\end{lemma}
\begin{proof}
For any $q \in G^\plaquettes$ such that $dq = 0$, $\supp(q) = \plaqset$, we have
\[ \prod_{p \in \plaqset} \varphi_\beta(q_p) \leq e^{-\abs{\plaqset} \beta \Delta_G}. \]
To finish, observe that since $q_p \neq 0$ if and only if $p \in \plaqset$, we have at most $(\abs{G} - 1)^{\abs{\plaqset}}$ possible choices of $q$. 
\end{proof}

\begin{remark}
The bound in Lemma \ref{lemma:activity-bound} may be fairly loose, since when bounding the number of possible $q$ such that $dq = 0, \supp(q) = \plaqset$, we did not even use the restriction $dq = 0$.
\end{remark}

Let $\mc{P}_\lbox$\label{notation:P-lbox} be the set of nonempty vortices $\vortex \sse \plaquettes$.
We will eventually apply a result of \cite{AL2005}, so to make our notation more in line with the notation of \cite{AL2005}, define the function $z^\beta : \mc{P}_\lbox \ra \mc{C}$ as follows:
\[ z^\beta(\vortex) := \Phi(\vortex). \]

To motivate the next definition, first observe that starting from Lemma \ref{lemma:abelian-cluster-expansion-setup}, we may write 
\[ Z_{\lbox, \beta}
= N_1\Bigg(1 +  \sum_{n \geq 1} \sum_{\{\vortex_1, \ldots, \vortex_n\} \sse \mc{P}_\lbox} \prod_{i=1}^n z^\beta(\vortex_i) \prod_{1 \leq i < j \leq n} \ind(\vortex_i \sim \vortex_j) \Bigg). \]
Here we have used Lemma \ref{lemma:activity_decomp}, and the fact that any nonempty plaquette set $\plaqset \sse \plaquettes$ may be decomposed uniquely into a collection of compatible vortices $\vortex_1, \ldots, \vortex_n$, and moreover, by definition, vortices $\vortex_1, \ldots, \vortex_n$ are compatible if and only if they are pairwise compatible.  

Now for a function $z: \mc{P}_\lbox \ra \C$, and $\mc{P} \sse \mc{P}_\lbox$, define
\beq\label{eq:partition-function-gas-of-particles-abelian} \Xi_\mc{P}(z) := 1 + \sum_{n \geq 1} \sum_{\{\vortex_1, \ldots, \vortex_n\} \sse \mc{P}} z(\vortex_1) \cdots z(\vortex_n) \prod_{1 \leq i < j \leq n} \ind(\vortex_i \sim \vortex_j). \eeq \label{notation:Xi-P-abelian}
Observe
\[ Z_{\lbox, \beta} = N_1 \Xi_{\mc{P}_\lbox}(z^\beta). \]
Furthermore, as we will soon see, we may express the probability distribution of $\plaqset(\Sigma)$ in terms of $\Xi_{\mc{P}}(z)$.

Given a vortex $\vortex \in \mc{P}_\lbox$, define the set $N_\lbox(\vortex) := \{\vortex' \in \mc{P}_\lbox : \vortex' \nsim \vortex\}$. Now for a collection of vortices $\vortex_1, \ldots, \vortex_{n_0}$, let
\[\label{notation:N-lbox} N_\lbox(\vortex_1, \ldots, \vortex_{n_0}) := \bigcup_{i=1}^{n_0} N_\lbox(\vortex_i). \]

\begin{lemma}\label{lemma:prob-vortices-appear}
For compatible vortices $\vortex_1, \ldots, \vortex_{n_0}$, we have
\[\p(\vortex_1, \ldots, \vortex_{n_0} \text{ are vortices of } \plaqset(\Sigma)) = \Bigg(\prod_{i=1}^{n_0} z^\beta(\vortex_i) \Bigg)\frac{\Xi_{\mc{P}_\lbox \backslash N_\lbox(\vortex_1, \ldots, \vortex_n)}(z^\beta)}{\Xi_{\mc{P}_\lbox}(z^\beta)}. \]
\end{lemma}
\begin{proof}
Observe that the event that $\vortex_1, \ldots, \vortex_{n_0}$ are vortices of $\plaqset(\Sigma)$ may be written
\[  \bigcup_{n \geq 0} \bigcup_{\{\vortex_1', \ldots, \vortex_n'\} \sse \mc{P}_\lbox \backslash N_\lbox(\vortex_1, \ldots, \vortex_{n_0})} \{\plaqset(\Sigma) = \vortex_1 \cup \cdots \cup \vortex_{n_0} \cup \vortex_1' \cup \cdots \cup \vortex_{n}'\} .\]
In the inner union over vortices $V_1', \ldots, V_{n}'$, we have the additional restriction that the vortices must be compatible, which due to space restrictions we omit. Now finish by observing that the above union is a disjoint union, applying Lemmas \ref{lemma:v-sigma-law} and \ref{lemma:activity_decomp}, and using the definition of $\Xi_{\mc{P}}(z)$.
\end{proof}

\begin{remark}
It was mentioned in the introduction that the random surface $\plaqset(\Sigma)$ has ``a lot of independence". We can now describe in what sense this is true. We will soon show that for certain subsets $\Gamma \sse \mc{P}_\lbox$, the ratio $\Xi_{\mc{P}_\lbox \backslash \Gamma} / \Xi_{\mc{P}_\lbox}$ is essentially 1 when $\beta$ is large. Combining this with Lemma \ref{lemma:prob-vortices-appear}, we then see that for fixed compatible vortices $\vortex_1, \ldots, \vortex_{n_0}$, the events that $\vortex_i$ appears as a vortex in $\plaqset(\Sigma)$, $1 \leq i \leq n_0$, are essentially independent. In other words, vortices which are not very close to each other appear approximately independently.
\end{remark}

Observe as $z^\beta(\vortex) \geq 0$ for all $\vortex \in \mc{P}_\lbox$, we have
\[\frac{\Xi_{\mc{P}_\lbox \backslash N_\lbox(\vortex_1, \ldots, \vortex_n)}(z^\beta)}{\Xi_{\mc{P}_\lbox}(z^\beta)} \leq 1, \]
and thus we have the following corollary of Lemma \ref{lemma:prob-vortices-appear}.

\begin{cor}\label{cor:vortex-probability-bound}
For compatible vortices $\vortex_1, \ldots, \vortex_{n_0}$, we have
\[ \p(\vortex_1, \ldots, \vortex_{n_0} \text{ are vortices of } \plaqset(\Sigma)) \leq \prod_{i=1}^{n_0} z^\beta(\vortex_i).\]
\end{cor}

We now have enough to prove Lemma \ref{lemma:only-minimal-vortices-contribute-probability-bound}. The proof is essentially a Peierls argument.

\begin{proof}[Proof of Lemma \ref{lemma:only-minimal-vortices-contribute-probability-bound}]
By Lemma \ref{lemma:vortex-contained-in-cube}, any vortex of size $m$ which is not on the boundary of $\lbox$ is contained in the interior of a cube of side length  at most $m+2$. Combining this with Lemma \ref{lemma:vortex-on-boundary-must-be-large-to-contribute}, we have that on the event $E^c$, there must exist some vortex $\vortex$ of $\plaqset(\Sigma)$ which is not a minimal vortex, such that either (1) $\vortex$ is not on the boundary of $\lbox$, and $\vortex$ contains a plaquette of $\plaquettes \backslash S_\wloop^{|V|}$, or (2) there is a plaquette of $\vortex$ on the boundary of $\lbox$, and $|\vortex| \geq L$. By Lemmas \ref{lemma:vortex-combinatorial-bound} and \ref{lemma:contributing-plaquette-bound}, the number of vortices of size $m$ of the first type is at most $48 \cdot 7^4 \ell m^4 (\vbdconst)^m$. Recall $N$ is the side length of $\lbox$, so that the boundary of $\lbox$ has at most $48 \cdot (N+1)^4$ plaquettes. Using this and Lemma \ref{lemma:vortex-combinatorial-bound}, the number of vortices of size $m' \geq L$ of the second type is at most $48 (N+1)^4 (\vbdconst)^{m'}$.

Now by Corollary \ref{cor:vortex-probability-bound} and Lemma \ref{lemma:activity-bound}, we have that for any vortex $\vortex$ of size $m$,
\[ \p(\vortex \text{ is a vortex of $\plaqset(\Sigma)$}) \leq z^\beta(\vortex) \leq \alpha_\beta^m. \]
Moreover, observe that if $\vortex$ contributes, and $\vortex$ has size at most 6, then by the assumption that $L \geq 50$, we have that the vortex $\vortex$ is far enough away from the boundary of $\lbox$, such that the condition of Lemma \ref{lemma:smaller-than-minimal-vortices-prob-0} is satisfied. Thus if we further assume that $\vortex$ is not a minimal vortex, then we must have $z^\beta(\vortex) = 0$. Thus we may consider only vortices with size at least 7. Now by a union bound, we have
\begin{align*}
\p(E^c) &\leq 48 \cdot 7^4 \ell \sum_{m=7}^\infty m^4 (\vbdconst)^m \alpha_\beta^m + 48 (N+1)^4 \sum_{m' = L}^\infty (\vbdconst)^{m'} \alpha_\beta^{m'} \\
&\leq 48 \cdot 7^4 \ell \sum_{m=7}^\infty  (80e\alpha_\beta)^m  + 48 \cdot 2^4 N^4 \frac{(\vbdconst\alpha_\beta)^{L}}{1 - \vbdconst \alpha_\beta}\\
&= 48 \cdot 7^4 \cdot (80e)^7 \ell (\abs{G} - 1)^7 \frac{e^{-7 \beta \Delta_G}}{1 - 80e \alpha_\beta} + 48 \cdot 2^4 N^4 \frac{(\vbdconst\alpha_\beta)^{L}}{1 - \vbdconst\alpha_\beta}.
\end{align*}
Note $e^{-6\beta \Delta_G} \leq r_\beta$, and the assumption on $\beta$ in the lemma statement implies
\[ 80e \alpha_\beta \leq 1/2,\]
\[2 \cdot 48 \cdot 7^4 \cdot (80e)^7 (\abs{G} - 1)^7 \leq e^{\beta \Delta_G / 2},  \]
\[ 2 \cdot 48 \cdot 2^4 \cdot (\vbdconst\alpha_\beta)^L \leq (1/2) e^{-\beta L\Delta_G / 2}. \qedhere\]
\end{proof}

We will need a more general version of Lemma \ref{lemma:prob-vortices-appear}, to be described next. Given a vortex $\vortex$, let $F_\vortex$ be the event that $\vortex$ is a vortex of $\plaqset(\Sigma)$. Given vortices $\vortex_1, \ldots, \vortex_{n_0}$, $\vortex_1', \ldots, \vortex_{k_0}'$, observe that
\[ \bigcap_{i=1}^{n_0} F_{\vortex_i} \cap \bigcap_{j=1}^{k_0} F_{\vortex_j'}^c \]
is the event every vortex $\vortex_i$, $1 \leq i \leq n_0$ is a vortex of $\plaqset(\Sigma)$, and no vortex $\vortex_j'$, $1 \leq j \leq k_0$ is a vortex of $\plaqset(\Sigma)$. For ease of notation, for $\Gamma, \mc{P} \sse \mc{P}_\lbox$, and $z : \mc{P}_\lbox \ra \C$, define the ratio of partition functions (often called the reduced correlations) \label{notation:rho-P-abelian}
\[  \rho_{\mc{P}}(z, \Gamma) := \frac{\Xi_{\mc{P} \backslash \Gamma}(z)}{\Xi_{\mc{P}}(z)}.\]
By essentially the same proof as Lemma \ref{lemma:prob-vortices-appear}, we may obtain the following lemma.

\begin{lemma}\label{lemma:vortex-appear-not-appear-probability}
For compatible vortices $\vortex_1, \ldots, \vortex_{n_0}$, and another collection of (possibly non-compatible) vortices $\vortex_1', \ldots, \vortex_{k_0}'$, we have
\[\begin{split}
\p\bigg(\bigcap_{i=1}^{n_0} F_{\vortex_i} \cap \bigcap_{j=1}^{k_0} F_{\vortex_j'}^c \bigg) &=
\Bigg(\prod_{i=1}^{n_0} z^\beta(\vortex_i)\Bigg) \times \\
&\rho_{\mc{P}_\lbox}(z^\beta, N_\lbox(\vortex_1, \ldots, \vortex_{n_0}) \cup \{\vortex_1', \ldots, \vortex_{k_0}'\}) .
\end{split}\]
\end{lemma}

For $z : \mc{P}_\lbox \ra \C$, $\vortex \in \mc{P} \sse \mc{P}_\lbox$, define
\[ \Theta^{\mc{P}}_\vortex(z) := -\log \rho_{\mc{P}}(z, \{\vortex\}) = \log \Xi_{\mc{P}}(z) - \log \Xi_{\mc{P} \backslash \{\vortex\}}(z).  \]
We now quote the following cluster expansion result from \cite{BFP2010}, adapted to our situation.

\begin{theorem}[Theorem 2.1 of \cite{BFP2010}]\label{thm:dobrushins-cluster-expansion}
Suppose $\nu : \mc{P}_\lbox \ra \R$, and moreover $\nu$ is nonnegative, i.e. $\nu(\vortex) \geq 0$ for all $\vortex \in \mc{P}_\lbox$. Suppose that for all $\vortex \in \mc{P}_\lbox$, we have
\[ z^\beta(\vortex) \prod_{\vortex' \in N_\lbox(\vortex)} (1 + \nu(\vortex')) \leq \nu(\vortex). \]
Then for any $\vortex\in \mc{P} \sse \mc{P}_\lbox$, we have
\beq\label{eq:dobrushin-conclusion-1} |\Theta^{\mc{P}}_\vortex(z^\beta)| \leq \log (1 + \nu(\vortex)),\eeq
\beq\label{eq:dobrushin-conclusion-2} |\Theta^{\mc{P}}_\vortex(z^\beta) - z^\beta(\vortex)| \leq \log(1 + \nu(\vortex)) - z^\beta(\vortex). \eeq
\end{theorem}

\begin{remark}
Technically, Theorem 2.1 of \cite{BFP2010} only gives \eqref{eq:dobrushin-conclusion-1}. To obtain \eqref{eq:dobrushin-conclusion-2}, combine \eqref{eq:dobrushin-conclusion-1} with equations (2.9) and (2.13) of \cite{BFP2010}.
\end{remark}


Theorem \ref{thm:dobrushins-cluster-expansion} allows us to obtain the following bounds on $\Theta^{\mc{P}}_\vortex$ at large $\beta$.

\begin{prop}\label{prop:cluster-expansion-convergence}
For $b > 0$, and $\beta$ large such that
\[ \betaexprone{b} < 1,\]
\[ \betaexprtwo{b} \leq b, \]
we have that for any $\vortex \in \mc{P} \sse \mc{P}_\lbox$,
\[  |\Theta^{\mc{P}}_\vortex(z^\beta)| \leq \log (1 + e^{b \abs{\vortex}} z^\beta(\vortex)) , \]
and
\[ |\Theta^{\mc{P}}_\vortex(z^\beta) - z^\beta(\vortex)| \leq \log (1 + e^{b \abs{\vortex}} z^\beta(\vortex)) - z^\beta(\vortex) \leq (e^{b \abs{\vortex}} - 1) z^\beta(\vortex).\]
\end{prop}
\begin{proof}
To apply Theorem \ref{thm:dobrushins-cluster-expansion}, define $\nu^\beta(\vortex) := z^\beta(\vortex) e^{b \abs{\vortex}}$. To verify the condition of the theorem, it suffices to show that for all $\vortex \in \mc{P}_\lbox$, we have
\[ z^\beta(\vortex) \exp(\sum_{\vortex' \in N_\lbox(\vortex)} z^\beta(\vortex') e^{b \abs{\vortex'}}) \leq z^\beta(\vortex) e^{b \abs{\vortex}},  \]
or recalling Lemma \ref{lemma:activity-bound},
\[ \sum_{\vortex' \in N_\lbox(\vortex)} (e^b \alpha_\beta)^{\abs{\vortex'}} \leq b \abs{\vortex}.\]
Applying Lemma \ref{lemma:neighbor-vortex-combinatorial-bound}, and using the assumption $\betaexprone{b} < 1$ to sum the resulting geometric series, we obtain
\begin{align*}
\sum_{\vortex' \in N_\lbox(\vortex)} (e^b \alpha_\beta)^{\abs{\vortex'}} &\leq \sum_{m=1}^\infty \sum_{\substack{\vortex' \in N_\lbox(\vortex) \\ \abs{\vortex'} = m}}  (e^b \alpha_\beta)^m \\
&\leq \abs{\vortex} \betaexprtwo{b}\\
&\leq b \abs{\vortex}. \qedhere
\end{align*}
\end{proof}

\begin{lemma}\label{lemma:partition-function-ratio-minus-neighborhood-minimal-vortex}
For $0 < \varep \leq 1$, take $b_\varep := \frac{1}{7} \log(1 + \varep)$, so that
$e^{7b_\varep} - 1 = \varep$.
Suppose $\beta$ is large enough so that
\[ \betaexprone{b_\varep} \leq \frac{1}{2},\]
\[ \betaexprtwo{b_\varep} \leq b_\varep, \]
\[ (2 \cdot 126 \cdot (\vbdconst)^7 \cdot 2) \alpha_\beta^7 \leq \varep r_\beta. \]
Let $\mc{P} \sse \mc{P}_\lbox$, and let $\vortex$ be a minimal vortex. Suppose $\vortex$ is far enough away from the boundary of $\lbox$ such that any cube of side length 25 (say) which contains a plaquette of $\vortex$ is completely contained in $\lbox$. Then for any $\Gamma \sse N_\lbox(\vortex) \cap \mc{P}$, we have
\[ |\rho_{\mc{P}}(z^\beta, \Gamma) -1| \leq 290 r_\beta .\]
Moreover, let $n_0$ be the number of minimal vortices in $\Gamma$. Then
\[ |\log \rho_{\mc{P}}(z^\beta, \Gamma) + n_0 r_\beta| \leq (n_0 + 2) \varep r_\beta. \]
\end{lemma}
\begin{proof}
The conditions on $\beta$ imply that we may apply Proposition \ref{prop:cluster-expansion-convergence} with $b = b_\varep$.
Write $\Gamma = \{\vortex_1, \ldots, \vortex_n\}$. For $1 \leq k \leq n$, let $\Gamma_k := \{\vortex_1, \ldots, \vortex_k\}$, and let $\Gamma_0 := \varnothing$. We may express
\[\rho_{\mc{P}}(z^\beta, \Gamma) = \exp(-\sum_{k=1}^n \Theta^{\Gamma_{k-1}}_{\vortex_k}(z^\beta)). \]
By the assumption that $\vortex$ is far away from the boundary of $\lbox$, for any $1 \leq k \leq n$ such that $|\vortex_k| \leq 6$, we have that any 3-cell $c$ in $\Z^4$ which contains a plaquette of $\vortex_i$ is contained in $\lbox$. Thus the condition of Lemma \ref{lemma:smaller-than-minimal-vortices-prob-0} is satisfied, and thus if $|\vortex_k| \leq 6$, then $z^\beta(\vortex_k) = \Phi(\vortex_k) \neq 0$ if and only if $\vortex_k$ is a minimal vortex. By Proposition \ref{prop:cluster-expansion-convergence}, we have that if $z^\beta(\vortex_k) = 0$, then $\Theta^{\Gamma_{k-1}}_{\vortex_k}(z^\beta) = 0$. Therefore we may as well assume that $|\vortex_k| \geq 6$ for all $1 \leq k \leq n$, and further if $|\vortex_k| = 6$, then $\vortex_k$ is a minimal vortex.

Now if $\vortex_k$ is a minimal vortex, by Proposition \ref{prop:cluster-expansion-convergence} and Remark \ref{remark:minimal-vortex-constant-activity}, we have
\[ |\Theta^{\Gamma_{k-1}}_{\vortex_k}(z^\beta) - r_\beta| \leq (e^{6b_\varep} -1) z^\beta(\vortex_k) = (e^{6b_\varep} - 1) r_\beta. \]
If $\vortex_k$ is not a minimal vortex, then by Proposition \ref{prop:cluster-expansion-convergence}, we have
\[ |\Theta^{\Gamma_{k-1}}_{\vortex_k}(z^\beta)| \leq \log(1 + e^{b_\varep |\vortex_k|} z^\beta(\vortex_k)) \leq e^{b_\varep |\vortex_k|} z^\beta(\vortex_k). \]
We then obtain
\[ \bigg|\sum_{k=1}^n \Theta^{\Gamma_{k-1}}_{\vortex_k}(z^\beta) - n_0 r_\beta \bigg| \leq n_0(e^{6b_\varep} - 1) r_\beta + R, \]
where
\[ R := \sum_{m=7}^\infty \sum_{\substack{\vortex' \in N_\lbox(\vortex) \\ |\vortex'| = m}} e^{b_\varep m} z^\beta(\vortex'). \]
Applying Lemmas \ref{lemma:vortex-combinatorial-bound} and \ref{lemma:activity-bound}, we obtain
\begin{align*}
R &\leq 21 \abs{\vortex} \sum_{m=7}^\infty (\vbdconst e^{b_\varep} \alpha_\beta)^m \\
&= 126 \frac{(\vbdconst)^7 e^{7b_\varep} \alpha_\beta^7}{1 - \betaexprone{b_\varep}}.
\end{align*}
By the assumptions, we have $e^{6b_\varep} - 1 \leq e^{7b_\varep} - 1 = \varep$, and $R \leq (1+\varep)\varep r_\beta$. Using $\varep \leq 1$ gives the second assertion. The first assertion follows from the second assertion by using the bound $n_0 \leq 144$ from Lemma \ref{lemma:minimal-vortex-incompatible-bound}, and by observing
\[\frac{\Xi_{\mc{P} \backslash \Gamma}(z^\beta)}{\Xi_{\mc{P}}(z^\beta)} \leq 1,  \] 
so that
\[ 0 \leq 1 - \exp(\log \rho_{\mc{P}}(z^\beta, \Gamma)) \leq - \log \rho_{\mc{P}}(z^\beta, \Gamma). \qedhere\]
\end{proof}

In what follows, we will abuse notation and think of $\wloop$ as a set of (unoriented) edges. For each edge $e \in \wloop$, define the set of edges
\[ B_e := \{e' \in \wloop : P(e) \nsim P(e')\}. \]
By Lemma \ref{lemma:minimal-vortex-incompatible-bound}, and the assumption that $\wloop$ is self avoiding, we have $\abs{B_e} \leq 144$. Recall that $F_\vortex$ is the event that $\vortex$ is a vortex of $\plaqset(\Sigma)$, and observe
\[ N_\wloop = \sum_{e \in \wloop} \ind_{F_{P(e)}}. \]
We quote the following Poisson approximation theorem, adapted to our situation.

\begin{theorem}[Theorem 4.1 of \cite{CR2013}]\label{thm:chen-rollin-poisson-approximation}
Let
\[ b_1 := \sum_{e \in \wloop} \sum_{e' \in B_e} \p(F_{P(e)}) \p(F_{P(e')}), \]
\[ b_2 := \sum_{e \in \wloop} \sum_{e' \in B_e \backslash \{e\}} \E [\ind_{F_{P(e)}} \ind_{F_{P(e')}}], \]
\[ b_3 := \sum_{e \in \wloop} \E \bigg[\big|\E[ \ind_{F_{P(e)}} ~|~ \ind_{F_{P(e')}}, e' \notin B_e] - \p(F_{P(e)})\big|\bigg]. \]
Let $\ms{L}(N_\wloop)$ denote the law of $N_\wloop$, and let $\lambda := \E N_\wloop$. Then
\[ d_{TV}(\ms{L}(N_\wloop), \mathrm{Poisson}(\lambda)) \leq \min (1, \lambda^{-1}) (b_1 + b_2) + \min(1, 1.4 \lambda^{-1/2})b_3.\]
\end{theorem}


\begin{lemma}\label{lemma:chen-rollin-appplication}
Let $\lambda := \E N_\wloop$. In the setting of Proposition \ref{prop:poisson-approximation-abelian}, we have
\[ d_{TV}(\ms{L}(N_\wloop), \mathrm{Poisson}(\lambda)) \leq 144 r_\beta + 1000 \min(\lambda, 1.4 \sqrt{\lambda}) r_\beta. \]
\end{lemma}
\begin{proof}
We will use Theorem \ref{thm:chen-rollin-poisson-approximation}. Let $b_1, b_2, b_3$ be as in the theorem. By Corollary \ref{cor:vortex-probability-bound} and Remark \ref{remark:minimal-vortex-constant-activity}, we have $\p(F_\vortex) \leq r_\beta$ for any minimal vortex $\vortex$. Combining this with the fact $|B_e| \leq 144$, we have
\[ b_1 \leq 144 \lambda r_\beta. \] 
We have that $b_2 = 0$, since by definition, $\plaqset(\Sigma)$ is decomposed into compatible vortices, so that no two incompatible vortices can both appear. We will now show
\[ b_3 \leq 1000 \lambda r_\beta, \]
which by Theorem \ref{thm:chen-rollin-poisson-approximation} will then give desired result.

Towards this end, first observe that the assumption on $\beta$ from Proposition \ref{prop:poisson-approximation-abelian} implies that the conditions on $\beta$ in Lemma \ref{lemma:partition-function-ratio-minus-neighborhood-minimal-vortex} are satisfied with $\varep = 1$. Also, the assumption that $L \geq 50$ implies that for any $e \in \wloop$, we have $P(e) \sse \plaquettes$, and moreover, we may apply Lemma \ref{lemma:partition-function-ratio-minus-neighborhood-minimal-vortex} with $\vortex = P(e)$. 

Now observe that $b_3$ is a sum of expectations. Our strategy will be to bound each expectation by an $L^\infty$ bound. I.e., fix any $e_0 \in \wloop$, and any $\wloop' \sse \wloop \backslash B_{e_0}$. Define $F_{e_0, \wloop'}$ to be the event that for every edge $e \in \wloop'$, the associated minimal vortex $P(e)$ is a vortex of $\plaqset(\Sigma)$, and for every edge $e' \in \wloop \backslash B_{e_0}$, $e' \notin \wloop'$, the associated minimal vortex $P(e')$ is not a vortex of $\plaqset(\Sigma)$. Note we may express $F_{e_0, \wloop'}$ as
\beq\label{eq:f-e0-wloop} F_{e_0, \wloop'} = \bigcap_{e \in \wloop'} F_{P(e)} \cap \bigcap_{\substack{e' \in \wloop \backslash B_{e_0} \\ e' \notin \wloop'}} F_{P(e')}^c. \eeq
It suffices to show that if $\p(F_{e_0, \wloop'}) > 0$, then
\[ |\E [\ind_{F_{P(e_0)}} ~|~ F_{e_0,\wloop'}] - \p(F_{P(e_0)})| \leq  \p(F_{P(e_0)}) 1000 r_\beta.\]
Note we may assume that the collection of vortices $P(e), e \in \wloop'$ is compatible, otherwise $\p(F_{e_0, \wloop'}) = 0$. Now define the collection of vortices
\[ \Gamma := N_\lbox(P(e), e \in \wloop') \cup \{P(e') : e' \in \wloop \backslash B_{e_0}, e' \notin \wloop' \}.\]
Applying Lemma \ref{lemma:vortex-appear-not-appear-probability}, we obtain
\[ \p(F_{P(e_0)}) = z^\beta(P(e_0)) \rho_{\mc{P}_\lbox}(z^\beta, N_\lbox(P(e_0)), \]
\[ \p(F_{e_0, \wloop'}) = \bigg(\prod_{e \in \wloop'} z^\beta(P(e))\bigg) \rho_{\mc{P}_\lbox}(z^\beta, \Gamma) , \]
\[ \p(F_{P(e_0)} \cap F_{e_0, \wloop'}) = z^\beta(P(e_0)) \bigg(\prod_{e \in \wloop'} z^\beta(P(e))\bigg) \rho_{\mc{P}_\lbox}(z^\beta, \Gamma \cup N_\lbox(P(e_0))). \]
Here we have used the definition of $B_{e_0}$, to ensure that the collection of vortices $P(e_0), P(e), e \in \wloop'$ is compatible.
We thus have
\[
\begin{split}
\E [\ind_{F_{P(e_0)}} ~|~ F_{e_0,\wloop'}] &- \p(F_{P(e_0)}) = \\
& \p(F_{P(e_0)})\bigg(\frac{\rho_{\mc{P}_\lbox \backslash \Gamma}(z^\beta, N_\lbox(P(e_0)))}{\rho_{\mc{P}_\lbox}(z^\beta, N_\lbox(P(e_0)))} - 1\bigg) .
\end{split}\]
For ease of notation, write the above as
\[ \p(F_{P(e_0)}) (R_1 / R_2 - 1). \]
As previously observed, we may apply Lemma \ref{lemma:partition-function-ratio-minus-neighborhood-minimal-vortex} to obtain
\[\abs{R_1 - 1} \leq 290 r_\beta,  \]
\[ \abs{R_2 - 1} \leq 290 r_\beta.\]
Now the assumption on $\beta$ in Proposition \ref{prop:poisson-approximation-abelian} implies that $\beta$ is large enough so that
\[\abs{R_1 / R_2 - 1} \leq  \frac{580 r_\beta}{1 - 290 r_\beta} \leq 1000 r_\beta, \]
and the claim now follows.
\end{proof}

\begin{lemma}\label{lemma:lambda-ell-beta-bound}
In the setting of Proposition \ref{prop:poisson-approximation-abelian}, we have
\[ \abs{\lambda - \ell r_\beta} \leq 290 (\ell r_\beta) r_\beta. \]
\end{lemma}
\begin{proof}
For each $e \in \wloop$, we have by Lemma \ref{lemma:prob-vortices-appear}
\[ \p(F_{P(e)}) = r_\beta \rho_{\mc{P}_\lbox}(z^\beta, N_\lbox(P(e))). \]
Recalling that $\lambda = \E N_\wloop = \sum_{e \in \wloop} \p(F_{P(e)})$, we obtain
\[\abs{\lambda - \ell r_\beta} \leq \sum_{e \in \wloop} |\p(F_{P(e)}) - r_\beta| = r_\beta \sum_{e \in \wloop}|\rho_{\mc{P}_\lbox}(z^\beta, N_\lbox(P(e))) - 1|.\]
As noted in the proof of Lemma \ref{lemma:chen-rollin-appplication}, 
we may apply Lemma \ref{lemma:partition-function-ratio-minus-neighborhood-minimal-vortex} with $\varep = 1$. Doing so, we may further upper bound the right hand side above by $290 (\ell r_\beta) r_\beta$, as desired.
\end{proof}

\begin{proof}[Proof of Proposition \ref{prop:poisson-approximation-abelian}]
By Lemma \ref{lemma:chen-rollin-appplication}, we have
\[ d_{TV}(\ms{L}(N_\wloop), \mathrm{Poisson}(\lambda)) \leq 144 r_\beta + 1000 \min(\lambda, 1.4 \sqrt{\lambda}) r_\beta. \]
By Corollary 3.1 of \cite{AL2005}, and Lemma \ref{lemma:lambda-ell-beta-bound}, we have
\[ d_{TV}(\mathrm{Poisson}(\lambda), \mathrm{Poisson}(\ell r_\beta)) \leq \abs{\lambda - \ell r_\beta} \leq 290 (\ell r_\beta) r_\beta. \]
We thus have
\[ d_{TV}(\ms{L}(N_\wloop), \mathrm{Poisson}(\ell r_\beta)) \leq 144 r_\beta + 1000 \min (\lambda, 1.4 \sqrt{\lambda}) r_\beta + 290 (\ell r_\beta) r_\beta. \]
Now observe
\[ 144 r_\beta + 290 (\ell r_\beta) r_\beta \leq 290 e^{\ell r_\beta} r_\beta, \]
\[ \min(\lambda, 1.4 \sqrt{\lambda}) \leq e^\lambda. \]
The assumption on $\beta$ from Proposition \ref{prop:poisson-approximation-abelian} implies $290 r_\beta \leq 1/2$, so that applying Lemma \ref{lemma:lambda-ell-beta-bound} once more, we obtain 
$e^\lambda \leq e^{1.5 \ell r_\beta}$.
Now finish by combining the bounds.
\end{proof}

\begin{proof}[Proof of Proposition \ref{prop:left-tail-bound-abelian}]
Throughout this proof, $\varep = 0.001$. The assumptions on $\beta$ and $L$ in the proposition allow us to apply Lemma \ref{lemma:partition-function-ratio-minus-neighborhood-minimal-vortex} with $\vortex = P(e)$, for any $e \in \wloop$, with this particular value of $\varep$. Let $X_e$ be the indicator of the event $F_{P(e)}$. Observe for $0 \leq k \leq \ell$, we have
\[ \p(N = k) = \sum_{\substack{\wloop' \sse \wloop \\ |\wloop'| = k}} \p(X_e = 1, e \in \wloop', X_{e'} = 0, e' \in \wloop \backslash \wloop').\]
Fix $k$, and $\wloop' \sse \wloop$, $|\wloop'| = k$. We may assume that the vortices $P(e), e \in \wloop'$ are compatible, otherwise the associated probability is zero. Let
\[  \Gamma := N_\lbox(P(e), e \in \wloop') \cup \{P(e') : e' \in \wloop \backslash \wloop' \}.\]
Now by Lemma \ref{lemma:vortex-appear-not-appear-probability} and Remark \ref{remark:minimal-vortex-constant-activity}, we have
\[\p(X_e = 1, e \in \wloop', X_{e'} = 0, e' \in \wloop \backslash \wloop') = r_\beta^k \rho_{\mc{P}_\lbox}(z^\beta, \Gamma). \]
We want to apply Lemma \ref{lemma:partition-function-ratio-minus-neighborhood-minimal-vortex}, so we write the right hand side above as a telescoping product. Let the edges of $\wloop$ be numbered $e_1, \ldots, e_\ell$, such that $\wloop' = \{e_1, \ldots, e_k\}$. For $1 \leq j \leq k$, let $\Gamma_j := N_\lbox(P(e_i), 1 \leq i \leq j)$. For $k+1 \leq j \leq \ell$, let $\Gamma_j := \Gamma_k \cup \{P(e_i), k+1 \leq j \leq \ell\}$. Let $\Gamma_0 := \varnothing$. We have
\[\rho_{\mc{P}_\lbox}(z^\beta, \Gamma) = \prod_{j=1}^\ell \rho_{\mc{P} \backslash \Gamma_{j-1}}(z^\beta, \Gamma_j) . \]
Observe for $1 \leq j \leq \ell$, we have
\[ \mc{P}_\lbox \backslash \Gamma_j = (\mc{P}_\lbox \backslash \Gamma_{j-1}) \backslash (\Gamma_j \backslash \Gamma_{j-1}),\]
and moreover $\Gamma_j \backslash \Gamma_{j-1} \sse N_\lbox(P(e_j))$. Let $n_0(j)$ be the number of minimal vortices in 
\[ (\Gamma_j \backslash \Gamma_{j-1}) \cap (\mc{P}_\lbox \backslash \Gamma_{j-1}). \]
We may apply Lemma \ref{lemma:partition-function-ratio-minus-neighborhood-minimal-vortex} to obtain
\[|\log \rho_{\mc{P}\backslash \Gamma_{j-1}}(z^\beta, \Gamma_j) + n_0(j) r_\beta| \leq (n_0(j) + 2) \varep r_\beta.  \]
Now let $n_0 := \sum_{j=1}^\ell n_0(j)$, the number of minimal vortices in $\Gamma$. We obtain 
\[ |\log \rho_{\mc{P}_\lbox}(z^\beta, \Gamma)  + n_0 r_\beta| \leq (n_0 + 2\ell) \varep r_\beta, \]
which gives
\begin{align*}
\log \rho_{\mc{P}_\lbox}(z^\beta, \Gamma) &\leq - n_0 r_\beta\bigg(1 - \frac{n_0 + 2\ell}{n_0}\varep \bigg). 
\intertext{Now $n_0 \geq \ell$ (here we've used the assumption that $\wloop$ is self avoiding), which gives (note $\varep \leq 1/3$)} 
&\leq - \ell r_\beta (1 - 3 \varep).
\end{align*}
We thus have
\[ \p(X_e = 1, e \in \wloop', X_{e'} = 0, e' \in \wloop \backslash \wloop') \leq r_\beta^k \exp(- \ell r_\beta (1 - 3 \varep)), \]
and thus
\[ \p(N_\wloop = k) \leq \binom{\ell}{k} r_\beta^k \exp(- \ell r_\beta (1 - 3 \varep)).\]
Let $m := \floor{(1/2) \ell r_\beta}$. We have
\begin{align*}
\p(N_\wloop \leq m) &\leq \bigg(\sum_{k=0}^m \binom{\ell}{k} r_\beta^k \bigg)\exp(- \ell r_\beta(1 - 3 \varep)) \\
&\leq \bigg(\sum_{k=0}^m \frac{(\ell r_\beta)^k}{k!} \bigg) \exp(- \ell r_\beta(1 - 3 \varep)). 
\end{align*}
Let $u := \ell r_\beta$. We now use the formula (see \cite{WOLFRAM})
\[ \sum_{k=0}^m \frac{u^k}{k!} = e^u \frac{\Gamma(m+1, u)}{\Gamma(m+1)}, \]
where 
\[ \Gamma(m+1, u) = \int_u^\infty t^m e^{-t} dt, \]
and $\Gamma(m+1) = \Gamma(m+1, 0)$ is the standard Gamma function. To bound the ratio of $\Gamma$'s, let $Y \sim \mathrm{Gamma}(m+1, 1)$. Observe
\[ \frac{\Gamma(m+1, u)}{\Gamma(m+1)} = \p(Y \geq u). \]
Note $\E Y = m+1$. So assuming for the moment $m+1 \leq u$ (i.e. $\floor{(1/2) u} + 1 \leq u$), by a standard Chernoff bound we obtain
\[ \p(Y \geq u) \leq e^{-(u - (m+1))} \bigg(\frac{u}{m+1}\bigg)^{m+1}.\]
Thus
\[ \sum_{k=0}^m \frac{u^k}{k!} \leq \bigg(\frac{eu}{m+1}\bigg)^{m+1},\] 
and thus
\begin{align*} 
\p(N_\wloop \leq m) &\leq e^{-u(1- 3 \varep)} \bigg(\frac{eu}{m+1}\bigg)^{m+1}.
\intertext{Now observe $u / (m+1) \leq 2$ and $m \leq u /2$, and so}
&\leq 2e \exp(- u \bigg( \frac{1}{2} \log \frac{e}{2} - 3 \varep \bigg)).
\end{align*}
Now $\frac{1}{2} \log \frac{e}{2} > 0.153$, and since $\varep = 0.001$, we obtain
\[ \p(N_\wloop \leq (1/2) \ell r_\beta) \leq 2e \exp(- 0.15 \ell r_\beta). \]
Recall we assumed $m+1 \leq u$, i.e. $\floor{(1/2)u} + 1 \leq u$. If this is not the case, then $u < 2$. This implies
\[ \p(N_\wloop \leq (1/2) \ell r_\beta) = \p(N_\wloop = 0) \leq \exp(- \ell r_\beta (1 - 3 \varep)). \]
As $\varep = 0.001$, the above is bounded by $2e \exp(-0.15 \ell r_\beta)$.
\end{proof}

\section{The non-Abelian case}\label{section:non-abelian-case}

To extend the argument to the non-Abelian case, the approach will be to show that the key intermediate results of Section \ref{section:wilson-loop-exp-abelian} (i.e. Lemma \ref{lemma:only-minimal-vortices-contribute-probability-bound} to Lemma \ref{lemma:large-loop-abelian}) extend to the case of general gauge groups. Our main result (Theorem \ref{thm:main-result}) will then follow fairly easily -- see the end of Section \ref{section:wilson-loop-exp-non-abelian}. The proof of the intermediate results in the Abelian setting relied primarily on cluster expansion, and the key lemma that allowed us to use cluster expansion was Lemma \ref{lemma:activity_decomp}. The issue now is that the direct analogue of Lemma \ref{lemma:activity_decomp} does not hold if the gauge group is non-Abelian; this was pointed out by Szlach\`{a}nyi and Vecserny\`{e}s \cite{SV1989} and is due to non-trivial topological considerations. The paper \cite{SV1989} does give a cluster expansion for lattice gauge theories with finite non-Abelian gauge groups; however because the specific results we will need are slightly different, we will start from scratch, and moreover full proofs will be given. I should stress however that the main insight, which is to use algebraic topology to handle the non-Abelian case, is already present in \cite{SV1989}.

Throughout this section, $\lbox$ is some fixed cube in $\Z^4$, and $\wloop$ is a self avoiding loop in $\lbox$ of length $\ell$ . We will later make some assumptions on $\lbox$, but for now, it can be any cube which contains $\wloop$. As before, $\edges$ is the set of positively oriented edges of $\lbox$, and $\plaquettes$ is the set of positively oriented plaquettes of $\lbox$. Also, let $\vertices$ be the the set of unoriented vertices of $\lbox$. We will again think of edge configurations on $\lbox$ as elements of $G^\edges$, with the understanding that if $e$ is a negatively oriented edge, and $\sigma \in G^\edges$, then $\sigma_{-e} = \sigma_e^{-1}$. This convention allows expressions such as (recall equation \eqref{eq:ordered-product})
\[ \prod_{e \in \wloop} \sigma_e, \]
where the orientation of the edges $e$ in the loop $\wloop$ is taken into account, to be well defined.

\subsection{Topological preliminaries}\label{section:topological-preliminaries}

Recall that a cell complex is a certain type of topological space obtained by assembling cells of varying dimension; see e.g. Section 0.2.4 of \cite{STILL1993}. In our case, the cells will be unit squares of dimension at most two, i.e. vertices, edges, and plaquettes. So for us, a one dimensional cell complex, or 1-complex, is a space consisting of vertices and edges, and thus it is a graph. A two dimensional cell complex, or 2-complex, is a space consisting of vertices, edges, and plaquettes. Note also that although the names are similar, this is different from the lattice cell complex which was introduced in Section \ref{section:lattice-cell-complex}; the former is a topological space, while the latter is a formal collection of oriented cells.

In what follows, if we define a 1-complex by specifying a collection of edges, then that 1-complex is understood to also include the vertices of the edges in the collection. Similarly, if we define a 2-complex by specifying a collection of plaquettes, then that 2-complex is understood to also include the vertices and edges of the plaquettes in the collection.

Let $\oneskel$\label{notation:oneskel} denote the 1-skeleton of $\lbox$, i.e. the 1-complex obtained from the edges of $\lbox$. Let $\twoskel$ denote the 2-skeleton of $\lbox$\label{notation:twoskel}, i.e. the 2-complex obtained from the plaquettes of $\lbox$. Fix a vertex $x_0 \in \vertices$\label{notation:x-0}, and a spanning tree $T$\label{notation:T} of $\oneskel$. The fundamental group $\pi_1(\oneskel, x_0)$ may be described as follows. For any vertex $x \in \vertices$, let $\upath_x$ denote the unique path in $T$ from $x_0$ to $x$. For any edge $e = (x, y) \in \oneskel$, let $a_e$ be the closed loop obtained by starting at $x_0$, following $\upath_x$ to $x$, then traversing $e = (x, y)$, then following the path $\upath_y$ in reverse, from $y$ to $x_0$. Symbolically, we write
\[\label{notation:a-e} a_e = \upath_x e \upath_y^{-1}. \]
(Note if $e$ is in the spanning tree $T$, and $x$ is closer than $y$ to $x_0$ (in the distance induced by $T$), then $a_e$ is the path which starts at $x_0$, follows the path $\upath_y$ to $y$, and then retraces its steps, following the path $\upath_y$ in reverse, from $y$ to $x_0$. Thus in this case $a_e$ is equivalent to the trivial path.)
We then have that $\pi_1(\oneskel, x_0)$ can be presented as the free group with generators $a_e, e \in \oneskel \backslash T$. I.e.,
\[ \pi_1(\oneskel, x_0) = \langle a_e, e \in \oneskel \backslash T \rangle. \]
For notational convenience, we will sometimes omit $x_0$ and write $\pi_1(\oneskel)$ instead of $\pi_1(\oneskel, x_0)$.

Now observe that edge configurations $\sigma \in G^\edges$ naturally induce a homomorphism from $\pi_1(\oneskel, x_0)$ to $G$. Let $\Hom(\pi_1(\oneskel, x_0), G)$ be the set of homomorphisms from $\pi_1(\oneskel, x_0)$ to $G$, and let the homomorphism $\homsym_T^{x_0}(\sigma) \in \Hom(\pi_1(\oneskel, x_0), G)$\label{notation:homsym} be defined as follows. We specify the value of $\homsym_T^{x_0}(\sigma)$ on each of the generators $a_e, e \in \oneskel \backslash T$. For such an $e$, let the edges traversed by $a_e$ be, in order, $e_1, \ldots, e_n$. Define
\[ \homsym_T^{x_0}(\sigma)(a_e) := \sigma_{e_1} \cdots \sigma_{e_n}. \]
The next three lemmas show that homomorphisms from $\pi_1(\oneskel, x_0)$ to $G$ can be thought of as equivalence classes of edge configurations. This is analogous in the Abelian case to thinking of 2-forms $q$ such that $dq = 0$ as equivalence classes of edge configurations.

\begin{lemma}\label{lemma:induced-homomorphism-onto}
Let $x_0 \in \vertices$, and let $T$ be a spanning tree of $\oneskel$. For any homomorphism $\homsym \in \Hom(\pi_1(\oneskel, x_0), G)$, there exists an edge configuration $\sigma \in G^\edges$ such that
\[\homsym_T^{x_0}(\sigma) = \homsym .\]
\end{lemma}
\begin{proof}
For $e \in T$, let $\sigma_e := \groupid$. For $e \in \oneskel \backslash T$, let $\sigma_e := \homsym(a_e)$. Then for all $e \in \oneskel \backslash T$, we have $\homsym_T^{x_0}(\sigma)(a_e) = \sigma_e = \homsym(a_e)$, and the desired result follows.
\end{proof}

\begin{lemma}\label{lemma:induced-homomorphism-equal-implies-gauge-equiv}
Let $x_0 \in \vertices$, and let $T$ be a spanning tree of $\oneskel$. Consider edge configurations $\sigma, \tau \in G^\edges$. Then $\homsym_T^{x_0}(\sigma) = \homsym_T^{x_0}(\tau)$ if and only if there exists a function $h \in G^{\vertices}$, with $h_{x_0} = \groupid$, such that for all edges $e = (x, y) \in \oneskel$, we have 
\[ \sigma_e = h_x \tau_e h_y^{-1}. \]
\end{lemma}
\begin{proof}
We prove the nontrivial direction. Suppose $\sigma, \tau$ induce the same homomorphism. We define the desired $h$ as follows. First, as required, $h_{x_0} := \groupid$. Now for any edge $e = (x_0, x) \in T$, define $h_x$ so that
\[ \sigma_e = h_{x_0} \tau_e h_x^{-1},  \]
i.e.
\[ h_x := \sigma_e^{-1} \tau_e. \]
More generally, for any $x \in \vertices$, suppose $\upath_x = e_1 \cdots e_n$. Define
\[ h_x := (\sigma_{e_1} \cdots \sigma_{e_n})^{-1} \tau_{e_1} \cdots \tau_{e_n}.\]
We now show that $h$ is as required. Fix an edge $e = (x, y)$. Suppose first that $e \in T$. Without loss of generality, suppose that $y$ is further from the root $x_0$ of $T$ than $x$, so that $\upath_y = \upath_x e$. If $\upath_x = e_1 \cdots e_n$, then
\[ h_x \tau_e h_y^{-1} = (\sigma_{e_1} \cdots \sigma_{e_n})^{-1} \tau_{e_1} \cdots \tau_{e_n} \tau_e (\tau_{e_1} \cdots \tau_{e_n} \tau_e)^{-1} \sigma_{e_1} \cdots \sigma_{e_n} \sigma_e = \sigma_e, \]
as desired. Now suppose $e \in \oneskel \backslash T$. Let $\upath_x = e_1, \ldots, e_n$, $\upath_y = f_1, \ldots, f_m$. Then
\begin{align*} 
h_x \tau_e h_y^{-1} &= (\sigma_{e_1} \cdots \sigma_{e_n})^{-1} \tau_{e_1} \cdots \tau_{e_n} \tau_e (\tau_{f_1} \cdots \tau_{f_m})^{-1} \sigma_{f_1} \cdots \sigma_{f_m} \\
&= (\sigma_{e_1} \cdots \sigma_{e_n})^{-1} \homsym_T^{x_0}(\tau)(a_e) \sigma_{f_1} \cdots \sigma_{f_m}.
\end{align*}
To finish, we want to show
\[(\sigma_{e_1} \cdots \sigma_{e_n})^{-1} \homsym_T^{x_0}(\tau)(a_e) \sigma_{f_1} \cdots \sigma_{f_m} = \sigma_e. \]
If we move all the $\sigma$'s to the right hand side, we see that we need to show
\[ \homsym_T^{x_0}(\tau)(a_e) = \homsym_T^{x_0}(\sigma)(a_e), \]
which is true by assumption.
\end{proof}

\begin{lemma}\label{lemma:different-gauge-transform-implies-different-edge-configuration}
Let $x_0 \in \vertices$. Let $\sigma \in G^\edges$, and $h \in G^{\vertices}$ with $h_{x_0} = \groupid$. Let $\tau \in G^\edges$ be the edge configuration given by $\tau_e := h_x \sigma_e h_y^{-1}$ for each $e = (x, y) \in \edges$. If there exists $x \in \vertices$ such that $h_x \neq \groupid$, then $\sigma \neq \tau$.
\end{lemma}
\begin{proof}
Fix a spanning tree $T$ of $\oneskel$.
Take a vertex $x \in \vertices$ such that $h_x \neq \groupid$, and also such that if $x_0, x_1, \ldots, x_n = x$ is the sequence of vertices in the unique path from $x_0$ to $x$ in $T$, then $h_{x_0} = \cdots = h_{x_{n-1}} = \groupid$. Then 
\[ \tau_{(x_{n-1}, x)} = \sigma_{(x_{n-1}, x)} h_x^{-1} \neq \sigma_{(x_{n-1}, x)}. \qedhere\]
\end{proof}

We have the following consequence of Lemmas \ref{lemma:induced-homomorphism-onto}, \ref{lemma:induced-homomorphism-equal-implies-gauge-equiv}, and \ref{lemma:different-gauge-transform-implies-different-edge-configuration}, which will be used later.

\begin{cor}\label{cor:number-of-edge-configurations-map-to-hom}
Let $x_0 \in \vertices$, and let $T$ be a spanning tree of $\oneskel$. For any homomorphism $\homsym \in \Hom(\pi_1(\oneskel, x_0), G)$, the number of edge configurations $\sigma \in G^\edges$ such that $\homsym_T^{x_0}(\sigma) = \homsym$ is $|G|^{|\vertices| - 1}$. 
\end{cor}

For each plaquette $p \in \plaquettes$, let $e_1^p, e_2^p, e_3^p, e_4^p$ be the edges on the boundary of $p$, traversed in a positive orientation. Let
\[\label{notation:C-p} C_p := a_{e_1^p} a_{e_2^p} a_{e_3^p} a_{e_4^p} \in \pi_1(\oneskel, x_0). \]
In what follows, we will often suppress the dependence on $p$, and write $e_1, \ldots, e_4$ instead of $e_1^p, \ldots, e_4^p$, for notational convenience.
Observe that $C_p$ depends on $(x_0, T)$, but we will hide this dependence. Moreover, $C_p$ is not uniquely determined, since there are four possible starting points, corresponding to the four vertices of the plaquette $p$, but this will not matter. Note that due to the backtracking of edges, we may also express
\beq\label{eq:cp-equiv-to-lasso} C_p = w_{x_1} e_1 e_2 e_3 e_4 w_{x_1}^{-1},\eeq
where $x_1$ is chosen starting vertex. This shows that for any edge configuration $\sigma \in G^\edges$, and for any plaquette $p \in \plaquettes$, there is some $g_p \in G$ such that (recall equation \eqref{eq:sigma-p-def} for the definition of $\sigma_p$)
\beq\label{eq:choice-of-cp-conjugate} \homsym_{T}^{x_0}(\sigma)(C_p) = g_p \sigma_p g_p^{-1}.\eeq
As $\varphi_\beta$ (see equation \eqref{eq:phi-beta-def} for the definition) is conjugate invariant (which follows since the trace is conjugate invariant), this implies
\[ \prod_{p \in \plaquettes} \varphi_\beta(\psi_T^{x_0}(\sigma)(C_p)) = \prod_{p \in \plaquettes} \varphi_\beta(\sigma_p).\]
Now for homomorphisms $\homsym \in \Hom(\pi_1(\oneskel, x_0), G)$, define 
\[ \supp(\homsym) := \{p \in \plaquettes: \homsym(C_p) \neq \groupid\} .\] 
For plaquette sets $\plaqset \sse \plaquettes$, define
\[ \Phi(\plaqset) \label{notation:Phi-nonabelian} := \sum_{\substack{\homsym \in \Hom(\pi_1(\oneskel, x_0), G) \\ \supp(\homsym) = \plaqset}} \prod_{p \in \plaqset} \varphi_\beta(\homsym(C_p)).\]
By \eqref{eq:choice-of-cp-conjugate}, the definitions of $\Phi$ and $\supp(\psi)$ do not depend on the choices of starting points for the $C_p$. Also, in principle, $\Phi$ depends on $(x_0, T)$, but we now show that it is in fact independent of these choices.

\begin{lemma}\label{lemma:choice-of-tree-invariant}
Let $x_0, x_0' \in \vertices$, and let $T, T'$ be spanning trees of $\oneskel$. Suppose $\Phi$ is defined using $(x_0, T)$, and $\Phi'$ is defined using $(x_0', T')$. Then $\Phi = \Phi'$.
\end{lemma}
\begin{proof}
Fix $\plaqset \sse \plaquettes$. We want to show $\Phi(\plaqset) = \Phi'(\plaqset)$. For $x \in \vertices$, $p \in \plaquettes$, let $w_x, C_p$ be defined in terms of $(x_0, T)$, and $w_x', C_p'$ be defined in terms of $(x_0', T')$. First, suppose we have an isomorphism $\xi : \pi_1(\oneskel, x_0) \ra \pi_1(\oneskel, x_0')$, such that for any $p \in \plaquettes$, we have some $b_p \in \pi_1(\oneskel, x_0')$ such that
\beq\label{eq:group-iso-conjugate-cp} 
\xi(C_p) = b_p C_p' b_p^{-1}. 
\eeq
Observe that $\xi$ induces a bijection between the sets $\Hom(\pi_1(\oneskel, x_0), G)$ and $\Hom(\pi_1(\oneskel, x_0'), G)$, for example by mapping the homomorphism $\homsym' \in \Hom(\pi_1(\oneskel, x_0'), G)$ to $\homsym' \circ \xi$. Moreover, by \eqref{eq:group-iso-conjugate-cp}, we have $\homsym'(C_p') = \groupid$ if and only if $\homsym' \circ \xi(C_p) =\groupid$, which gives $\supp(\homsym') = \supp(\homsym' \circ \xi)$. We then have
\begin{align*}
\sum_{\substack{\homsym \in \Hom(\pi_1(\oneskel, x_0), G) \\ \supp(\homsym) = \plaqset}} \prod_{p \in \plaqset} \varphi_\beta(\homsym(C_p)) &= \sum_{\substack{\homsym' \in \Hom(\pi_1(\oneskel, x_0'), G) \\ \supp(\homsym') = \plaqset}} \prod_{p \in \plaqset} \varphi_\beta(\homsym' \circ \xi(C_p)).
\intertext{Again by \eqref{eq:group-iso-conjugate-cp}, and the fact that $\varphi_\beta$ is conjugate invariant, we obtain}
&= \sum_{\substack{\homsym' \in \Hom(\pi_1(\oneskel, x_0'), G) \\ \supp(\homsym') = \plaqset}} \prod_{p \in \plaqset} \varphi_\beta(\homsym'(C_p')),
\end{align*}
as desired.

It remains to exhibit the isomorphism $\xi$. Let $\upath$ be the unique path in $T'$ from $x_0'$ to $x_0$. For an element $a \in \pi_1(\oneskel, x_0)$, define
\[ \xi(a) := \upath a \upath^{-1}.\]
Then $\xi$ is an isomorphism, and it remains to verify \eqref{eq:group-iso-conjugate-cp}. Recalling \eqref{eq:cp-equiv-to-lasso}, we have that $C_p$ may be written as a loop of the form $\upath_x e_1 e_2 e_3 e_4 \upath_x^{-1}$, for some vertex $x$ of $p$. Then $\xi(C_p) = \upath \upath_x e_1 e_2 e_3 e_4 \upath_x^{-1} \upath^{-1}$.
As the starting point of $C_p'$ does not affect $\Phi'$, we may assume that $C_p'$ has the same starting point as $C_p$. Thus we have $C_p' = \upath_{x}' e_1 e_2 e_3 e_4 (\upath_{x}')^{-1}$. Setting $b_p := w w_x (w_x')^{-1}$, we see that \eqref{eq:group-iso-conjugate-cp} holds, as desired.
\end{proof}

We have translated edge configurations into homomorphisms, but it will sometimes be convenient to work directly with the edge configurations. The next two lemmas allow us to go from homomorphisms back to edge configurations. 

\begin{lemma}\label{lemma:uniqueness-of-gauge-fix}
Let $x_0 \in \vertices$, and let $T$ be a spanning tree of $\oneskel$. For every homomorphism $\homsym \in \Hom(\pi_1(\oneskel, x_0), G)$, there is a unique edge configuration $\sigma \in G^\edges$ such that $\sigma = \groupid$ on $T$, and $\homsym_T^{x_0}(\sigma) = \homsym$.
\end{lemma}
\begin{proof}
The existence of such $\sigma$ follows from the proof of Lemma \ref{lemma:induced-homomorphism-onto}. For uniqueness, observe for $e \in \oneskel \backslash T$, we must have
\[ \sigma_e = \homsym_T^{x_0}(\sigma)(a_e) = \homsym(a_e). \qedhere\]
\end{proof}

For an edge configuration $\sigma \in G^\edges$, define the support of $\sigma$ by $\supp(\sigma) := \{p \in \plaquettes : \sigma_p \neq \groupid\}$\label{notation:supp-sigma}. For a spanning tree $T$ of $\oneskel$, let $GF(T) := \{\sigma \in G^\edges : \sigma = \groupid \text{ on T}\}$\label{notation:GF-T} (``GF" stands for gauge-fixed).

\begin{lemma}\label{lemma:phi-sum-gauge-fixed}
For a plaquette set $\plaqset \sse \plaquettes$, and any spanning tree $T$ of $\oneskel$, we have
\[ \Phi(\plaqset) = \sum_{\substack{\sigma \in GF(T) \\ \supp(\sigma) = \plaqset}} \prod_{p \in \plaqset} \varphi_\beta(\sigma_p). \]
\end{lemma}
\begin{proof}
Fix a vertex $x_0 \in \vertices$. By Lemma \ref{lemma:uniqueness-of-gauge-fix}, for any homomorphism $\homsym$, there is a unique edge configuration $\sigma \in GF(T)$ such that $\homsym = \homsym_T^{x_0}(\sigma)$. Moreover, by \eqref{eq:group-iso-conjugate-cp}, we have that $\supp(\homsym_T^{x_0}(\sigma)) = \supp(\sigma)$, and $\varphi_\beta(\homsym_T^{x_0}(\sigma)(C_p)) = \varphi_\beta(\sigma_p)$ for all $p \in \plaquettes$.
\end{proof}

We now begin to explore the factorization properties of $\Phi$. In particular, we want to show partial analogues of Lemma \ref{lemma:activity_decomp}, which was the key result that allowed us to use cluster expansion to show that the random collection of surfaces we were considering had a lot of independence. The analogues will be given by Lemmas \ref{lemma:minimal-vortex-activity-factor-nonabelian} and \ref{lemma:well-separated-vortices-activity-decomp}, but first we will need some preliminary results. Recall that a topological space is said to be simply connected if it is path connected and has trivial fundamental group. We say that $T$ is a spanning tree of a 2-complex $S$, if $T$ is a spanning tree of the 1-skeleton of $S$.

\begin{lemma}\label{lemma:identity-on-spanning-tree-implies-identity-on-everything}
Let $S$ be a simply connected 2-complex, and a subcomplex of $\twoskel$. Let $T_0$ be a spanning tree of $S$. Suppose $\sigma \in G^\edges$ is such that $\sigma_e = 1$ for all edges $e \in T_0$, and $\sigma_p = \groupid$ for all plaquettes $p \in S$. Then $\sigma_e = 1$ for all edges $e \in S$.
\end{lemma}
\begin{proof}
Let $S_1$ be the 1-skeleton of $S$. Fix a vertex $x_0$ in $S_1$. We may obtain a presentation of $\pi_1(S_1, x_0)$ just as we did for $\pi_1(\oneskel, x_0)$. I.e., for a vertex $x$ of $S_1$, let $\upath_x^0$ be the unique path from $x_0$ to $x$ in $T_0$. For $e = (x, y) \in S_1 \backslash T_0$, let $a_e^0 := \upath_x^0 e (\upath_y^0)^{-1}$. Then $\pi_1(S_1, x_0) = \langle a_e^0, e \in S_1 \backslash T_0\rangle$.

Just as for $\oneskel$, observe that $\sigma$ induces a homormophism $\pi_1(S_1, x_0) \ra G$. Call this homomorphism $\homsym$. As $\sigma = \groupid$ on $T_0$, we additionally have that $\homsym(a_e^0) = \sigma_e$ for all $e \in S_1 \backslash T_0$. Now for plaquettes $p \in S$, let $C^0_p := a_{e_1}^0 a_{e_2}^0 a_{e_3}^0 a_{e_4}^0$, where $e_1, e_2, e_3, e_4$ are the edges of $p$, traversed in a positive orientation. For all plaquettes $p \in S$, we have that $\psi(C^0_p)$ and $\sigma_p$ are in the same conjugacy class of $G$. Thus by the assumption that $\sigma_p = 1$ for all plaquettes $p \in S$, we also have $\homsym(C^0_p) = \groupid$ for all plaquettes $p \in S$. Next, observe that (see e.g. Section 4.1.3 of \cite{STILL1993})
\[ \pi_1(S, x_0) = \langle a_e^0, e \in S_1 \backslash T_0 ~|~ C_p^0, p \in S \rangle =  \pi_1(S_1, x_0) / N, \]
where $N$ is the normal subgroup of $\pi_1(S_1, x_0)$ generated by $C^0_p, p \in S$. Note that $N \sse \ker \homsym$, and thus by the fundamental theorem of homomorphisms, there is a homormophism $\twoskelhom : \pi_1(S, x_0) \ra G$ such that $\homsym = \twoskelhom \circ \Pi$, where $\Pi : \pi_1(S_1, x_0) \ra \pi_1(S, x_0)$ is the natural projection map. But by assumption, we have that $\pi_1(S, x_0) = \{\groupid\}$. Thus for any $e \in S_1 \backslash T_0$, we have that
\[ \sigma_e =  \homsym(a_e^0) = \zeta \circ \Pi (a_e^0) = \zeta(\groupid) = \groupid. \]
As $\sigma = \groupid$ on $T_0$ by assumption, we conclude that $\sigma_e = \groupid$ for all $e \in S$, as desired.
\end{proof}

The following lemma can be thought of as an analogue of Lemma \ref{lemma:gauge_eq_decomp}.

\begin{lemma}\label{lemma:general-edge-configuration-decomposition}
Let $S_1, S_2$ be 2-complexes, and subcomplexes of $\twoskel$. Suppose $S_1 \cup S_2 = \twoskel$, and additionally, suppose $S_1, S_2$, and $S_1 \cap S_2$ are simply connected. Let $T$ be a spanning tree of $\oneskel$, which contains spanning trees of $S_1, S_2$, and $S_1 \cap S_2$. Suppose $\plaqset_1, \plaqset_2 \sse \plaquettes$, such that $\plaqset_1 \sse S_1$, $\plaqset_2 \sse S_2$, and no plaquette of $\plaqset_1$ or $\plaqset_2$ is in $S_1 \cap S_2$. Then there is a bijection between edge configurations $\sigma \in GF(T)$ such that $\supp(\sigma) = \plaqset_1 \cup \plaqset_2$, and tuples of edge configurations $(\sigma^1, \sigma^2)$ such that $\sigma^i \in GF(T)$, $\supp(\sigma^i) = \plaqset_i$, $i = 1, 2$. Moreover, if $\sigma$ is mapped to $(\sigma^1, \sigma^2)$, then $\sigma = \sigma^1 \sigma^2$, $\sigma^1 = \groupid$ on $S_2$, and $\sigma^2 = \groupid$ on $S_1$. Consequently, for $i = 1, 2$, $p \in S_i$, we have $\sigma_p = \sigma^i_p$.
\end{lemma}
\begin{proof}
Suppose we have $\sigma$ such that $\sigma = 1$ on $T$, and $\supp(\sigma) = \plaqset_1 \cup \plaqset_2$. Since $T$ contains a spanning tree of $S_1 \cap S_2$, we have that $\sigma = 1$ on a spanning tree of $S_1 \cap S_2$. Moreover, we have that $\sigma_p = 1$ on all $p \in S_1 \cap S_2$, since no plaquette of $\plaqset_1$ or $\plaqset_2$ is in $S_1 \cap S_2$. Thus by Lemma \ref{lemma:identity-on-spanning-tree-implies-identity-on-everything}, we have that $\sigma = 1$ on $S_1 \cap S_2$. For $i = 1, 2,$ define
\[ \sigma^i_e := \begin{cases} \sigma_e & e \in S_i \\ 1 & e \notin S_i\end{cases}. \]
Clearly $\sigma^i = 1$ on $T$. 

We now look at $\supp(\sigma^i)$. First, as $\sigma = 1$ on $S_1 \cap S_2$ and $S_1 \cup S_2 = \twoskel$, we have that $\sigma = \sigma^1 \sigma^2$, and $\sigma^1 = 1$ on $S_2$, $\sigma^2 = 1$ on $S_1$. For $p \in S_1$, all edges of $p$ are in $S_1$. Thus $\sigma_p =\sigma^1_p$, and $\sigma^2_p = 1$. Likewise, for $p \in S_2$, we obtain $\sigma_p = \sigma^2_p$, and $\sigma^1_p = 1$. This shows $\supp(\sigma^i) = \plaqset_i$, $i = 1, 2$.

Conversely, suppose we start with $(\sigma^1, \sigma^2)$ such that $\sigma^i \in GF(T)$, and $\supp(\sigma^i) = \plaqset_i$, $i = 1, 2$. Then $\sigma^1 = 1$ on $T$, and $\sigma^1_p = 1$ for all $p \in S_2$. As $T$ contains a spanning tree of $S_2$, by Lemma \ref{lemma:identity-on-spanning-tree-implies-identity-on-everything}, we have that $\sigma^1 = 1$ on $S_2$. Similarly, we obtain $\sigma^2 = 1$ on $S_1$. Thus if we define $\sigma = \sigma^1 \sigma^2$, then $\sigma =1$ on $T$, and for all $p \in S_i$, we have $\sigma_p = \sigma^i_p$. This shows $\supp(\sigma) = \plaqset_1 \cup \plaqset_2$.
\end{proof}

What Lemma \ref{lemma:general-edge-configuration-decomposition} says is that while the analogue of Lemma \ref{lemma:activity_decomp} may not be true for general compatible plaquette sets $\plaqset_1, \plaqset_2 \sse \plaquettes$, it will be true for those plaquette sets for which we can find a decomposition $\twoskel = S_1 \cup S_2$ which satisfies the conditions of Lemma \ref{lemma:general-edge-configuration-decomposition}. It turns out we can always do so if one of the plaquette sets is a minimal vortex. We begin to show this next.

Let $e \in \edges$ be such that $P(e) \sse \plaquettes$. Let $S_e$\label{notation:S-e} be the cell complex obtained by including any plaquette which is in a 3-cell that contains the edge $e$ (note if $P(e) \sse \plaquettes$, then any 3-cell that contains $e$ is in $\lbox$). Let $\partial S_e$ be the cell complex obtained by deleting $e$, and all plaquettes in $P(e)$, from $S_e$. Let $S_e^c$ be the cell complex obtained by deleting $e$, and all plaquettes in $P(e)$, from $\twoskel$. Observe $\partial S_e = S_e \cap S_e^c$.

To help visualize, suppose first that we are in three dimensions, and $e$ is the edge between the vertices $(0, 0, 0)$ and $(1, 0, 0)$. Then $\partial S_e$ is the boundary of the rectangular prism $[0, 1] \times [-1, 1]^2$. When we go to four dimensions, with $e$ now the edge between $(0, 0, 0, 0)$ and $(1, 0, 0, 0)$, $\partial S_e$ is the union of the boundaries of the three rectangular prisms $[0, 1] \times [-1, 1]^2 \times \{0\}$, $[0, 1] \times [-1, 1] \times \{0\} \times [-1, 1]$, and $[0, 1] \times \{0\} \times [-1, 1]^2$. The boundaries of the first two prisms intersect at the closed curve which is the boundary of the rectangle $[0, 1] \times [-1, 1] \times \{0\}^2$, and similarly the boundaries of the first and third prisms intersect at the boundary of $[0, 1] \times \{0\} \times [-1, 1] \times \{0\}$, and the boundaries of the second and third prisms intersect at the boundary of $[0, 1] \times \{0\}^2 \times [-1, 1]$. 

We will need the following topological statement, whose proof is left to the appendix.

\begin{lemma}\label{lemma:minimal-vortex-cell-complexes-simply-connected}
Suppose $e \in \edges$ is such that $P(e) \sse \plaquettes$. Then $S_e, S_e^c$, and  $\partial S_e$ are all simply connected.
\end{lemma}

\begin{lemma}\label{lemma:minimal-vortex-edge-config-decomp}
Suppose $\plaqset_1, \plaqset_2 \sse \plaquettes$ are compatible, and $\plaqset_1$ is a minimal vortex, i.e. $\plaqset_1 = P(e_0)$ for some edge $e_0$. Let $T$ be a spanning tree of $\oneskel$ which contains a spanning tree of $\partial S_{e_0}$. Then there is a bijection between edge configurations $\sigma \in GF(T)$ such that $\supp(\sigma) = \plaqset_1 \cup \plaqset_2$, and tuples of edge configurations $(\sigma^1, \sigma^2)$, such that $\sigma^i \in GF(T)$, $\supp(\sigma^i) = \plaqset_i$, $i = 1, 2$. Moreover, if $\sigma$ is mapped to $(\sigma^1, \sigma^2)$, then $\sigma = \sigma^1 \sigma^2$, and $\sigma^1 = \groupid$ outside $e_0$, and $\sigma^2 = \groupid$ on $S_{e_0}$.
\end{lemma}
\begin{proof}
We apply Lemma \ref{lemma:general-edge-configuration-decomposition}. To verify the conditions, observe $S_{e_0} \cap S_{e_0}^c = \partial S_{e_0}$, and by Lemma \ref{lemma:minimal-vortex-cell-complexes-simply-connected}, $S_{e_0}, S_{e_0}^c, \partial S_{e_0}$ are all simply connected. Observe also that a spanning tree of $\partial S_{e_0}$ is also a spanning tree of $S_{e_0}$, and a spanning tree of $\oneskel$ which does not use $e_0$ is also a spanning tree of $S_{e_0}^c$. Finally, no plaquette of $P(e_0)$ is in $S_{e_0}^c$, and by definition of compatibility, we have that no plaquette of $\plaqset_2$ can be in $S_{e_0}$.
\end{proof}

\begin{lemma}\label{lemma:minimal-vortex-activity-factor-nonabelian}
Suppose $\plaqset = \plaqset_1 \cup \plaqset_2 \sse \plaquettes$, where $\plaqset_1, \plaqset_2$ are compatible, and $\plaqset_1$ is a minimal vortex. Then $\Phi(\plaqset) = \Phi(\plaqset_1) \Phi(\plaqset_2)$.
\end{lemma}
\begin{proof}
Let $e$ be such that $\plaqset_1 = P(e)$. Let $T$ be a spanning tree of $\oneskel$ which contains a spanning tree of $\partial S_e$. Applying Lemmas \ref{lemma:phi-sum-gauge-fixed} and \ref{lemma:minimal-vortex-edge-config-decomp}, we have
\begin{align*}
\Phi(\plaqset) &= \sum_{\substack{\sigma \in GF(T) \\ \supp(\sigma) = \plaqset}} \prod_{p \in \plaqset} \varphi_\beta(\sigma_p) \\
&=  \sum_{\substack{\sigma^1 \in GF(T) \\ \supp(\sigma^1) = \plaqset_1}} \sum_{\substack{\sigma^2 \in GF(T) \\ \supp(\sigma^2) = \plaqset_2}} \prod_{p \in \plaqset_1} \varphi_\beta(\sigma^1_p) \prod_{p \in \plaqset_2} \varphi_\beta(\sigma^2_p) \\
&= \Phi(\plaqset_1) \Phi(\plaqset_2). \qedhere
\end{align*}
\end{proof}

By repeatedly applying Lemma \ref{lemma:minimal-vortex-activity-factor-nonabelian}, we obtain the following corollary.

\begin{cor}\label{cor:minimal-vortex-many-factor-non-abelian}
Let $\plaqset \sse \plaquettes$. Suppose we have a decomposition $\plaqset = \plaqset_1 \cup \cdots \cup \plaqset_n$, such that $\plaqset_1, \ldots, \plaqset_n$ are compatible, and such that for some $0 \leq k \leq n$, $\plaqset_1, \ldots, \plaqset_k$ are minimal vortices. Then
\[ \Phi(\plaqset) = \bigg(\prod_{i=1}^k \Phi(\plaqset_i)\bigg) \Phi(\plaqset_{k+1} \cup \cdots \cup \plaqset_n).  \] 
\end{cor}

We also have the following consequences of Lemmas \ref{lemma:identity-on-spanning-tree-implies-identity-on-everything} and \ref{lemma:minimal-vortex-edge-config-decomp}.

\begin{cor}\label{cor:homomorphism-number-factors-minimal-vortex}
Suppose $\plaqset_1, \plaqset_2 \sse \plaquettes$ are compatible, and $\plaqset_1$ is a minimal vortex. Then the number of $\psi \in \Hom(\pi_1(\oneskel), G)$ such that $\supp(\psi) = \plaqset_1 \cup \plaqset_2$ is equal to $(\abs{G} - 1)$ times the number of $\psi \in \Hom(\pi_1(\oneskel), G)$ such that $\supp(\psi) = \plaqset_2$.
\end{cor}
\begin{proof}
Let $e_0$ be such that $\plaqset_1 = P(e_0)$. Fix a spanning tree $T$ of $\oneskel$ which contains a spanning tree of $\partial S_{e_0}$. By Lemma \ref{lemma:uniqueness-of-gauge-fix}, we have that the number of $\psi \in \Hom(\pi_1(\oneskel), G)$ such that $\supp(\psi) = \plaqset_1 \cup \plaqset_2$ is equal to the number of $\sigma \in GF(T)$ such that $\supp(\sigma) = \plaqset_1 \cup \plaqset_2$. By Lemma \ref{lemma:minimal-vortex-edge-config-decomp}, this number is equal to the number of $(\sigma^1, \sigma^2)$ such that $\sigma^i \in GF(T)$, and $\supp(\sigma^i) = \plaqset_i$, $i = 1, 2$. Moreover, we have that $\sigma^1 = \groupid$ outside of $e_0$. Thus there are $\abs{G} - 1$ choices of $\sigma^1$, corresponding to the assignment of $\sigma^1_{e_0}$ to a non-identity element of $G$. To finish, observe again by Lemma \ref{lemma:uniqueness-of-gauge-fix} that the number of $\sigma^2 \in GF(T)$ such that $\supp(\sigma^2) = \plaqset_2$ is equal to the number of $\homsym \in \Hom(\pi_1(\oneskel), G)$ such that $\supp(\homsym) = \plaqset_2$.
\end{proof}

\begin{cor}\label{cor:homomorphism-number-single-minimal-vortex}
Suppose $e_0 \in \edges$ is such that $P(e_0) \sse \plaquettes$.
The number of homomorphisms $\psi \in \Hom(\pi_1(\oneskel), G)$ such that $\supp(\psi) = P(e_0)$ is $\abs{G} - 1$. Also, $\Phi(P(e_0)) = r_\beta$ (recall the definition of $r_\beta$ \eqref{eq:r-beta-def}).
\end{cor}
\begin{proof}
Take a spanning tree $T$ of $\oneskel$ which does not use $e_0$, so that it is also a spanning tree of $S_{e_0}^c$. The number of such $\psi$ is equal to the number of $\sigma \in GF(T)$ such that $\supp(\sigma) = P(e_0)$. For any such $\sigma$, note that as $T$ is also a spanning tree of $S_{e_0}^c$, we have $\sigma = \groupid$ on a spanning tree of $S_{e_0}^c$. Also, for all $p \in S_{e_0}^c$, we have $\sigma_p  =1$. Thus by Lemma \ref{lemma:identity-on-spanning-tree-implies-identity-on-everything}, we have that $\sigma_e = 1$ for all $e \in S_{e_0}^c$, that is, for all $e \neq e_0$. As $\supp(\sigma) = P(e_0)$, we must have $\sigma_{e_0} \neq \groupid$. Thus there are $\abs{G} - 1$ possible choices for $\sigma$, corresponding to the assignment of $\sigma_{e_0}$ to a non-identity element of $G$. Moreover, this implies $\Phi(P(e_0)) = r_\beta$, as desired.
\end{proof}

By repeatedly applying Corollary \ref{cor:homomorphism-number-factors-minimal-vortex}, and applying Corollary \ref{cor:homomorphism-number-single-minimal-vortex} once, we arrive at the following.

\begin{cor}\label{cor:only-minimal-vortices-edge-config}
Let $\vortex_i = P(e_i) \sse \plaquettes$, $1 \leq i \leq k$, be compatible minimal vortices. Fix $x_0 \in \vertices$. Then the number of $\psi \in \Hom(\pi_1(\oneskel, x_0), G)$ such that $\supp(\psi) = \vortex_1 \cup \cdots \cup \vortex_k$ is equal to $(\abs{G} - 1)^k$. Consequently, for a spanning tree $T$ of $\oneskel$ which does not contain $e_1, \ldots, e_k$, any such $\psi$ is of the form $\psi = \psi_T^{x_0}(\sigma)$, where $\sigma$ is such that $\sigma_e \neq 1$ if and only if $e \in \{e_1, \ldots, e_k\}$.
\end{cor}

\begin{remark}
Note if $e_1, \ldots, e_k$ are such that $P(e_1), \ldots, P(e_k) \sse \plaquettes$, and moreover the vortices are compatible, then the graph obtained by removing $e_1, \ldots, e_k$ from $\oneskel$ is connected. Thus there always exists a spanning tree $T$ of $\oneskel$ which does not contain $e_1, \ldots, e_k$.
\end{remark}

By repeatedly applying Lemma \ref{lemma:minimal-vortex-edge-config-decomp}, we arrive at the following.

\begin{cor}\label{cor:edge-config-bijection-tuple}
Let $\vortex_1, \ldots, \vortex_k, \plaqset' \sse \plaquettes$ be compatible, with $\vortex_i = P(e_i)$, $1 \leq i \leq k$. Let $\plaqset := \vortex_1 \cup \cdots \cup \vortex_k \cup \plaqset'$. Suppose $T$ is a spanning tree of $\oneskel$ which contains spanning trees of $\partial S_{e_i}$, $1 \leq i \leq k$. Then there is a bijection between $\sigma \in GF(T)$ such that $\supp(\sigma) = \plaqset$, and tuples $(\tilde{\sigma}, \sigma')$, such that $\tilde{\sigma}, \sigma' \in GF(T)$, $\sigma = \tilde{\sigma} \sigma'$, $\tilde{\sigma}_e \neq \groupid$ if and only if $e \in\{e_1, \ldots, e_k\}$, and $\supp(\sigma') = \plaqset'$.
\end{cor}

\begin{remark}
It is not clear whether there always exists a spanning tree $T$ of $\oneskel$ which contains spanning trees of $\partial S_{e_i}$, $1 \leq i \leq k$. Thus we will need to ensure that this is the case when we apply this corollary.
\end{remark}

We now seek to prove the analogue of Lemma \ref{lemma:activity_decomp} for more general sets of plaquettes. Given a rectangle $B$ contained in $\lbox$, let $S_2(B)$\label{notation:S-2-B} denote the 2-complex obtained by including all plaquettes of $B$. Let $\partial S_2(B)$\label{notation:partial-S-2-B} be the 2-complex obtained by including all plaquettes which are on the boundary of $B$, but not on the boundary of $\lbox$. Note if $B$ is contained in the interior of $\lbox$, then $\partial S_2(B)$ is simply the 2-complex made of all boundary plaquettes of $B$. Let $S_2^c(B)$\label{notation:S-2-c-B} be the 2-complex obtained by including all plaquettes of $\lbox$ that are not in $B$, as well as all plaquettes in $\partial S_2(B)$.


Given plaquette sets $\plaqset_1, \plaqset_2 \sse \plaquettes$, we say that $\plaqset_1, \plaqset_2$ are well separated, or $\plaqset_1$ is well separated from $\plaqset_2$, if there exists a rectangle $B$ in $\lbox$ such that $\plaqset_1 \sse S_2(B)$, $\plaqset_2 \sse S_2^c(B)$, and no plaquettes of $\plaqset_1$ or $\plaqset_2$ are contained in $\partial S_2(B)$. For such a $B$, we say that $\plaqset_1, \plaqset_2$ are well separated by $B$, or that $B$ well separates $\plaqset_1, \plaqset_2$. Note this definition is not symmetric in $\plaqset_1, \plaqset_2$.


The proof of the following topological fact is left to the appendix.

\begin{lemma}\label{lemma:rectangle-cell-complexes-simply-connected}
For a rectangle $B$ in $\lbox$, we have that $S_2(B)$ is simply connected. If in addition all side lengths of $B$ are strictly less than the side length of $\lbox$, then $\partial S_2(B), S_2^c(B)$ are simply connected.
\end{lemma}

We now show the analogue of Lemma \ref{lemma:gauge_eq_decomp} when $\plaqset_1, \plaqset_2$ are well separated.

\begin{lemma}\label{lemma:vortices-well-separated-gauge-fix-decomp}
Let $\plaqset_1, \plaqset_2 \sse \plaquettes$. Suppose $\plaqset_1, \plaqset_2$ are well separated by a rectangle $B$ in $\lbox$, such that all side lengths of $B$ are strictly less than the side length of $\lbox$. Let $T$ be a spanning tree of $\oneskel$ which contains spanning trees of $S_2(B), S_2^c(B)$, and $\partial S_2(B)$. There is a bijection between the set of $\sigma \in GF(T)$ such that $\supp(\sigma) = \plaqset_1 \cup \plaqset_2$, and the set of tuples $(\sigma^1, \sigma^2)$ such that $\sigma^i \in GF(T)$, $\supp(\sigma^i) = \plaqset_i$, $i = 1, 2$. Moreover, if $\sigma$ is mapped to $(\sigma^1, \sigma^2)$, then $\sigma = \sigma^1 \sigma^2$, $\sigma^1 = \groupid$ on $S_2^c(B)$, $\sigma^2 = \groupid$ on $S_2(B)$. Consequently, for all $p \in S_2(B)$, $\sigma^1_p = \sigma_p$, and for all $p \in S_2^c(B)$, $\sigma^2_p = \sigma_p$.
\end{lemma}
\begin{proof}
By Lemma \ref{lemma:rectangle-cell-complexes-simply-connected}, we have that $S_2(B), S_2^c(B)$ and $\partial S_2(B)$ are simply connected. Moreover, observe $\partial S_2(B) = S_2(B) \cap S_2^c(B)$. Now apply Lemma \ref{lemma:general-edge-configuration-decomposition}.
\end{proof}

Lemma \ref{lemma:vortices-well-separated-gauge-fix-decomp} now implies the following analogue of Lemma \ref{lemma:activity_decomp} for well separated plaquette sets $\plaqset_1, \plaqset_2$.

\begin{lemma}\label{lemma:well-separated-vortices-activity-decomp}
Let $\plaqset_1, \plaqset_2 \sse \plaquettes$ be well separated by a rectangle $B$ contained in $\lbox$, such that all side lengths of $B$ are strictly less than the side length of $\lbox$. Then 
\[\Phi(\plaqset_1 \cup \plaqset_2) = \Phi(\plaqset_1) \Phi(\plaqset_2). \]
\end{lemma}
\begin{proof}
Take a spanning tree $T$ as in the statement of Lemma \ref{lemma:vortices-well-separated-gauge-fix-decomp}. To see why such a spanning tree exists, first take a spanning tree $\tilde{T}$ of $\partial S_2(B)$
. Extend $\tilde{T}$ to spanning trees $T_1, T_2$ of $S_2(B), S_2^c(B)$, respectively. Then let $T$ be the union of $T_1, T_2$. Now apply Lemma \ref{lemma:vortices-well-separated-gauge-fix-decomp}, and proceed as in the proof of Lemma \ref{lemma:minimal-vortex-activity-factor-nonabelian}.
\end{proof}

Now given a plaquette set $\plaqset \sse \plaquettes$, we want to partition $\plaqset$ such that $\Phi(\plaqset)$ factors into a product. Moreover, we want this partition to be as fine as possible. Towards this end, first decompose $\plaqset = \vortex_1 \cup \cdots \cup \vortex_n$ into compatible vortices. We may assume that for some $0 \leq k \leq n$, $\vortex_1, \ldots, \vortex_k$ are minimal vortices, while $\vortex_{k+1}, \ldots, \vortex_n$ are not minimal vortices. Include in the partition $\vortex_1, \ldots, \vortex_k$. It then remains to partition $\plaqset' := \vortex_{k+1} \cup \cdots \cup \vortex_n$. Observe that given a partition of the index set $\{k+1, \ldots, n\} = I_1 \cup \cdots \cup I_m$, for $1 \leq j \leq m$ we may define
\[ K_j := \bigcup_{i \in I_j} \vortex_i. \]
Then $K_1, \ldots, K_m$ is a partition of $\plaqset'$. We now impose the following condition on $K_1, \ldots, K_m$. For each $1 \leq j \leq m-1$, $K_j$ is well separated from $K_{j+1} \cup \cdots \cup K_m$ by a cube in $\lbox$. Such a partition always exists; e.g. vacuously take $m=1$, $I_1 = \{k+1, \ldots, n\}$. Now take a ``maximal" partition $\plaqset' = K_1 \cup \cdots \cup K_m$, in the sense that for every $1 \leq j \leq m$, there does not exist a further partition $K_j = K_j^1 \cup K_j^2$ into nonempty components, such that $K_j^1, K_j^2$ are well separated by a cube in $\lbox$.

Note such a maximal partition $\plaqset' = K_1 \cup \cdots \cup K_m$ may or may not be unique; we will not assume uniqueness. For each $\plaqset'$ which does not have any minimal vortices in its vortex decomposition, fix a maximal partition $\plaqset' = K_1 \cup \cdots \cup K_m$. Note that for all $1 \leq j \leq m-1$, the cube which well separates $K_j$ from $K_{j+1} \cup \cdots \cup K_m$ must have side length strictly less than the side length of $\lbox$. Thus by applying Corollary \ref{cor:minimal-vortex-many-factor-non-abelian} once and repeatedly applying Lemma \ref{lemma:well-separated-vortices-activity-decomp}, we arrive at the following result.

\begin{lemma}
Suppose $\plaqset \sse \plaquettes$, and $\plaqset = \vortex_1 \cup \cdots \cup \vortex_k \cup K_1 \cup \cdots \cup K_m$ is partitioned as just described. Then
\[ \Phi(\plaqset) = \prod_{i=1}^k \Phi(\vortex_i) \prod_{j=1}^m \Phi(K_j).\]
\end{lemma}

Hereafter, the partition $\plaqset = \vortex_1 \cup \cdots \cup \vortex_k \cup K_1 \cdots \cup K_m$ which was just described will be referred to as the ``knot decomposition" of $\plaqset$. Let $\mc{K}$\label{notation:script-K} be the collection of all $K \sse \plaquettes$\label{notation:K} which appear in the knot decomposition of some $\plaqset \sse \plaquettes$. In particular, $\mc{K}$ contains all minimal vortices. Following \cite{SV1989}, the elements of $\mc{K}$ will be called knots.

\subsection{Wilson loop expectation in terms of random collections of knots}\label{section:wilson-loop-exp-non-abelian}

Recall $\rho$ is a unitary representation of $G$, $d$ is the dimension of $\rho$, and $\chi$ is the character of $\rho$. For a given vertex $x_0 \in \vertices$, and a spanning tree $T$ of $\oneskel$, recall the presentation $\pi_1(\oneskel, x_0) = \langle a_e, e \in \oneskel \backslash T \rangle$, where the $a_e$ are defined in terms of $(x_0, T)$. For the self avoiding loop $\wloop = e_1 \cdots e_\ell$, let $C_\wloop := a_{e_1} \cdots a_{e_\ell}$ (note if $e_i$ is such that $e_i \in T$, then define $a_{e_i} := \groupid$). For $\plaqset \sse \plaquettes$, define
\[ \Phi_\wloop(\plaqset) := \sum_{\substack{\homsym \in \Hom(\pi_1(\oneskel, x_0), G) \\ \supp(\homsym) = \plaqset}} \chi(\homsym(C_\wloop)) \prod_{p \in \plaqset} \varphi_\beta(\homsym(C_p)).\]

By the same proof as in Lemma \ref{lemma:choice-of-tree-invariant}, we have that $\Phi_\wloop$ does not depend on the choice of $(x_0, T)$. We also have the following analogue of Lemma \ref{lemma:phi-sum-gauge-fixed}, whose proof we omit.

\begin{lemma}\label{lemma:phi-wloop-gauge-fix}
For any spanning tree $T$ of $\oneskel$, and for any plaquette set $P \sse \plaquettes$, we have
\[ \Phi_\wloop(\plaqset) = \sum_{\substack{\sigma \in GF(T) \\ \supp(\sigma) = \plaqset}} \chi\bigg(\prod_{e \in \wloop} \sigma_e \bigg) \prod_{p \in \plaqset} \varphi_\beta(\sigma_p). \]
\end{lemma}

Recall $\mc{P}_\lbox$ is the set of nonempty vortices $\vortex \sse \plaquettes$. For $\Gamma \sse \mc{P}_\lbox$\label{notation:Gamma}, define 
\[\label{notation:P-Gamma} P(\Gamma) := \bigcup_{\vortex \in \Gamma} \vortex, \]
\[\label{notation:I-Gamma} I(\Gamma) := \ind(\text{elements of $\Gamma$ are compatible}). \]
In particular, we have $P(\varnothing) := \varnothing \sse \plaquettes$, and $I(\varnothing) := 1$. The following lemma expresses Wilson loop expectations in terms of a random subset of $\mc{P}_\lbox$, which can be thought of as a random collection of surfaces. First, let $N_1 := |G|^{|\vertices| - 1}$.

\begin{lemma}\label{lemma:wilson-loop-exp-in-terms-of-vortices}
We have
\[ Z_{\lbox, \beta} = N_1 \sum_{\Gamma \sse \mc{P}_\lbox} \Phi(P(\Gamma)) I(\Gamma), \]
and
\[ \langle W_\wloop \rangle_{\lbox, \beta} = \frac{\sum_{\Gamma \sse \plaquettes} \Phi_\wloop(P(\Gamma)) I(\Gamma)}{\sum_{\Gamma \sse \plaquettes} \Phi(P(\Gamma)) I(\Gamma)}. \]
\end{lemma}
\begin{proof}
Fix some vertex $x_0 \in \vertices$, and a spanning tree $T$ of $\oneskel$. Recalling Corollary \ref{cor:number-of-edge-configurations-map-to-hom}, we have
\begin{align*}
Z_{\lbox, \beta} &= \sum_{\sigma \in G^\edges} \prod_{p \in \plaquettes} \varphi_\beta(\sigma_p) \\
&= N_1 \sum_{\homsym \in \Hom(\pi_1(\oneskel, x_0), G)}  \prod_{p \in \plaquettes} \varphi_\beta(\homsym(C_p)) \\
&= N_1 \sum_{\plaqset \sse \plaquettes} \Phi(\plaqset) \\
&= N_1 \sum_{\Gamma \sse \mc{P}_\lbox} \Phi(P(\Gamma)) I(\Gamma).
\end{align*}
Note the last equality follows because any plaquette set $\plaqset \sse \plaquettes$ has a unique decomposition into compatible vortices. For the second assertion, start with
\[ \langle W_\wloop \rangle_{\lbox, \beta} = Z_{\lbox, \beta}^{-1} \sum_{\sigma \in G^\edges} \chi\bigg(\prod_{e \in \wloop} \sigma_e \bigg) \prod_{p \in \plaquettes} \varphi_\beta(\sigma_p), \]
and then use the same manipulations as before to show that the numerator is equal to
\[ N_1 \sum_{\Gamma \sse \mc{P}_\lbox} \Phi_\wloop(P(\Gamma)) I(\Gamma). \qedhere\]
\end{proof}

Now let $\randomv \sse \mc{P}_\lbox$\label{notation:randomv} be a random collection vortices, with law given by
\[ \p(\randomv = \Gamma) := \frac{\Phi(P(\Gamma)) I(\Gamma)}{\sum_{\Gamma' \sse \mc{P}_\lbox} \Phi(P(\Gamma')) I(\Gamma')}, ~~~ \Gamma \sse \mc{P}_\lbox.\]
(As an aside, one way to sample $\randomv$ would be to first sample the edge configuration $\Sigma \sim \mu_{\lbox, \beta}$, and then take the vortex decomposition of $\supp(\Sigma)$.) By Lemma \ref{lemma:wilson-loop-exp-in-terms-of-vortices}, we have
\[ \E \bigg[\frac{\Phi_\wloop(P(\randomv))}{\Phi(P(\randomv))}\bigg] = \langle W_\wloop \rangle_{\lbox, \beta}.\]
It thus suffices to study the behavior of $\Phi_\wloop(P(\randomv)) / \Phi(P(\randomv))$. The general strategy will be the same as in the Abelian case. We will define an event $E$, and a random variable $N_\wloop$, such that on the event $E$, we may write $\Phi_\wloop(P(\randomv)) / \Phi(P(\randomv)) = \mathrm{Tr}(A_\beta^{N_{\wloop}})$. We will be able to approximate $N_\wloop$ by a Poisson random variable, which will enable us to do an exact calculation to obtain a formula for the first order of $\langle W_\wloop \rangle_{\lbox, \beta}$. To obtain an error bound, we will bound $\p(E^c)$, as well as the error in approximating $N_\wloop$ by a Poisson.

Towards this end, let $P_\wloop$ denote the collection of all plaquettes contained in at least one $S_e, e \in \wloop$. I.e., $p \in P_\wloop$ if there exists $e \in \wloop$, and a 3-cell $c$ which contains both $e$ and $p$. For a given realization of $\randomv$, we may take the knot decomposition $P(\randomv) = \vortex_1 \cup \cdots \cup \vortex_k \cup K_1 \cup \cdots \cup K_m$. Now let $E$\label{notation:E-nonabelian} be the event that for all $1 \leq j \leq m$, there is a cube $B_j'$ in $\lbox$ which well separates $K_j$ from $P_\wloop$. This definition of $E$ is a bit more complicated than the corresponding definition in the Abelian setting, and this is again due to the presence of additional topological considerations. Though the main point is the same - $E$ is defined so that on this event, only minimal vortices can contribute. Let $N_\wloop$\label{notation:N-wloop-nonabelian} be the number of $e \in \wloop$ such that $P(e)$ is an element of $\randomv$. 
The following lemma expresses $\Phi_\wloop(P(\randomv)) / \Phi(P(\randomv))$ in terms of $N_\wloop$, on the event $E$. 
It is the analogue of Lemma \ref{lemma:wilson-loop-on-event-E}. First, recall the definition of $A_\beta$ in equation \eqref{eq:elemmatrix-def}.

\begin{lemma}\label{lemma:formula-for-wilson-loop-on-E}
On the event $E$, we have
\[ \frac{\Phi_\wloop(P(\randomv))}{\Phi(P(\randomv))} = \mathrm{Tr}(\elemmatrix^{N_\wloop}). \]
\end{lemma}

Before proving this lemma, we need some preliminary results.

\begin{lemma}\label{lemma:minimal-vortex-not-winding-around-loop-doesnt-contribute}
Suppose $\plaqset \sse \plaquettes$, such that for some edge $e$ not in the loop $\wloop$, $P(e)$ is a vortex of $\plaqset$. If $\Phi(\plaqset) > 0$, then
\[ \frac{\Phi_\wloop(\plaqset)}{\Phi(\plaqset)} = \frac{\Phi_\wloop(\plaqset \backslash P(e))}{\Phi(\plaqset \backslash P(e))}.\]
\end{lemma}
\begin{proof}
By Lemma \ref{lemma:minimal-vortex-activity-factor-nonabelian}, we have $\Phi(\plaqset) = \Phi(P(e)) \Phi(\plaqset \backslash P(e))$, so it suffices to show $\Phi_\wloop(\plaqset) = \Phi(P(e)) \Phi_\wloop(\plaqset \backslash P(e))$. Let $T$ be a spanning tree of $\oneskel$ which contains a spanning tree of $\partial S_e$. By Lemma \ref{lemma:phi-wloop-gauge-fix}, $\Phi_\wloop(\plaqset)$ is a sum over edge configurations $\sigma \in G^\edges$ such that $\supp(\sigma) = \plaqset$. By Lemma \ref{lemma:minimal-vortex-edge-config-decomp}, $\sigma$ corresponds uniquely to a tuple of edge configurations $(\sigma^1, \sigma^2)$, such that $\sigma = \sigma^1 \sigma^2$, $\sigma^1 = \groupid$ outside of $e$, $\supp(\sigma^1) = P(e)$, $\supp(\sigma^2) = \plaqset \backslash P(e)$. Observe as $e \notin \wloop$, we have
$\prod_{e \in \wloop} \sigma_e = \prod_{e \in \wloop} \sigma^2_e$.
The desired result now follows by factoring the sum over $\sigma$ into sums over $\sigma^1$, $\sigma^2$.
\end{proof}

\begin{lemma}\label{lemma:knot-well-separated-from-loop-doesnt-contribute}
Suppose $\plaqset = \plaqset' \cup K \sse \plaquettes$, such that $K$ is well separated from $\plaqset' \cup P_\wloop$ by a rectangle whose side lengths are all strictly less than the side length of $\lbox$. If $\Phi(\plaqset) > 0$, then
\[ \frac{\Phi_\wloop(\plaqset)}{\Phi(\plaqset)} = \frac{\Phi_\wloop(\plaqset')}{\Phi(\plaqset')}.\]
\end{lemma}
\begin{proof}
By Lemma \ref{lemma:well-separated-vortices-activity-decomp}, we have $\Phi(\plaqset) = \Phi(\plaqset') \Phi(K)$, so it remains to show $\Phi_\wloop(\plaqset) = \Phi_\wloop(\plaqset') \Phi(K)$. This follows by the same argument as in the proof of Lemma \ref{lemma:minimal-vortex-not-winding-around-loop-doesnt-contribute}, except we use Lemma \ref{lemma:vortices-well-separated-gauge-fix-decomp} in place of Lemma \ref{lemma:minimal-vortex-edge-config-decomp}.
\end{proof}


\begin{lemma}\label{lemma:only-minimal-vortices-winding-around-loop-computation}
Suppose we have $\plaqset = \vortex_1 \cup \cdots \cup \vortex_k \sse \plaquettes$, such that $\vortex_1, \ldots, \vortex_k$ are compatible minimal vortices, and for all $1 \leq i \leq k$, $\vortex_i = P(e_i)$, with $e_i \in \wloop$. Then
\[ \frac{\Phi_\wloop(\plaqset)}{\Phi(\plaqset)} = \mathrm{Tr}(\elemmatrix^k). \]
\end{lemma}
\begin{proof}
Take $x_0 \in \vertices$, and a spanning tree $T$ of $\oneskel$ which does not contain $e_1, \ldots, e_k$. 
By Corollary \ref{cor:only-minimal-vortices-edge-config}, we obtain that any homomorphism $\psi \in \Hom(\pi_1(\oneskel, x_0), G)$ such that $\supp(\psi) = \plaqset$ is of the form $\psi_T^{x_0}(\sigma)$, where $\sigma$ is of the form 
\[ \sigma_e = \begin{cases} g_i \neq 1 & e = e_i, 1 \leq i \leq k \\ 1 & e \notin \{e_1 \ldots, e_k\}\end{cases}.\]
We thus have
\begin{align*}
\Phi_\wloop(\plaqset) &= \sum_{\substack{g_i \neq \groupid \\ 1 \leq i \leq k}} \chi\bigg(\prod_{i=1}^k g_i \bigg) \prod_{i=1}^k \varphi_\beta(g_i)^6 \\
&= \Tr\Bigg(\sum_{\substack{g_i \neq \groupid \\ 1 \leq i \leq k}} \prod_{i=1}^k \rho(g_i) \varphi_\beta(g_i)^6  \Bigg) \\
&= \Tr\Bigg(\prod_{i=1}^k  \bigg(\sum_{g \neq 1} \rho(g) \varphi_\beta(g)^6 \bigg)\Bigg).
\end{align*}
In the first equality, we used the assumption that $\wloop$ is self avoiding, so that each edge appears only once. Also note that technically, for the edges $e_i$ which are negatively oriented in $\wloop$, the corresponding contribution should be $\rho(g_i^{-1})$ instead of $\rho(g_i)$. However, we can ignore this, since $\rho$ is a unitary representation, so that we have $\varphi_\beta(g_i) = \varphi_\beta(g_i^{-1})$. Now finish by dividing through by $\Phi(\plaqset) = \Phi(\vortex_1) \cdots \Phi(\vortex_k) = r_\beta^k$.
\end{proof}

\begin{proof}[Proof of Lemma \ref{lemma:formula-for-wilson-loop-on-E}]
Start with the knot decomposition 
\[ P(\randomv) = \vortex_1 \cup \cdots \cup \vortex_k \cup K_1 \cup \cdots \cup K_m. \]
By relabeling if necessary, we may suppose that there is some $0 \leq k' \leq k$ such that for all $1 \leq i \leq k'$, $\vortex_i = P(e_i)$, with $e_i \in \wloop$, and for all $k'+1 \leq j \leq k$, $\vortex_j = P(e_j), e_j \notin \wloop$. Note then $N_\wloop = k'$. Let $\plaqset' := \vortex_1 \cup \cdots \cup \vortex_{k'} \cup K_1 \cup \cdots \cup K_m$. By repeatedly applying Lemma \ref{lemma:minimal-vortex-not-winding-around-loop-doesnt-contribute}, we have
\[ \frac{\Phi_\wloop(P(\randomv))}{\Phi(P(\randomv))} = \frac{\Phi_\wloop(\plaqset')}{\Phi(\plaqset')}.\]
Let $B_1$ be a cube which well separates $K_1$ from $K_2 \cup \cdots \cup K_m$. On the event $E$, there is a cube $B_1'$ which well separates $K_1$ from $P_\wloop \supseteq \vortex_1 \cup \cdots \cup \vortex_{k'}$. Thus $B_1 \cap B_1'$ is a rectangle which well separates $K_1$ from $(\plaqset' \backslash K_1) \cup P_\wloop$, and moreover as all side lengths of $B_1'$ must be strictly less than the side length of $\lbox$, the same is true for $B_1 \cap B_1'$. Thus by Lemma \ref{lemma:knot-well-separated-from-loop-doesnt-contribute}, we have
\[ \frac{\Phi_\wloop(\plaqset')}{\Phi(\plaqset')} = \frac{\Phi_\wloop(\plaqset' \backslash K_1)}{\Phi(\plaqset' \backslash K_1)}. \]
We may repeat this argument to get rid of $K_2, \ldots, K_m$, and obtain
\[ \frac{\Phi_\wloop(\plaqset')}{\Phi(\plaqset')} = \frac{\Phi_\wloop(\vortex_1 \cup \cdots \cup \vortex_{k'})}{\Phi(\vortex_1 \cup \cdots \cup \vortex_{k'})}.\]
To finish, apply Lemma \ref{lemma:only-minimal-vortices-winding-around-loop-computation}.
\end{proof}

Now recall that in the Abelian case, to handle the case of long loops, we used the bound $|\E [ W_\wloop(\Sigma) ~|~ \plaqset(\Sigma)]| \leq |\elemmatrix|^{N_\wloop}$ (see the proof of Lemma \ref{lemma:large-loop-abelian}). We now set out to prove an analogue of this: Lemma \ref{lemma:wilson-loop-bound-general-non-abelian}. To be clear, this bound will always hold, so that we don't need to work with the event $E$ here. Thus let us forget about $E$ until after Lemma \ref{lemma:wilson-loop-bound-general-non-abelian}.

\begin{lemma}\label{lemma:se-intersection-bound}
For any $e \in \edges$, the number of $e'$ such that $S_e, S_{e'}$ share a vertex is at most $2^{17} - 1$.
\end{lemma}
\begin{proof}
The number of vertices in $S_e$ is at most $3 \cdot 2 \cdot 9 = 54$. For a given vertex $x_0$, in order for $e'$ to be such that $S_{e'}$ contains $x_0$, there must exist a 3-cell $c$ which contains both $e'$ and $x_0$. The number of 3-cells which contain $x_0$ is at most $8 \cdot 6 \cdot 4$, and the number of edges in a 3-cell is at most 12. To finish, observe $54 \cdot 8 \cdot 6 \cdot 4 \cdot 12 \leq 2^{17} - 1$.
\end{proof}

\begin{cor}\label{cor:getting-disjoint-s-e}
For any collection $\vortex_i = P(e_i) \sse \plaquettes$, $1 \leq i \leq k$, there exists a subset $I \sse [k]$, of size at least $\max(\floor{ 2^{-17} k}, 1)$, such that the 2-complexes $S_{e_i}, i \in I$ have disjoint vertex sets.
\end{cor}
\begin{proof}
Use Lemma \ref{lemma:se-intersection-bound}, along with the following result from graph theory. For a graph on $n$ vertices with maximum degree $d$, there exists an independent set of size at least $n / (d+1)$.
\end{proof}

\begin{lemma}\label{lemma:wilson-loop-bound-general-non-abelian}
Let $\plaqset \sse \plaquettes$, and suppose we have the knot decomposition $\plaqset = \vortex_1 \cup \cdots \cup \vortex_k \cup K_1 \cup \cdots \cup K_m$. Let $\tilde{k}$ be the number of $\vortex_i$, $1 \leq i \leq k$, such that $\vortex_i = P(e_i)$, with $e_i \in \wloop$. Then
\[ \abs{\frac{\Phi_\wloop(\plaqset)}{\Phi(\plaqset)}} \leq d \norm{\elemmatrix}_{op}^{\max(\floor{2^{-17}\tilde{k}}, 1)}. \]
\end{lemma}
\begin{proof}
We may assume that $\vortex_1, \ldots, \vortex_{\tilde{k}}$ are such that $\vortex_i = P(e_i)$, with $e_i \in \wloop$. By Corollary \ref{cor:getting-disjoint-s-e}, there exists a subset $I \sse [\tilde{k}]$ of size at least $\max(\floor{2^{-17}\tilde{k}}, 1)$, such that the 2-complexes $S_{e_i}, i \in I$ have disjoint vertex sets. Let $k' := |I|$, and by relabeling, assume $I = [k']$. Let 
\[ \plaqset' := \vortex_{k'+1} \cup \cdots \cup \vortex_k \cup K_1 \cup \cdots \cup K_m. \]
Now because the $S_{e_i}, 1 \leq i \leq k'$ have disjoint vertex sets, we may take a spanning tree $T$ of $\oneskel$ which contains spanning trees of $\partial S_{e_i}$, $1 \leq i \leq k'$. Let $EC$ be the set of $\sigma \in GF(T)$ such that $\sigma_e \neq \groupid$ if and only if $e \in \{e_1, \ldots, e_{k'}\}$. 
Applying Corollary \ref{cor:edge-config-bijection-tuple}, we obtain
\[ \Phi_\wloop(\plaqset) = \sum_{\substack{\sigma' \in GF(T) \\ \supp(\sigma') = \plaqset'}} \sum_{\tilde{\sigma} \in EC} \Tr(\prod_{e \in \wloop} \rho(\tilde{\sigma}_e \sigma'_e)) \prod_{p \in \plaqset'} \varphi_\beta(\sigma'_p)\prod_{\substack{p \in \vortex_i \\ 1 \leq i \leq k'}} \varphi_\beta(\tilde{\sigma}_p). \]
Let us now fix $\sigma'$. We proceed to ``average out" the sum over $\tilde{\sigma}$. Assume $e_1, \ldots, e_{k'}$ are ordered by appearance in $\wloop$. Observe then that there are matrices $A_1, \ldots, A_{k'+1}$ (depending on $\sigma'$ but not $\tilde{\sigma}$) such that for any $\tilde{\sigma} \in EC$, we have 
\[\prod_{e \in \wloop} \rho(\tilde{\sigma}_e \sigma'_e) = A_1 \rho(\tilde{\sigma}_{e_1})A_2 \rho(\tilde{\sigma}_{e_2}) \cdots A_{k'} \rho(\tilde{\sigma}_{e_{k'}}) A_{k'+1}. \] 
For instance, if $e_1$ is the $j$th edge of $\wloop$, then $A_1$ will be $\rho$ applied to the (ordered) product of $\sigma'_e$, as $e$ ranges over the first $j-1$ edges of $\wloop$. Now to sum in $\tilde{\sigma}$, it suffices to sum over the non-identity values of $\tilde{\sigma}$, i.e. $g_i \in G$, $g_i \neq \groupid$, $1 \leq i \leq k'$. We have
\[
\begin{split}
\sum_{i=1}^{k'} \sum_{g_i \neq \groupid} &A_1 \rho(g_1) \cdots A_{k'} \rho(g_{k'}) A_{k'+1} \prod_{i=1}^{k'} \varphi_\beta(g_i)^6 = \\ &A_1 \bigg(\sum_{g_1 \neq 1} \rho(g_1) \varphi_\beta(g_1)^6\bigg) \cdots A_{k'} \bigg(\sum_{g_{k'} \neq 1} \rho(g_{k'}) \varphi_\beta(g_{k'})^6\bigg) A_{k'+1}.
\end{split}
\]
As observed in the proof of Lemma \ref{lemma:only-minimal-vortices-winding-around-loop-computation}, if the edge $e_i$ appears in $\wloop$ in a negative orientation, then the contribution should be $\rho(g_i^{-1})$, not $\rho(g_i)$. However, once we sum over $g_i$, this difference gets washed out, since $\varphi_\beta(g_i) = \varphi_\beta(g_i^{-1})$. Now observe that dividing through by $\Phi(\plaqset) =  \Phi(\plaqset') \prod_{i=1}^{k'} \Phi(\vortex_i) = \Phi(\plaqset') r_\beta^{k'}$, we obtain
\[\Phi(\plaqset')^{-1} A_1 A_\beta \cdots A_{k'} A_\beta A_{k'+1}. \]
We may bound
\[ \Tr(A_1 A_\beta \cdots A_{k'} A_\beta A_{k'+1}) \leq d \norm{A_1 A_\beta \cdots A_{k'} A_\beta A_{k'+1}}_{op} \leq d \norm{A_\beta}_{op}^{k'}. \]
As this inequality holds for any fixed $\sigma'$, the same holds once we sum over $\sigma'$, and this yields the desired inequality.
\end{proof}

In the remainder of the subsection, suppose $\wloop$ is far enough away from the boundary of $\lbox$, such that any cube of side length 50 (say) which contains a vertex of $\wloop$ is completely contained in $\lbox$. Recall also the event $E$ and the random variable $N_\wloop$, which were defined just after the proof of Lemma \ref{lemma:wilson-loop-exp-in-terms-of-vortices}. We now state the key intermediate results about $E$ and $N_\wloop$, and then use these results to give a proof of Theorem \ref{thm:main-result}. These intermediate results are entirely analogous to those of Section \ref{section:wilson-loop-exp-abelian} (i.e. Lemma \ref{lemma:only-minimal-vortices-contribute-probability-bound} to Lemma \ref{lemma:large-loop-abelian}). Most proofs will be deferred to Sections \ref{section:proofs-non-abelian} and \ref{section:knot-upper-bound-non-abelian}. 

\begin{lemma}\label{lemma:only-minimal-vortices-contribute-probability-bound-non-abelian}
Suppose 
\beq\label{eq:beta-cond-only-minimal-vortices-contribute-probability-bound-non-abelian} \beta \geq \frac{1}{\Delta_G} (1000 +  14\log \abs{G}). \eeq
Then
\[ \p(E^c) \leq (\ell r_\beta) e^{-\beta \Delta_G / 2}.\]
\end{lemma}

\begin{cor}\label{cor:only-minimal-vortices-approximation-non-abelian}
Suppose 
\[ \beta \geq \frac{1}{\Delta_G} (1000 +  14\log \abs{G}). \]
Then
\[ |\langle W_\wloop \rangle_{\lbox, \beta} - \E \mathrm{Tr}(\elemmatrix^{N_\wloop})| \leq 2d(\ell r_\beta) e^{-\beta \Delta_G / 2}. \]
\end{cor}
\begin{proof}
Recall that $\langle W_\wloop \rangle_{\lbox, \beta} = \E [\Phi_\wloop (P(\randomv)) / \Phi(P(\randomv))]$. By Lemma \ref{lemma:formula-for-wilson-loop-on-E}, we have
\[ \E [\Phi_\wloop(\randomv) / \Phi(\randomv)] - \E \mathrm{Tr}(\elemmatrix^{N_\wloop}) = \E \bigg[\bigg(\frac{\Phi_\wloop(P(\randomv))}{\Phi(P(\randomv))} - \mathrm{Tr}(\elemmatrix^{N_\wloop})\bigg) \ind_{E^c}\bigg]. \]
Note $|\Phi_\wloop(P(\randomv)) / \Phi(P(\randomv))| \leq d$, and also $\norm{\elemmatrix}_{op} \leq 1$, so that for any integer $k \geq 0$, we have $|\mathrm{Tr}(\elemmatrix^k)| \leq d \|\elemmatrix^k\|_{op} \leq d$. Now finish by applying Lemma \ref{lemma:only-minimal-vortices-contribute-probability-bound-non-abelian}.
\end{proof}

\begin{prop}\label{prop:poisson-approximation-non-abelian}
Suppose
\beq\label{eq:beta-cond-poisson-approximation-non-abelian} \beta \geq \frac{1}{\Delta_G} (500 + 7 \log\abs{G}). \eeq
Let $\ms{L}(N_\wloop)$ be the law of $N_\wloop$. Then
\[ d_{TV}(\ms{L}(N_{\wloop}), \mathrm{Poisson}(\ell r_\beta)) \leq 300 (e^{1.5 \ell r_\beta}) r_\beta. \]
\end{prop}

\begin{prop}\label{prop:left-tail-bound-non-abelian}
Suppose 
\beq\label{eq:beta-cond-left-tail-bound-non-abelian} \beta \geq \frac{1}{\Delta_G} (500 + 7 \log \abs{G}).\eeq
Then we have the left tail bound
\[ \p\bigg(N_\wloop \leq \frac{1}{2} \ell r_\beta\bigg) \leq 2e \exp(- 0.15 \ell r_\beta). \]
\end{prop}

\begin{cor}\label{cor:wilson-loop-bound-long-loop}
Suppose 
\[ \beta \geq \frac{1}{\Delta_G} (500 + 7 \log \abs{G}).\]
Then
\[ |\langle W_\wloop \rangle_{\lbox, \beta}| \leq 2ed \exp(-0.15 \ell r_\beta) + d\exp(-(2^{-19} \log \norm{A_\beta}_{op}^{-1}) \ell r_\beta).\]
\end{cor}
\begin{proof}
Recall that $\langle W_\wloop \rangle_{\lbox, \beta} = \E [\Phi_\wloop (P(\randomv)) / \Phi(P(\randomv))]$.
By Lemma \ref{lemma:wilson-loop-bound-general-non-abelian} and Proposition \ref{prop:left-tail-bound-non-abelian}, we have
\begin{align*}
\bigg|\E \bigg[\frac{\Phi_\wloop(P(\randomv))}{\Phi(P(\randomv))}\bigg]\bigg| &\leq d \p\bigg(N_\wloop \leq \frac{1}{2} \ell r_\beta \bigg) + d \norm{A_\beta}^{\max(\floor{2^{-18} \ell r_\beta}, 1)}  \\
&\leq 2ed \exp(-0.15 \ell r_\beta) + d \norm{A_\beta}^{\max(\floor{2^{-18} \ell r_\beta}, 1)}.
\end{align*}
To finish, observe $\max(\floor{2^{-18} \ell r_\beta}, 1) \geq 2^{-19} \ell r_\beta$.
\end{proof}

\begin{proof}[Proof of Theorem \ref{thm:main-result}]
It suffices to show the bound \eqref{eq:main-result-wilson-loop-bound} with $\langle W_\wloop \rangle_\beta$ replaced by $\langle W_\wloop \rangle_{\lbox, \beta}$, with all $\lbox$ sufficiently large. In particular, we may assume that $\lbox$ is such that any cube of side length 50 which contains a vertex of $\wloop$ is completely contained in $\lbox$, so that we may apply the preceding results of this section.

First, by combining Corollary \ref{cor:only-minimal-vortices-approximation-non-abelian} and Proposition \ref{prop:poisson-approximation-non-abelian}, and using the fact $\norm{A_\beta}_{op} \leq 1$, we have
\[ |\langle W_\wloop\rangle_{\lbox, \beta} - e^{-\ell r_\beta} \mathrm{Tr}(e^{\ell r_\beta A_\beta})| \leq 2d (\ell r_\beta) e^{-\beta \Delta_G / 2} + 2d \cdot 300 e^{1.5 \ell r_\beta} r_\beta. \]
As $r_\beta \leq (|G| - 1) e^{-6 \beta \Delta_G}$, and $\beta$ is large enough so that $600 (|G| - 1) \leq e^{\beta \Delta_G}$, we further obtain
\beq\label{eq:short-loop} |\langle W_\wloop\rangle_{\lbox, \beta} - e^{-\ell r_\beta} \mathrm{Tr}(e^{\ell r_\beta A_\beta})| \leq 3 d e^{1.5 \ell r_\beta} e^{-\beta \Delta_G / 2}. \eeq
Letting $\lambda_1(\beta), \ldots,  \lambda_d(\beta)$ be the eigenvalues of $A_\beta$, observe
\[ e^{-\ell r_\beta} \mathrm{Tr}(e^{\ell r_\beta A_\beta}) = \sum_{i=1}^d e^{-\ell r_\beta (1 - \lambda_i(\beta))}. \]
For notational convenience, denote this quantity by $B_\beta$. Now by Corollary \ref{cor:wilson-loop-bound-long-loop} and the definition of $c_\beta$ (see equation \eqref{eq:main-result-c-beta-def}), we have
\beq\label{eq:long-loop} |\langle W_\wloop \rangle_{\lbox, \beta} - B_\beta| \leq |\langle W_\wloop \rangle_{\lbox, \beta}| + B_\beta \leq (2e + 2) d e^{-\ell r_\beta c_\beta}. \eeq
Upon combining equations \eqref{eq:short-loop} and \eqref{eq:long-loop}, we obtain
\[ |\langle W_\wloop \rangle_{\lbox, \beta} - B_\beta|^{1 + 1.5 / c_\beta} \leq 3d e^{1.5 \ell r_\beta} e^{-\beta \Delta_G / 2} ((2e + 2) d e^{-\ell r_\beta c_\beta})^{1.5 / c_\beta} .\]
Now finish by taking both sides to the power $c_\beta / (c_\beta + 1.5)$.
\end{proof}

\subsection{Proofs}\label{section:proofs-non-abelian}

For $\plaqset \sse \plaquettes$, let $\twoskelminus{\plaqset}$\label{notation:twoskelminus} denote the 2-complex obtained by including all edges and vertices of $\lbox$, and all plaquettes in $\plaquettes \backslash \plaqset$. For $x_0 \in \vertices$, $T$ a spanning tree of $\oneskel$, we have the presentation (see e.g. Section 4.1.3 of \cite{STILL1993})
\[ \pi_1(\twoskelminus{\plaqset}, x_0) = \langle a_e, e \in \oneskel \backslash T ~|~ C_p, p \in \plaquettes \backslash \plaqset\rangle  = \pi_1(\oneskel, x_0) / N_\plaqset, \]
where $N_\plaqset$ is the normal subgroup of $\pi_1(\oneskel, x_0)$ generated by $C_p, p \in \plaquettes \backslash \plaqset$. Let 
\[ \Pi_\plaqset : \pi_1(\oneskel, x_0) \ra \pi_1(\twoskelminus{\plaqset}, x_0) \]
be the natural projection map induced by $N_\plaqset$. 
For a homomorphism $\twoskelhom \in \Hom(\pi_1(\twoskelminus{\plaqset}, x_0), G)$, define $\supp(\twoskelhom) := \{p : \zeta(\Pi_\plaqset(C_p)) \neq \groupid\}$.

\begin{lemma}\label{lemma:homomorphisms-on-2-complex}
Fix $\plaqset \sse \plaquettes$. The quotient map $\Pi_\plaqset$ induces a bijection between the set of $\homsym \in \Hom(\pi_1(\oneskel, x_0), G)$ such that $\supp(\homsym) = \plaqset$, and the set of $\twoskelhom \in \Hom(\pi_1(\twoskelminus{\plaqset}, x_0), G)$ such that $\supp(\twoskelhom) = \plaqset$.
\end{lemma}
\begin{proof}
Observe that if $\homsym \in \Hom(\pi_1(\oneskel, x_0), G)$, $\supp(\homsym) = \plaqset$, then $N_\plaqset \sse \ker(\homsym)$. Then by the fundamental theorem on homomorphisms, there exists a unique $\twoskelhom$ such that $\homsym = \twoskelhom \circ \Pi_\plaqset$. Note this implies $\supp(\twoskelhom) = \supp(\homsym) = \plaqset$.

Conversely, suppose $\twoskelhom \in \Hom(\pi_1(\twoskelminus{\plaqset}, x_0), G)$, $\supp(\twoskelhom) = \plaqset$. Then $\twoskelhom \circ \Pi_\plaqset \in \Hom(\pi_1(\oneskel, x_0), G)$, and $\supp(\twoskelhom \circ \Pi_\plaqset) = \plaqset$.
\end{proof}

\begin{lemma}\label{lemma:knot-either-minimal-or-no-minimal-components}
Let $K \in \mc{K}$ be a knot. Consider the vortex decomposition $K = \vortex_1 \cup \cdots \cup \vortex_n$. Then either $n = 1$ and $\vortex_1$ is a minimal vortex, or for all $1 \leq i \leq n$, $\vortex_i$ is not a minimal vortex.
\end{lemma}
\begin{proof}
This follows by the definition of the knot decomposition, and the definition of a vortex.
\end{proof}

The following lemma is an analogue of Lemma \ref{lemma:smaller-than-minimal-vortices-prob-0}.

\begin{lemma}\label{lemma:knot-minimal-size}
Let $\plaqset \sse \plaquettes$. Suppose for all plaquettes $p \in \plaqset$, all 3-cells which contain $p$ are completely contained in $\lbox$. If $|\plaqset| \leq 5$, then $\Phi(\plaqset) = 0$. If $|\plaqset| = 6$, then $\Phi(\plaqset) > 0$ if and only if $\plaqset = P(e) \sse \plaquettes$.
\end{lemma}
\begin{proof}
We have that $\Phi(\plaqset)$ is a sum over homomorphisms $\homsym : \pi_1(\oneskel) \ra G$ such that $\supp(\homsym) = \plaqset$. By Lemma \ref{lemma:homomorphisms-on-2-complex}, the number of such $\homsym$ is equal to the number of homomorphisms $\twoskelhom : \pi_1(\twoskelminus{\plaqset}) \ra G$ such that $\supp(\twoskelhom) = \plaqset$. I now claim that if $|\plaqset| \leq 5$ and $\plaqset$ satisfies the stated assumptions, then $\pi_1(\twoskelminus{\plaqset}) = \{\groupid\}$, which implies that there does not exist $\twoskelhom$ such that $\supp(\twoskelhom) = \plaqset$. This would then imply that $\Phi(\plaqset) = 0$. Similarly, I claim that if $|\plaqset| = 6$, and $\plaqset$ satisfies the stated assumptions, then $\pi_1(\twoskelminus{\plaqset}) \neq \{\groupid\}$ if and only if $\plaqset = P(e) \sse \plaquettes$. As these two claims are purely topological, they are left to the appendix -- see Lemma \ref{lemma:plaquette-sets-too-small-means-simply-connected}.
\end{proof}

The following two lemmas are the analogues of Lemmas \ref{lemma:vortex-combinatorial-bound} and \ref{lemma:vortex-contained-in-cube}. The proofs are a bit more involved, so to not distract from the main thrust of the argument, they are left to Section \ref{section:knot-upper-bound-non-abelian}.

\begin{lemma}\label{lemma:knot-combinatorial-bound}
Let $p \in \plaquettes$. For any $m \geq 1$, the number of knots in $\mc{K}$ of size $m$ which contain $p$ is at most $(\knotbound)^m$.
\end{lemma}

\begin{lemma}\label{lemma:knot-contained-in-cube}
Let $m \geq 1$. For any knot $K \in \mc{K}$ of size $m$, there exists a cube $B$ in $\lbox$ of side length at most $3m$, such that all plaquettes of $K$ are contained in $S_2(B)$, but not in $\partial S_2(B)$.
\end{lemma}

The following lemma is the analogue of Lemma \ref{lemma:contributing-plaquette-bound}. First, recall the definition of $\plaqset_\wloop$ given just before Lemma \ref{lemma:formula-for-wilson-loop-on-E}.

\begin{lemma}\label{lemma:contributing-plaquette-bound-non-abelian}
For $m \geq 1$, let $S_\wloop^m$ be the set of plaquettes $p \in \plaquettes$ such that any cube of side length $m$ containing $p$ does not contain a plaquette of $\plaqset_\wloop$. Then
\[ |\plaquettes \backslash S_\wloop^m| \leq  96 \cdot 48 \ell (2m+1)^4.\]
\end{lemma}
\begin{proof}
Essentially the same proof as for Lemma \ref{lemma:contributing-plaquette-bound}, except we bound the number of vertices in some plaquette of $\plaqset_\wloop$ by $96 \ell$, since each such vertex must be in a 3-cell which contains an edge of $\wloop$, and any given edge has at most 12 3-cells which contain it, and any given 3-cell has 8 vertices. 
\end{proof}

We now give an upper bound on $\Phi$, which will be the analogue of Lemma \ref{lemma:activity-bound}. First, define
\[ \alpha_\beta \label{notation:alpha-beta-nonabelian} := |G| e^{-\beta \Delta_G}. \]

\begin{lemma}\label{lemma:activity-bound-non-abelian}
For any $\plaqset \sse \plaquettes$, we have
\[ \Phi(\plaqset) \leq \alpha_\beta^{|\plaqset|}.\]
\end{lemma}
\begin{proof}
For any homomorphism $\homsym \in \Hom(\pi_1(\oneskel), G)$ such that $\supp(\homsym) = \plaqset$, observe
\[ \prod_{p \in \plaqset} \varphi_\beta(\homsym(C_p)) \leq e^{-\beta \Delta_G |\plaqset|}. \]
To finish, we want to bound the number of such $\psi$ by $|G|^{|\plaqset|}$. By Lemma \ref{lemma:homomorphisms-on-2-complex}, it suffices to bound the number of homomorphisms from $\pi_1(\twoskelminus{\plaqset})$ to $G$ by $|G|^{|\plaqset|}$. To do so, it suffices to show that $\mathrm{rk}(\pi_1(\twoskelminus{\plaqset})) \leq |\plaqset|$, where $\mathrm{rk}$ denotes the rank of the group, i.e. the minimum size of a generating subset. As this is a purely topological fact, it is left to the appendix -- see Lemma \ref{lemma:rank-bound}.
\end{proof}

For $\mc{P} \sse \mc{P}_\lbox$, define
\[\label{notation:Xi-P-nonabelian} \Xi_\mc{P} := \sum_{\Gamma \sse \mc{P}} \Phi(P(\Gamma)) I(\Gamma).\]
This is the analogue to the quantity defined in \eqref{eq:partition-function-gas-of-particles-abelian}, the analysis of which was a crucial step in the argument for the Abelian case. Note that $Z_{\lbox, \beta} \propto \Xi_{\mc{P}_\lbox}$ (recall Lemma \ref{lemma:wilson-loop-exp-in-terms-of-vortices}).
For $\Gamma \sse \mc{P} \sse \mc{P}_\lbox$, define
\[\label{notation:rho-P-nonabelian} \rho_\mc{P}(\Gamma) := \frac{\Xi_{\mc{P} \backslash \Gamma}}{\Xi_{\mc{P}}}. \]
Observe we always have
\beq\label{eq:correlation-function-between-0-1} 0 \leq \rho_{\mc{P}}(\Gamma) \leq 1. \eeq
The next lemma shows that the probabilities of certain events involving $P(\randomv)$ may be expressed in terms of $\rho_{\mc{P}_\lbox}$. It is an analogue of Lemma \ref{lemma:vortex-appear-not-appear-probability}. First, recall that for vortices $\vortex_1, \ldots, \vortex_k$, the set $N_\lbox(\vortex_1, \ldots, \vortex_k)$ is defined as the collection of vortices in $\mc{P}_\lbox$ which are incompatible with at least one of $\vortex_1, \ldots, \vortex_k$.

\begin{lemma}\label{lemma:prob-as-correlation-function}
Let $\vortex_1, \ldots, \vortex_k \in \mc{P}_\lbox$ be compatible minimal vortices, and let $\vortex_1', \ldots, \vortex_m' \in \mc{P}_\lbox$ be vortices, not necessarily compatible nor minimal. Let $F$ be the event that $\vortex_i \in \randomv$ for all $1 \leq i \leq k$. Let $F'$ be the event that $\vortex_j'\notin \randomv$ for all $1 \leq j \leq m$. Let $\Gamma := N_\lbox(\vortex_1, \ldots, \vortex_k) \cup \{\vortex_1', \ldots, \vortex_m'\}$. We have
\[ \p(F \cap F') =  r_\beta^k \rho_{\mc{P}_\lbox}(\Gamma). \]
\end{lemma}
\begin{proof}
We have
\beq\label{eq:prob-as-correlation-function-intermediate-sum} \p(F \cap F') = \Xi_{\mc{P}_\lbox}^{-1} \sum_{\substack{\Gamma' \sse \mc{P}_\lbox \\ \vortex_1, \ldots, \vortex_k \in \Gamma' \\ \vortex_1', \ldots, \vortex_m' \notin \Gamma'}} \Phi(P(\Gamma')) I(\Gamma'). \eeq
Now for any $\Gamma' \sse \mc{P}_\lbox$ such that $I(\Gamma') = 1$, $\vortex_1, \ldots, \vortex_k \in \Gamma'$, and $\vortex_1', \ldots, \vortex_m' \notin \Gamma'$, we may partition $\Gamma' = \{\vortex_1, \ldots, \vortex_k\} \cup \tilde{\Gamma}$, where $I(\tilde{\Gamma}) = 1$ and $\tilde{\Gamma} \cap \Gamma = \varnothing$. Conversely, given any such $\tilde{\Gamma}$, we have that $I(\{\vortex_1, \ldots, \vortex_k\} \cup \tilde{\Gamma}) = 1$, since $\tilde{\Gamma} \cap N_\lbox(\vortex_1, \ldots, \vortex_k) = \varnothing$. Also, by Corollaries \ref{cor:minimal-vortex-many-factor-non-abelian} and \ref{cor:homomorphism-number-single-minimal-vortex}, we have
\[ \Phi(P(\Gamma')) = \bigg(\prod_{i=1}^k \Phi(\vortex_i)\bigg) \Phi(P(\tilde{\Gamma})) = r_\beta^k \Phi(P(\tilde{\Gamma})).\]
We may thus rewrite the right hand side of \eqref{eq:prob-as-correlation-function-intermediate-sum} as
\[ \Xi_{\mc{P}_\lbox}^{-1} r_\beta^k \sum_{\tilde{\Gamma} \sse \mc{P}_\lbox \backslash \Gamma} \Phi(P(\tilde{\Gamma})) I(\tilde{\Gamma}) = r_\beta^k \frac{\Xi_{\mc{P}_\lbox \backslash \Gamma}}{ \Xi_{\mc{P}_\lbox}}. \] 
The desired result now follows.
\end{proof}

Applying Lemma \ref{lemma:prob-as-correlation-function} with $k =1, m = 0$, and using \eqref{eq:correlation-function-between-0-1}, we obtain the following corollary.

\begin{cor}\label{cor:minimal-vortex-prob-bound-non-abelian}
Let $\vortex$ be a minimal vortex. Then
\[ \p(\vortex \in \randomv) \leq r_\beta. \]
\end{cor}

For $K \in \mc{K}$, let $F_K$ be the event that $K$ is present in the knot decomposition of $P(\randomv)$. The following lemma bounds the probability of $F_K$.

\begin{lemma}\label{lemma:knot-prob-bound}
For any $K \in \mc{K}$, we have
\[ \p(F_K) \leq \Phi(K). \]
\end{lemma}
\begin{proof}
Let $\Gamma_0$ be the set of vortices appearing in the vortex decomposition of $K$. For any $\Gamma \sse \mc{P}_\lbox$ such that $I(\Gamma) = 1$ and $K$ appears in the knot decomposition of $P(\Gamma)$, we may partition $\Gamma = \Gamma_0 \cup \Gamma'$, where $\Gamma' \sse \mc{P}_\lbox \backslash \Gamma_0$, $I(\Gamma') = 1$, and $\Phi(P(\Gamma)) = \Phi(K) \Phi(P(\Gamma'))$. We may thus upper bound $\p(F_K)$ by
\[ \Xi_{\mc{P}_\lbox}^{-1} \Phi(K) \sum_{\Gamma' \sse \mc{P}_\lbox \backslash \Gamma_0} \Phi(P(\Gamma')) I(\Gamma').\]
The sum in the above display is at most $\Xi_{\mc{P}_\lbox}$, and thus the desired result now follows.
\end{proof}

\begin{proof}[Proof of Lemma \ref{lemma:only-minimal-vortices-contribute-probability-bound-non-abelian}]
By Lemma \ref{lemma:knot-contained-in-cube}, on the event $E^c$, there must be a knot $K$ of some size $m \geq 1$ in the knot decomposition of $P(\randomv)$, which is not a minimal vortex, such that there is a cube of side length $3m$ which contains both $K$ and some plaquette of $P_\wloop$. Recalling the definition of $S_\wloop^{3m}$ from Lemma \ref{lemma:contributing-plaquette-bound-non-abelian}, we then have that $K$ contains a plaquette of $\plaquettes \backslash S_\wloop^{3m}$. Now if $m \leq 6$, then by the assumption that the loop $\wloop$ is far away from the boundary of $\lbox$, we also have that $K$ is far away from the boundary of $\lbox$. In particular, $K$ satisfies the  condition of Lemma \ref{lemma:knot-minimal-size}, and thus $\Phi(K) = 0$. By Lemma \ref{lemma:knot-prob-bound}, we have $\p(F_K) \leq \Phi(K)$, and thus we may assume that $m \geq 7$.
Now by a union bound, and then applying Lemmas \ref{lemma:knot-prob-bound}, \ref{lemma:knot-combinatorial-bound}, and \ref{lemma:contributing-plaquette-bound-non-abelian}, we obtain 
\begin{align*} 
\p(E^c) &\leq \sum_{m=7}^\infty \sum_{p \in \plaquettes \backslash S_\wloop^{3m}} \sum_{\substack{K \ni p \\ |K| = m}} \p(F_K) \\
&\leq 96 \cdot 48 \cdot 7^4 \ell \frac{(e^4 \knotbound \alpha_\beta)^7}{1 - e^4 \knotbound \alpha_\beta}.
\end{align*}
Now observe $\alpha_\beta^7 = |G|^7 e^{-7\beta \Delta_G}$, $e^{-6 \beta \Delta_G} \leq r_\beta$, and the assumption on $\beta$ implies
\[ 96 \cdot 48 \cdot 7^4 \frac{(e^4 \knotbound \abs{G})^7}{1 - e^4 \knotbound \alpha_\beta} e^{-\beta \Delta_G / 2} \leq 1. \]
The desired result now follows.
\end{proof}

We now begin to build towards the proofs of Propositions \ref{prop:poisson-approximation-non-abelian} and \ref{prop:left-tail-bound-non-abelian}. Recall that the main ingredient in the proofs (Section \ref{section:abelian-proofs}) of the Abelian analogues of these propositions (Propositions \ref{prop:poisson-approximation-abelian} and \ref{prop:left-tail-bound-abelian}) was that we were able to obtain good estimates on the Abelian analogue of $\rho_\mc{P}$ by using cluster expansion (Theorem \ref{thm:dobrushins-cluster-expansion}). The trouble now is we are not able to directly use cluster expansion, because we are not actually able to express $\Xi_\mc{P}$ in the form of equation \eqref{eq:partition-function-gas-of-particles-abelian}. This is due to the fact that there can be vortices $\vortex_1, \vortex_2, \vortex_3$ which are pairwise well separated, but their union $\vortex_1 \cup \vortex_2 \cup \vortex_3$ cannot be partitioned into two nonempty well separated pieces. This is analogous to the existence of Borromean rings in knot theory; see e.g. the Wikipedia page on Borromean rings. We thus have to resort to a more bare-hands approach, which we begin next.

\begin{lemma}\label{lemma:knot-contain-single-plaquette-sum-bound}
Suppose $\knotbound \alpha_\beta < 1$. Then for any $p \in \plaquettes$, we have
\[ \sum_{\substack{K \in \mc{K} \\ K \ni p \\ |K| \geq 7}} \Phi(K) \leq \frac{(\knotbound \alpha_\beta)^7}{1 - \knotbound \alpha_\beta} .\]
\end{lemma}
\begin{proof}
By Lemmas \ref{lemma:activity-bound-non-abelian} and \ref{lemma:knot-combinatorial-bound}, we have that the sum in question may be bounded by
\[\sum_{m=7}^\infty \sum_{\substack{K \in \mc{K} \\ K \ni p \\ |K| = m}} \Phi(K) \leq \sum_{m=7}^\infty (\knotbound \alpha_\beta)^m. \]
Now use the assumption $\knotbound \alpha_\beta < 1$ to sum the geometric series.
\end{proof}

The following two lemmas are crucial in allowing us to obtain good enough bounds on $\rho_{\mc{P}}$.

\begin{lemma}\label{lemma:no-minimal-vortex-correlation-function-lower-bd}
Suppose $\beta$ is large enough so that $\knotbound \alpha_\beta < 1$.
Suppose $\vortex  \in \mc{P} \sse \mc{P}_\lbox$, $\vortex$ is a minimal vortex, and $\vortex$ is far enough away from the boundary of $\lbox$, such that any cube of side length 25 (say) which contains a plaquette of $\vortex$ is completely contained in $\lbox$. Suppose $\Gamma \sse N_\lbox(\vortex) \cap \mc{P}$ is such that $\Gamma$ does not contain any minimal vortices. Then
\[ 1 - \rho_{\mc{P}}(\Gamma) \leq \betathreshnonab. \]
\end{lemma}
\begin{proof}
We have
\[ 
1 - \rho_{\mc{P}}(\Gamma) = \frac{1}{\Xi_{\mc{P}}} \sum_{\substack{\Gamma' \sse \mc{P} \\ \Gamma' \cap \Gamma \neq \varnothing}} \Phi(P(\Gamma')) I(\Gamma'). \]
Given $\Gamma' \sse \mc{P}$ such that $\Gamma' \cap \Gamma \neq \varnothing$, $I(\Gamma') = 1$, consider the knot decomposition of $P(\Gamma')$. Using this decomposition, we may partition $P(\Gamma')$ into two parts: a knot which is incompatible with $\vortex$, and everything else. This gives a partition $\Gamma' = \tilde{\Gamma} \cup \Gamma''$ such that $P(\tilde{\Gamma}) \in \mc{K}$, $\tilde{\Gamma} \cap \Gamma \neq \varnothing$, and $I(\tilde{\Gamma}) = I(\Gamma'') = 1$. Moreover, we have $\Phi(P(\Gamma)) = \Phi(P(\tilde{\Gamma})) \Phi(P(\Gamma''))$. Thus we may bound
\begin{align*}
\frac{1}{\Xi_{\mc{P}}}\sum_{\substack{\Gamma' \sse \mc{P} \\ \Gamma' \cap \Gamma \neq \varnothing}} \Phi(P(\Gamma')) I(\Gamma') &\leq \frac{1}{\Xi_{\mc{P}}} \sum_{\substack{\tilde{\Gamma} \sse \mc{P} \\ P(\tilde{\Gamma}) \in \mc{K} \\ \tilde{\Gamma} \cap \Gamma \neq \varnothing}} \sum_{\Gamma'' \sse \mc{P}} \Phi(P(\tilde{\Gamma}))I(\tilde{\Gamma}) \Phi(P(\Gamma''))  I(\Gamma'')  \\
&\leq  \sum_{\substack{\tilde{\Gamma} \sse \mc{P} \\ P(\tilde{\Gamma}) \in \mc{K} \\ \tilde{\Gamma} \cap \Gamma \neq \varnothing}} \Phi(P(\tilde{\Gamma})) I(\tilde{\Gamma}).
\end{align*}
Now suppose $\tilde{\Gamma}$ is such that $\tilde{\Gamma} \cap \Gamma \neq \varnothing$ and $P(\tilde{\Gamma}) \in \mc{K}$. By assumption, $\tilde{\Gamma}$ must contain a vortex which is not a minimal vortex. By Lemma \ref{lemma:knot-either-minimal-or-no-minimal-components}, this implies that $P(\tilde{\Gamma})$ cannot be a minimal vortex. Moreover, we have $P(\tilde{\Gamma}) \nsim \vortex$, and by Lemma \ref{lemma:knot-contained-in-cube}, there is a cube of side length $3 |P(\tilde{\Gamma})|$ which contains $P(\tilde{\Gamma})$. If $|P(\tilde{\Gamma})| \leq 6$, then the assumption that $\vortex$ is far away from the boundary of $\lbox$ also implies that $P(\tilde{\Gamma})$ is far enough away from the boundary of $\lbox$, so that $P(\tilde{\Gamma})$ satisfies the condition of Lemma \ref{lemma:knot-minimal-size}, and so $\Phi(P(\tilde{\Gamma})) = 0$.
Thus we may further upper bound
\[  \sum_{\substack{\tilde{\Gamma} \sse \mc{P} \\ P(\tilde{\Gamma}) \in \mc{K} \\ \tilde{\Gamma} \cap \Gamma \neq \varnothing}} \Phi(P(\tilde{\Gamma})) I(\tilde{\Gamma}) \leq \sum_{\substack{K \in \mc{K} \\ K \nsim \vortex \\ |K| \geq 7}} \Phi(K). \]
Now in order for $K$ to be incompatible with $\vortex$, we must have that $K$ contains a plaquette which is in the same 3-cell as a plaquette of $\vortex$. The number of possible plaquettes is at most $21 |\vortex|$. To finish, note by assumption $|\vortex| = 6$, and apply Lemma \ref{lemma:knot-contain-single-plaquette-sum-bound}.
\end{proof}

\begin{lemma}\label{lemma:vortex-removal-identity}
Suppose $\Gamma \sse \mc{P} \sse \mc{P}_\lbox$, and $\Gamma$ contains a minimal vortex $\vortex$. Then
\[ \rho_{\mc{P}}(\Gamma) = \rho_{\mc{P}}(\Gamma \backslash \{\vortex\}) - r_\beta \rho_{\mc{P}}(\Gamma \cup (N_\lbox(\vortex) \cap \mc{P})).\]
\end{lemma}
\begin{proof}
Observe $\mc{P} \backslash (\Gamma \backslash \{\vortex\}) = (\mc{P} \backslash \Gamma) \cup \{\vortex\}$. Thus
\[ \Xi_{\mc{P} \backslash(\Gamma \backslash \{\vortex\})} = \sum_{\Gamma' \sse \mc{P} \backslash \Gamma} \Phi(P(\Gamma')) I(\Gamma') + \sum_{\substack{\Gamma' \sse (\mc{P} \backslash \Gamma) \cup \{\vortex\} \\ \Gamma' \ni \vortex}} \Phi(P(\Gamma')) I(\Gamma'). \]
Now by Lemma \ref{lemma:minimal-vortex-activity-factor-nonabelian}, for any $\Gamma'$ such that $I(\Gamma') = 1$, $\Gamma' \ni \vortex$, we have $\Phi(P(\Gamma')) = \Phi(\vortex) \Phi(P(\Gamma' \backslash \{\vortex\}))$. Moreover, since $I(\Gamma') = 1$, we have that every element of $\Gamma' \backslash \{\vortex\}$ must be  compatible with $\vortex$, i.e. $\Gamma' \backslash \{\vortex\} \sse \mc{P} \backslash (\Gamma \cup N_\lbox(\vortex))$. Conversely, for any $\tilde{\Gamma} \sse \mc{P} \backslash (\Gamma \cup N_\lbox(\vortex))$ such that $I(\tilde{\Gamma}) = 1$, we have $I(\tilde{\Gamma} \cup \{\vortex\}) = 1$, and by Lemma \ref{lemma:minimal-vortex-activity-factor-nonabelian}, we also have $\Phi(P(\tilde{\Gamma} \cup \{\vortex\})) = \Phi(\vortex) \Phi(P(\tilde{\Gamma}))$. Thus
\begin{align*}
\Xi_{\mc{P} \backslash(\Gamma \backslash \{\vortex\})}
&= \sum_{\Gamma' \sse \mc{P} \backslash \Gamma} \Phi(P(\Gamma')) I(\Gamma') + \Phi(\vortex) \sum_{\Gamma' \sse \mc{P} \backslash(\Gamma \cup N_\lbox(\vortex))} \Phi(P(\Gamma')) I(\Gamma') \\
&= \Xi_{\mc{P} \backslash \Gamma} + r_\beta \Xi_{\mc{P} \backslash (\Gamma \cup (N_\lbox(\vortex) \cap \mc{P}))}.
\end{align*}
Now finish by dividing through by $\Xi_{\mc{P}}$, and rearranging.
\end{proof}


We now apply Lemmas \ref{lemma:no-minimal-vortex-correlation-function-lower-bd} and \ref{lemma:vortex-removal-identity} to obtain lower (Lemma \ref{lemma:minimal-vortex-correlation-function-lower-bound}) and upper (Lemma \ref{lemma:correlation-function-upper-bound}) bounds on $\rho_{\mc{P}}$.

\begin{lemma}\label{lemma:minimal-vortex-correlation-function-lower-bound}
Suppose $\beta$ is large enough so that $\knotbound \alpha_\beta < 1$, and also
\[ \betathreshnonab \leq r_\beta. \]
Suppose $\vortex$ is a minimal vortex, $\vortex \in \mc{P} \sse \mc{P}_\lbox$, and $\vortex$ is far enough away from the boundary of $\lbox$, such that any cube of side length 25 (say) which contains a plaquette of $\vortex$ is completely contained in $\lbox$. Let $\Gamma \sse N_\lbox(\vortex) \cap \mc{P}$, and let $k$ be the number of minimal vortices in $\Gamma$. Then
\[  \rho_{\mc{P}}(\Gamma) \geq  1 - (k+1) r_\beta . \]
\end{lemma}
\begin{proof}
Let $\Gamma = \{\vortex_1, \ldots, \vortex_n\}$, with $\vortex_1, \ldots, \vortex_k$ minimal, and $\vortex_{k+1}, \ldots, \vortex_n$ not minimal. By Lemma \ref{lemma:no-minimal-vortex-correlation-function-lower-bd}, it suffices to assume $k \geq 1$, i.e. $\Gamma$ contains a minimal vortex. By Lemma \ref{lemma:vortex-removal-identity}, we have
\begin{align*}
\rho_{\mc{P}}(\Gamma) &= \rho_{\mc{P}}(\Gamma \backslash \{\vortex_1\}) - r_\beta \rho_{\mc{P}} (\Gamma \cup (N_\lbox(\vortex_1) \cap \mc{P})).
\intertext{Now observe that $\rho_{\mc{P}} (\Gamma \cup (N_\lbox(\vortex_1) \cap \mc{P})) \leq 1$, and thus}
&\geq \rho_{\mc{P}}(\Gamma \backslash \{\vortex_1\}) - r_\beta .
\end{align*}
Iterating this, we obtain
\[ \rho_{\mc{P}}(\Gamma) \geq \rho_{\mc{P}}(\Gamma \backslash \{\vortex_1, \ldots, \vortex_k\}) - k r_\beta. \]
Now finish by Lemma \ref{lemma:no-minimal-vortex-correlation-function-lower-bd}.
\end{proof}

\begin{lemma}\label{lemma:correlation-function-upper-bound}
Suppose $\beta$ is large enough so that $\knotbound \alpha_\beta < 1$, and also
\[ \betathreshnonab \leq r_\beta. \]
Suppose $\vortex$ is a minimal vortex, $\vortex \in \mc{P} \sse \mc{P}_\lbox$, and $\vortex$ is far enough away from the boundary of $\lbox$, such that for any minimal vortex $\vortex'$ which is incompatible with $\vortex$, any cube of side length 25 (say) which contains a plaquette of $\vortex'$ is completely contained in $\lbox$. Let $\Gamma \sse N_\lbox(\vortex) \cap \mc{P}$, and let $k$ be the number of minimal vortices in $\Gamma$. Then
\[  \rho_{\mc{P}}(\Gamma) \leq 1 - k r_\beta (1 - 290 r_\beta). \]
\end{lemma}
\begin{proof}
Let $\Gamma = \{\vortex_1, \ldots, \vortex_n\}$, with $\vortex_1, \ldots, \vortex_k$ minimal, and $\vortex_{k+1}, \ldots, \vortex_n$ not minimal. It suffices to assume $k \geq 1$, i.e. $\Gamma$ contains a minimal vortex. By Lemma \ref{lemma:vortex-removal-identity}, we have
\begin{align*}
\rho_{\mc{P}}(\Gamma) &= \rho_{\mc{P}}(\Gamma \backslash \{\vortex_1\}) - r_\beta \rho_{\mc{P}} (\Gamma \cup (N_\lbox(\vortex_1) \cap \mc{P})).
\end{align*}
Let $\Gamma_1  := N_\lbox(\vortex_1) \cap \mc{P}$. To lower bound $\rho_{\mc{P}} (\Gamma \cup \Gamma_1)$, observe
\[ 1 - \rho_{\mc{P}} (\Gamma \cup \Gamma_1) \leq (1 - \rho_{\mc{P}}(\Gamma)) + (1 - \rho_{\mc{P}}(\Gamma_1)), \]
so that
\[ \rho_{\mc{P}} (\Gamma \cup \Gamma_1) \geq \rho_{\mc{P}}(\Gamma) + \rho_{\mc{P}}(\Gamma_1) - 1.\]
Note by assumption that both $\vortex, \vortex_1$ are far enough away from the boundary of $\lbox$ to satisfy the conditions of Lemma \ref{lemma:minimal-vortex-correlation-function-lower-bound}. We may thus apply Lemma \ref{lemma:minimal-vortex-correlation-function-lower-bound} to both $\rho_{\mc{P}}(\Gamma)$ and $\rho_{\mc{P}}(\Gamma_1)$, along with Lemma \ref{lemma:minimal-vortex-incompatible-bound}, to obtain
\[ \rho_{\mc{P}}(\Gamma \cup \Gamma_1) \geq 1 - 290 r_\beta.\]
We thus obtain (note the assumption on $\beta$ implies $290 r_\beta \leq 1$)
\[ \rho_{\mc{P}}(\Gamma) \leq \rho_{\mc{P}}(\Gamma \backslash \vortex_1) - r_\beta(1 - 290 r_\beta). \] 
Upon iterating, we obtain
\[ \rho_{\mc{P}}(\Gamma) \leq \rho_{\mc{P}}(\Gamma \backslash \{\vortex_1, \ldots, \vortex_k\}) - k r_\beta(1 - 290 r_\beta). \]
Now finish by noting $\rho_{\mc{P}}(\Gamma \backslash \{\vortex_1, \ldots, \vortex_k\}) \leq 1$.
\end{proof}

In the remainder of the subsection, we will abuse notation and think of the loop $\wloop$ as a set of edges, ignoring orientation. We now have enough to prove Propositions \ref{prop:poisson-approximation-non-abelian} and \ref{prop:left-tail-bound-non-abelian}.

\begin{proof}[Proof of Proposition \ref{prop:poisson-approximation-non-abelian}]
The assumption \eqref{eq:beta-cond-poisson-approximation-non-abelian} on $\beta$ implies that the conditions on $\alpha_\beta, r_\beta$ in Lemma \ref{lemma:minimal-vortex-correlation-function-lower-bound} are satisfied. Also, the assumption that the loop $\wloop$ is far away from the boundary of $\lbox$ implies $P(e) \sse \plaquettes$ for any $e \in \wloop$, and moreover, $P(e)$ is far enough away from the boundary of $\lbox$ to satisfy the conditions of Lemma \ref{lemma:minimal-vortex-correlation-function-lower-bound}. We may thus apply Lemma \ref{lemma:minimal-vortex-correlation-function-lower-bound} with $\vortex = P(e)$, for any $e \in \wloop$.

The proof is essentially the same as in the Abelian case. For $e \in \wloop$, let $F_{P(e)}$ be the event that $P(e) \in \randomv$. Then $N_\wloop = \sum_{e \in \wloop} \ind_{F_{P(e)}}$. Let $\lambda := \E N_\wloop$. For $e \in \wloop$, define $B_e := \{e' \in \wloop : P(e) \nsim P(e')\}$. We will apply Theorem \ref{thm:chen-rollin-poisson-approximation}, so let $b_1, b_2, b_3$ be as in the statement of the theorem. To bound $b_1$, observe by Lemma \ref{lemma:minimal-vortex-incompatible-bound}, and the assumption that $\wloop$ is self avoiding, that $|B_e| \leq 144$ for all $e \in \wloop$. Then use Corollary \ref{cor:minimal-vortex-prob-bound-non-abelian} to obtain $b_1 \leq 144 \lambda r_\beta$. Note that $b_2 = 0$.

It remains to bound $b_3$. To do so, for $e_0 \in \wloop$, and $\wloop' \sse \wloop \backslash B_{e_0}$, let $F_{e_0, \wloop'}$ be defined as in equation \eqref{eq:f-e0-wloop}. We want to bound the quantity
\[  |\E [\ind_{F_{P(e_0)}} ~|~ F_{e_0,\wloop'}] - \p(F_{P(e_0)})|, \]
when $\p(F_{e_0, \wloop'}) > 0$. Define
\[\Gamma := N_\lbox(P(e), e \in \wloop') \cup \{P(e') : e' \in \wloop \backslash B_{e_0}, e' \notin \wloop' \}.\]
By Lemma \ref{lemma:prob-as-correlation-function}, we have
\[ \p(F_{P(e_0)}) = r_\beta \rho_{\mc{P}_\lbox}(N_\lbox(P(e_0))), \]
\[ \p(F_{e_0, \wloop'}) = r_\beta^{|\wloop'|} \rho_{\mc{P}_\lbox}(\Gamma) , \]
\[ \p(F_{P(e_0)} \cap F_{e_0, \wloop'}) = r_\beta^{|\wloop'| + 1} \rho_{\mc{P}_\lbox}(\Gamma \cup N_\lbox(P(e_0))). \]
Thus
\[ \E [\ind_{F_{P(e_0)}} ~|~ F_{e_0,\wloop'}] - \p(F_{P(e_0)}) = \p(F_{P(e_0)}) (R_1 / R_2 - 1), \]
with
\[ R_1 = \rho_{\mc{P}_\lbox \backslash \Gamma} (N_\lbox(P(e_0)) \cap (\mc{P}_\lbox \backslash \Gamma)), \]
\[ R_2 = \rho_{\mc{P}_\lbox}(N_\lbox(P(e_0))) .\]
Now apply Lemmas \ref{lemma:minimal-vortex-correlation-function-lower-bound} and \ref{lemma:minimal-vortex-incompatible-bound} to obtain
\[ 0 \leq 1 - R_1 \leq 145 r_\beta, \]
\[ 0 \leq 1 - R_2 \leq 145 r_\beta. \]
The assumption \eqref{eq:beta-cond-poisson-approximation-non-abelian} on $\beta$ implies $145 r_\beta$ is small enough for us to bound
\[ \abs{R_1 / R_2 - 1} \leq 150 r_\beta, \]
and thus
\[ b_3 \leq 150 \lambda r_\beta. \]
Thus by Theorem \ref{thm:chen-rollin-poisson-approximation}, we obtain
\[ d_{TV}(\ms{L}(N_\wloop), \mathrm{Poisson}(\lambda)) \leq 144 r_\beta + 150 \min (\lambda, 1.4 \sqrt{\lambda}) r_\beta. \]
We proceed to bound (recall by Corollary \ref{cor:minimal-vortex-prob-bound-non-abelian} that $\p(F_{P(e)}) \leq r_\beta$ for all $e \in \wloop$)
\[
\abs{\ell r_\beta - \lambda} = \sum_{e \in \wloop} (r_\beta - \p(F_{P(e)})). 
\]
Note $\p(F_{P(e)}) = r_\beta \rho_{\mc{P}_\lbox}(N_\lbox(P(e)))$. Combining this with Lemmas \ref{lemma:minimal-vortex-correlation-function-lower-bound} and \ref{lemma:minimal-vortex-incompatible-bound}, we may upper bound the above by
$ 145 (\ell r_\beta) r_\beta$.
Applying Corollary 1 of \cite{AL2005}, we thus obtain
\[ d_{TV}(\ms{L}(N_\wloop), \mathrm{Poisson}(\ell r_\beta)) \leq 144 r_\beta + 150 \min (\lambda, 1.4 \sqrt{\lambda}) r_\beta + 145 (\ell r_\beta) r_\beta. \]
To finish, proceed as in the proof of Proposition \ref{prop:poisson-approximation-abelian}.
\end{proof}

\begin{proof}[Proof of Proposition \ref{prop:left-tail-bound-non-abelian}]
The assumption \eqref{eq:beta-cond-left-tail-bound-non-abelian} on $\beta$ implies that the conditions on $\alpha_\beta, r_\beta$ of Lemma \ref{lemma:correlation-function-upper-bound} are satisfied. Also, the assumption that the loop $\wloop$ is far away from the boundary of $\lbox$ implies $P(e) \sse \plaquettes$ for any $e \in \wloop$, and moreover, $P(e)$ is far enough away from the boundary of $\lbox$ to satisfy the conditions of Lemma \ref{lemma:correlation-function-upper-bound}. We may thus apply Lemma \ref{lemma:correlation-function-upper-bound} with $\vortex = P(e)$, for any $e \in \wloop$.

For $e \in \wloop$, let $X_e := \ind_{F_{P(e)}}$. From the proof of Proposition \ref{prop:left-tail-bound-abelian}, we see that it suffices to show that for any $0 \leq k \leq \frac{1}{2} \ell r_\beta$, and any $\wloop' \sse \wloop$, $\abs{\wloop'} = k$, we have
\[ \p(X_e = 1, e \in \wloop', X_{e'} = 0, e' \in \wloop \backslash \wloop') \leq r_\beta^k \exp(-0.997\ell r_\beta). \]
It suffices to assume that the vortices $P(e), e \in \wloop'$ are compatible. Let
\[ \Gamma := N_\lbox(P(e), e \in \wloop') \cup \{P(e'): e' \in \wloop \backslash \wloop'\}. \]
By Lemma \ref{lemma:prob-as-correlation-function}, we have
\[ \p(X_e = 1, e \in \wloop', X_{e'} = 0, e' \in \wloop \backslash \wloop') = r_\beta^k \thinspace \rho_{\mc{P}_\lbox}(\Gamma). \]
Number the edges of $\wloop$ by $e_1, \ldots, e_n$, such that $\wloop' = \{e_1, \ldots, e_k\}$. For $1 \leq j \leq k$, define $\Gamma_j := N_\lbox(P(e_i), 1 \leq i \leq j)$, and for $k+1 \leq j \leq \ell$, define $\Gamma_j := \Gamma_k \cup \{P(e_i) : k+1 \leq i \leq j\}$. Define $\Gamma_0 = \varnothing$. We then have
\[\rho_{\mc{P}_\lbox}(\Gamma) = \prod_{j=1}^k \rho_{\mc{P}_\lbox \backslash \Gamma_{j-1}}(\Gamma_j \backslash \Gamma_{j-1}). \]
Let $k_j$ be the number of minimal vortices in $\Gamma_j \backslash \Gamma_{j-1}$. By Lemma \ref{lemma:correlation-function-upper-bound}, we obtain 
\begin{align*}
\rho_{\mc{P}_\lbox \backslash \Gamma_{j-1}}(\Gamma_j \backslash \Gamma_{j-1}) &\leq 1 - k_j r_\beta(1 - 290 r_\beta).
\intertext{The assumption \eqref{eq:beta-cond-left-tail-bound-non-abelian} on $\beta$ implies $290 r_\beta \leq 0.003$. Thus we further have}
\rho_{\mc{P}_\lbox \backslash \Gamma_{j-1}}(\Gamma_j \backslash \Gamma_{j-1}) &\leq \exp(-0.997 k_j r_\beta).
\end{align*}
We thus obtain
\[ \rho_{\mc{P}_\lbox}(\Gamma) \leq \exp(-0.997\sum_{j=1}^k k_j r_\beta). \] 
Now finish by observing that $\sum_{j=1}^k k_j \geq \ell$, since $\Gamma$ at least contains $P(e), e \in \wloop$ (note here we have assumed that $\wloop$ is self avoiding).
\end{proof}

\subsection{Upper bounding the number of knots}\label{section:knot-upper-bound-non-abelian}

Given a plaquette set $P \sse \plaquettes$, define $B(P)$ as a cube of minimal side length in $\lbox$ such that all plaquettes of $P$ are in $S_2(B(P))$, but not in $\partial S_2(B(P))$. If the choice of $B(P)$ is not unique, fix one such cube. For $P, P' \sse \plaquettes$, define the function
\[ J(P, P') := \begin{cases} 1 & P \cap B(P') \neq \varnothing \text{ or } P' \cap B(P) \neq \varnothing \\ 0 & \text{otherwise} \end{cases} .\]
To be clear, we are slightly abusing notation here by writing $P \cap B(P') \neq \varnothing$; what this means is that there is a plaquette $p \in P$ which is contained in $B(P')$. We now inductively define a hierarchy of undirected graphs $G^\scale(P)$, for integers $\scale \geq 0$. First, to define $G^0(P)$, consider the vortex decomposition $P = \vortex_1 \cup \cdots \cup \vortex_{n_0}$. Define $P^0_i := V_i$, $1 \leq i \leq n_0$. The vertex set of $G^0(P)$ is $\{P^0_1, \ldots, P^0_{n_0}\}$. The edge set is 
\[ \{ \{P^0_i, P^0_j\} : i \neq j, J(P^0_i, P^0_j) = 1\}. \]
Now suppose for some $\scale \geq 0$, $G^\scale(P)$\label{notation:G-scale-P} is defined, with vertices $P^\scale_1, \ldots, P^\scale_{n_\scale} \sse \plaquettes$ which are compatible (and thus disjoint), and such that $P = P^\scale_1 \cup \cdots \cup P^\scale_{n_\scale}$. To define $G^{\scale+1}(P)$, first let $n_{\scale+1}$ be the number of connected components of $G^\scale(P)$, with connected components given by the partition $I_1 \cup \cdots \cup I_{n_{\scale+1}} = [n_\scale]$. For $1 \leq i \leq n_{\scale+1}$, define
\[ P^{\scale+1}_i := \bigcup_{j \in I_i} P^\scale_j. \]
The vertex set of $G^{\scale+1}(P)$ is $\{P^{\scale+1}_1, \ldots, P^{\scale+1}_{n_{\scale+1}}\}$, and the edge set is
\[ \{ \{P^{\scale+1}_i, P^{\scale+1}_j\} : i \neq j, J(P^{\scale+1}_i, P^{\scale+1}_j) = 1\}. \]
Observe that if $\scale \geq 1$ is such that $n_\scale = 1$, then $G^{\scale-1}(P)$ is connected. Also, for any $\scale' \leq \scale$, any vertex of $G^{\scale'}(P)$ must be contained in one of the vertices of $G^\scale(P)$. Finally, for all $\scale \geq 0$, the vertices of $G^\scale(P)$ form a partition of $P$.

\begin{lemma}\label{lemma:vertices-of-g-ell-contained-in-cube}
For any $\scale$, and any vertex $P^\scale_i$ of $G^\scale(P)$, the cube $B(P^\scale_i)$ has side length at most $3 |P^\scale_i|$.
\end{lemma}
\begin{proof}
We induct on $\scale$. First, for $\scale = 0$, the vertices of $G^\scale(P)$ are vortices. Let $P^0_i$ be a vertex of $G^\scale(P)$. By Lemma \ref{lemma:vortex-contained-in-cube}, there is a cube of side length $|P^0_i|$ which contains $P^0_i$. Extending this cube by (at most) one in all directions, we obtain a desired cube of side length at most $|P^0_i| + 2 \leq 3|P^0_i|$.

The key to the inductive step is the following. Suppose we have $P, P' \sse \plaquettes$, such that $B(P), B(P')$ have side lengths at most $3|P|, 3|P'|$, respectively. Suppose also that $J(P, P') = 1$. This implies that $B(P), B(P')$ must intersect. It is then possible to obtain a cube $B$ of side length at most $3|P| + 3|P'|$, which contains both $B(P)$ and $B(P')$.
\end{proof}

If $K \in \mc{K}$ is a knot, then the graphs $G^\scale(K), \scale \geq 0$ have some crucial additional properties. First, for $P \sse \plaquettes$, define $\scale^*(P) := \min \{ \scale : n_\scale = 1\}$. If $n_\scale > 1$ for all $\scale$, define $\scale^*(P) = \infty$.

\begin{lemma}
Let $K \in \mc{K}$, $\scale \geq 0$. If $n_\scale > 1$, then there are no isolated vertices of $G^\scale(K)$. Consequently, every connected component of $G^\scale(K)$ is of size at least $2$, and thus if also $\scale \leq \scale^*(K)$, then for all vertices $P^\scale_i$ of $G^\scale(K)$, we have $|P^\scale_i| \geq 2^\scale$.
\end{lemma}
\begin{proof}
Let $P^\scale_i$ be a vertex of $G^\scale(K)$. Since $n_\scale > 1$ by assumption, we have $K \backslash P^\scale_i \neq \varnothing$. Then by definition, there cannot exist a cube which well separates $P^\scale_i$ from $K \backslash P^\scale_i$. Therefore we must have $B(P^\scale_i) \cap (K \backslash P^\scale_i) \neq \varnothing$, and thus $P^\scale_i$ cannot be an isolated vertex in $G^\scale(K)$. 
\end{proof}

Note this lemma implies that if $K$ is a knot of size $m$, then $\scale^*(K) \leq \floor{\log_2 m}$, and thus $G^{\floor{\log_2 m} - 1}(K)$ is connected. 

\begin{proof}[Proof of Lemma \ref{lemma:knot-contained-in-cube}]
Given a knot $K \in \mc{K}$, by the previous discussion, we have that $G^{\floor{\log_2 m}}$ has a single vertex. Since the vertices of $G^\scale(K)$ form a partition of $K$ for all $\scale \geq 0$, we thus have that the lone vertex of $G^{\floor{\log_2 m}}$ must be $K$. Now finish by Lemma \ref{lemma:vertices-of-g-ell-contained-in-cube}.
\end{proof}

Let $\mc{D}$ be the collection of $P \sse \plaquettes$ such that $\scale^*(P) < \infty$, and for all $\scale \leq \scale^*(P)$, any vertex of $G^\scale(K)$ is of size at least $2^\scale$. Observe that if $P \in \mc{D}$ and $|P| = m$, then $\scale^*(P) \leq \floor{\log_2 m}$, and consequently $G^{\floor{\log_2 m} - 1}(P)$ is connected. Now define
\[\label{notation:A-m-scale} \mc{A}(m, \scale) := \{ P \in \mc{D} : \abs{P} = m, G^\scale(P) \text{ is connected}\}. \]
By the previous observation, note if $\scale \geq \floor{\log_2 m} - 1$, then $\mc{A}(m, \scale) = \mc{A}(m, \floor{\log_2 m} - 1)$. Also, as previously noted, if $K \in \mc{K}$ is a knot of size $m$, then $K \in \mc{A}(m, \floor{\log_2 m} - 1)$.
Now for $p \in \plaquettes$, define
\[\label{notation:A-m-scale-p} \mc{A}(m, \scale, p) := \{P \in \mc{A}(m, \scale) : P \ni p\}. \]
To upper bound the number of knots of size $m$ which contain $p$, it suffices to upper bound $|\mc{A}(m, \floor{\log_2 m} -1, p)|$. We will do this by inductively upper bounding $|\mc{A}(m, \scale, p)|$, for $\scale \leq \floor{\log_2 m} - 1$. We first start with the base case $\scale = 0$.

\begin{lemma}\label{lemma:ell-equals-zero-bound}
For any $m \geq 1$, and any $p \in \plaquettes$, we have
\[|\mc{A}(m, 0, p)| \leq (\vbdconst)^m (e^{9/e} 2^{19} e^4)^m. \]
\end{lemma}

Before we prove Lemma \ref{lemma:ell-equals-zero-bound}, we need the following preliminary result.

\begin{lemma}\label{lemma:connectivity-graph-edge}
Fix a vortex $\vortex \sse \plaquettes$, and $m' \geq 1$. The number of vortices $\vortex' \sse \plaquettes$, such that $|\vortex'| = m'$, and $J(\vortex, \vortex') = 1$, is at most 
\[ 2^{19} |\vortex|^4 (m')^4 (\vbdconst)^{m'}.\]
\end{lemma}
\begin{proof}
We upper bound the number of $\vortex'$ such that $\vortex' \cap B(\vortex) \neq \varnothing$, and the number of $\vortex'$ such that $\vortex \cap B(\vortex') \neq \varnothing$. For the first case, observe that as $B(\vortex)$ has side length at most $3 |\vortex|$, it has at most $(3|\vortex|+1)^4$ vertices and each vertex is incident to at most $8 \cdot 6 = 48$ plaquettes. Thus $B(\vortex)$ contains at most $48 (3 |\vortex|+1)^4 \leq 48 \cdot 4^4 |\vortex|^4$ plaquettes. Now apply Lemma \ref{lemma:vortex-combinatorial-bound} to upper bound the number of $\vortex'$ such that $|\vortex'| = m'$, and $\vortex' \cap B(\vortex) \neq \varnothing$ by $48 \cdot 4^4 |\vortex|^4 (\vbdconst)^{m'}$.

In the second case, observe if $\vortex \cap B(\vortex') \neq \varnothing$, then $\vortex'$ must contain some plaquette $p$, for which there exists a cube of side length $3m'$ which contains both $p$ and a plaquette of $\vortex$. Any such plaquette $p$ must be contained in a cube of side length $6m'$ centered at a vertex of a plaquette in $\vortex$. The number of plaquettes in a cube of side length $6m'$ is at most $48 (6m'+1)^4 \leq 48 \cdot 7^4 (m')^4$, and the number of vertices in $\vortex$ is at most $4 |\vortex|$. Using these observations along with Lemma \ref{lemma:vortex-combinatorial-bound}, we thus obtain the upper bound $4 \cdot 48 \cdot 7^4 \cdot |\vortex| (m')^4 (\vbdconst)^{m'}$.
To finish, we may further upper bound
\[48 \cdot 4^4 |\vortex|^4 (\vbdconst)^{m'} +  4 \cdot 48 \cdot 7^4 |\vortex| (m')^4 (\vbdconst)^{m'} \leq (48 \cdot 4^4 + 4 \cdot 48 \cdot 7^4) |\vortex|^4 (m')^4 (\vbdconst)^{m'}, \] 
and note $48 \cdot 4^4 + 4 \cdot 48 \cdot 7^4 \leq 2^{19}$.
\end{proof}

\begin{proof}[Proof of Lemma \ref{lemma:ell-equals-zero-bound}]
Given an ordered tuple of plaquette sets $(\plaqset_1, \ldots, \plaqset_n)$, define $H(\plaqset_1, \ldots, \plaqset_n)$ to be the undirected graph with vertex set $[n]$, and edge set $\{\{i, j\} : i \neq j, J(\plaqset_i, \plaqset_j) = 1\}$. (Note $H(\cdot)$ is reminiscent of $G^0(\cdot)$, although one difference is that the vertices of $H(\cdot)$ are integers, whereas the vertices of $G^0(\cdot)$ are plaquette sets. Another difference is that $H$ can encode the connectivity between general plaquette sets, while $G^0$ only encodes the connectivity between vortices.)  Now given $P \in \mc{A}(m, 0, p)$, suppose $G^0(P)$ has vertices $\plaqset^0_1, \ldots, \plaqset^0_{n_0}$, which by definition are compatible vortices. Note then that $P$ corresponds to the collection of ordered tuples
\[ \{ (\plaqset^0_{\pi(1)}, \ldots, \plaqset^0_{\pi(n_0)}) : \pi \text{ a permutation of $[n_0]$} \}. \]
Moreover, since $G^0(P)$ is connected, we have that for any permutation $\pi$, $H(\plaqset^0_{\pi(1)}, \ldots, \plaqset^0_{\pi(n_0)})$ is also connected. Also, $\plaqset^0_{\pi(1)} \cup \cdots \cup \plaqset^0_{\pi(n_0)} = P$ for all $\pi$, and thus $|\plaqset^0_{\pi(1)}| + \cdots + |\plaqset^0_{\pi(n)}| = m$ for all $\pi$. We thus have
\[ |\mc{A}(m, 0, p)| \leq \sum_{n_0=1}^m \frac{1}{n_0!} \sum_{\substack{\plaqset_1, \ldots, \plaqset_{n_0} \\ |\plaqset_1| + \cdots + |\plaqset_{n_0}| = m \\ \exists\, i\, : \, \plaqset_i \ni p}} \ind(H(\plaqset_1, \ldots, \plaqset_{n_0}) \text{ is connected}). \]
Note the sum over $\plaqset_1, \ldots, \plaqset_{n_0}$ also includes the additional restriction that the $\plaqset_1, \ldots, \plaqset_{n_0}$ are compatible vortices. Due to space reasons, we will not explicitly write this (actually, we can even drop the condition that the $\plaqset_i$ are compatible -- we only need the fact that the $\plaqset_i$ are vortices, to eventually apply Lemma \ref{lemma:connectivity-graph-edge}). Now applying Lemma \ref{lemma:vortex-combinatorial-bound} to handle the case $n_0 = 1$, we can upper bound the above display by
\beq \label{eq:a-m-split-into-number-components-bound} (\vbdconst)^m + \sum_{n_0=2}^m \frac{1}{n_0!} \sum_{k=1}^{n_0} \sum_{\substack{\plaqset_1, \ldots, \plaqset_{n_0} \\ |\plaqset_1| + \cdots + |\plaqset_{n_0}| = m \\ \plaqset_k \ni p}} \ind(H(\plaqset_1, \ldots, \plaqset_{n_0}) \text{ is connected}). \eeq
Fix $2 \leq n_0 \leq m$, and set $k = 1$ (say). We proceed to bound
\beq\label{eq:fix-vortex-size}\begin{split}
\sum_{\substack{\plaqset_1, \ldots, \plaqset_{n_0} \\ |\plaqset_1| + \cdots + |\plaqset_{n_0}| = m \\ \plaqset_1 \ni p}}& \ind(H(\plaqset_1, \ldots, \plaqset_{n_0}) \text{ is connected}) \leq \\
&\sum_{\substack{m_1 + \cdots + m_{n_0} = m \\ m_1, \ldots, m_{n_0} \geq 1}} \sum_{\substack{|\plaqset_i| = m_i, 1 \leq i \leq n_0  \\ \plaqset_1 \ni p}} \ind(H(\plaqset_1, \ldots, \plaqset_{n_0}) \text{ is connected}). 
\end{split}\eeq
Fixing $m_1, \ldots, m_{n_0} \geq 1$, observe
\beq\label{eq:introduce-trees} \begin{split}
\sum_{\substack{|\plaqset_i| = m_i, 1 \leq i \leq n_0  \\ \plaqset_1 \ni p}}& \ind(H(\plaqset_1, \ldots, \plaqset_{n_0}) \text{ is connected}) \leq \\
&\sum_{T \text{ on } [n_0]} \sum_{\substack{|\plaqset_i| = m_i, 1 \leq i \leq n_0 \\ \plaqset_1 \ni p}} \ind(H(\plaqset_1, \ldots, \plaqset_{n_0}) \supseteq T). 
\end{split}\eeq
Here $\sum_{T \text{ on } [n_0]}$ denotes a sum over trees with vertex set $[n_0]$, and the expression $H(\plaqset_1, \ldots, \plaqset_{n_0}) \supseteq T$ means $T$ is a subgraph of $H(\plaqset_1, \ldots, \plaqset_{n_0})$. Now fix a tree $T$ with vertex set $[n_0]$. Let the $d_i$ be the degree of vertex $i$ in $T$, for all $1 \leq i \leq n_0$. I claim
\beq\label{eq:sum-of-vortices-bound-claim} \sum_{\substack{|\plaqset_i| = m_i, 1 \leq i \leq n_0  \\ \plaqset_1 \ni p}} \ind(H(\plaqset_1, \ldots, \plaqset_{n_0}) \supseteq T) \leq \prod_{i=1}^{n_0} 2^{19} m_i^{4d_i} (\vbdconst)^{m_i}.\eeq
Given this claim (whose proof we postpone until the end), the rest of the proof is straightforward, but slightly tedious. We obtain
\beq\label{eq:bound-in-terms-of-degrees}\sum_{T \text{ on } [n_0]} \sum_{\substack{|\plaqset_i| = m_i, 1 \leq i \leq n_0 \\ \plaqset_1 \ni p}} \ind(H(\plaqset_1, \ldots, \plaqset_{n_0}) \supseteq T) \leq 2^{19 n_0} (\vbdconst)^m \sum_{T \text{ on } [n_0]} \prod_{i=1}^{n_0} m_i^{4d_i}.  \eeq
Note $d_1 + \cdots + d_{n_0} = 2(n_0 - 1)$, and for any such sequence $d_1, \ldots, d_{n_0}$, the number of trees with this degree sequence is $\frac{(n_0 - 2)!}{(d_1 - 1)! \cdots (d_{n_0} - 1)!}$ (see e.g. \cite{BOL1998}, Chapter VIII, Corollary 21). We may thus bound
\beq\label{eq:introduce-degrees} \sum_{T \text{ on } [n_0]} \prod_{i=1}^{n_0} m_i^{4d_i} \leq (n_0 - 2)! \sum_{\substack{d_1 + \cdots + d_{n_0} = 2(n_0 - 1) \\ d_1, \ldots, d_{n_0} \geq 1}} \prod_{i=1}^{n_0} \frac{m_i^{4d_i}}{(d_i - 1)!}. \eeq
Combining equations \eqref{eq:fix-vortex-size}, \eqref{eq:introduce-trees}, \eqref{eq:bound-in-terms-of-degrees}, and \eqref{eq:introduce-degrees}, we obtain
\beq \label{eq:bound-in-terms-of-degrees-and-m-i}\begin{split}
&\sum_{\substack{\plaqset_1, \ldots, \plaqset_{n_0} \\ |\plaqset_1| + \cdots + |\plaqset_{n_0}| = m \\ \plaqset_1 \ni p}} \ind(H(\plaqset_1, \ldots, \plaqset_{n_0}) \text{ is connected}) \leq \\
&(n_0 - 2)! 2^{19n_0} (\vbdconst)^m \sum_{\substack{d_1 + \cdots + d_{n_0} = 2(n_0-1) \\ d_1, \ldots, d_{n_0} \geq 1}} \sum_{\substack{m_1 + \cdots + m_{n_0} = m \\ m_1, \ldots, m_{n_0} \geq 1}}\prod_{i=1}^{n_0} \frac{m_i^{4d_i}}{(d_i - 1)!}. 
\end{split}\eeq
Fixing $d_1 + \cdots + d_{n_0} = 2(n_0 - 1)$, for any $m_1 + \cdots + m_{n_0} = m$ with $m_1, \ldots, m_{n_0} \geq 1$, we may bound
\[ \prod_{i=1}^{n_0} m_i^{4d_i} \leq \prod_{i=1}^{n_0} \bigg(\frac{md_i}{2(n_0-1)}\bigg)^{4d_i} = \bigg(\frac{m}{2(n_0-1)}\bigg)^{8(n_0 - 1)} \prod_{i=1}^{n_0} d_i^{4d_i}.  \]
Let $C := m / (2(n_0 - 1))$. We thus have
\[\begin{split}
\sum_{\substack{d_1 + \cdots + d_{n_0} = 2(n_0-1) \\ d_1, \ldots, d_{n_0} \geq 1}} &\sum_{\substack{m_1 + \cdots + m_{n_0} = m \\ m_1, \ldots, m_{n_0} \geq 1}} \prod_{i=1}^{n_0} \frac{m_i^{4d_i}}{(d_i - 1)!} \leq \\
&C^{8(n_0 - 1)} \sum_{\substack{d_1 + \cdots + d_{n_0} = 2(n_0 - 1) \\ d_1, \ldots, d_{n_0} \geq 1}} \prod_{i=1}^{n_0} \frac{d_i^{4d_i}}{(d_i - 1)!} \sum_{\substack{m_1 + \cdots + m_{n_0} = m \\ m_1, \ldots, m_{n_0} \geq 1}} 1.
\end{split}\]
We may further upper bound the above by
\beq\label{eq:get-rid-of-m-i-bound} C^{8(n_0 - 1)} \binom{m-1}{n_0-1} \sum_{\substack{d_1 + \cdots + d_{n_0} = 2(n_0-1) \\ d_1, \ldots, d_{n_0} \geq 1}} \prod_{i=1}^{n_0} \frac{d_i^{4d_i}}{(d_i - 1)!}. \eeq
Now for any $d_1 + \cdots + d_{n_0} = 2(n_0 - 1)$ with $d_1, \ldots, d_{n_0} \geq 1$, we may upper bound
\[ \prod_{i=1}^{n_0} d_i^{4d_i} \leq \prod_{i=1}^{n_0} \bigg(\frac{2(n_0-1)}{n_0}\bigg)^{4d_i} = \bigg(\frac{2(n_0 - 1)}{n_0}\bigg)^{8(n_0-1)}. \]
Let $C' := 2(n_0 - 1) / n_0$. We thus have
\beq\label{eq:power-of-degree-bound}\sum_{\substack{d_1 + \cdots + d_{n_0} = 2(n_0-1) \\ d_1, \ldots, d_{n_0}}} \prod_{i=1}^{n_0} \frac{d_i^{4d_i}}{(d_i - 1)!} \leq (C')^{8(n_0-1)} \sum_{\substack{d_1 + \cdots + d_{n_0} = 2(n_0 - 1) \\ d_1, \ldots, d_{n_0} \geq 1}} \prod_{i=1}^{n_0} \frac{1}{(d_i - 1)!}. \eeq
Note
\beq\label{eq:product-of-degree-bound-by-exp} \sum_{\substack{d_1 + \cdots + d_{n_0} = 2(n_0 - 1) \\ d_1,\ldots, d_{n_0} \geq 1}} \prod_{i=1}^{n_0} \frac{1}{(d_i - 1)!} \leq \prod_{i=1}^{n_0} \bigg(\sum_{d_i \geq 1} \frac{1}{(d_i - 1)!} \bigg) = e^{n_0}. \eeq
Combining equations \eqref{eq:bound-in-terms-of-degrees-and-m-i}, \eqref{eq:get-rid-of-m-i-bound}, \eqref{eq:power-of-degree-bound}, and \eqref{eq:product-of-degree-bound-by-exp}, we obtain
\[\begin{split} 
\sum_{\substack{\plaqset_1, \ldots, \plaqset_{n_0} \\ |\plaqset_1| + \cdots + |\plaqset_{n_0}| = m \\ \plaqset_1 \ni p}} &\ind(H(\plaqset_1, \ldots, \plaqset_{n_0}) \text{ is connected}) \leq \\
&(n_0 - 2)! 2^{19n_0} (\vbdconst)^m \binom{m-1}{n_0-1} \bigg(\frac{m}{n_0}\bigg)^{8(n_0-1)} e^{n_0}.
\end{split}\] 
By the same argument, the above inequality still holds with the condition $\plaqset_1 \ni p$ replaced by $\plaqset_k \ni p$, for any $2 \leq k \leq n_0$. Combining this with equation \eqref{eq:a-m-split-into-number-components-bound}, we obtain
\[ |\mc{A}(m, 0, p)| \leq (\vbdconst)^m + (\vbdconst)^m \sum_{n_0=2}^m \frac{(2^{19} e)^{n_0}}{n_0 -1} \bigg(\frac{m}{n_0}\bigg)^{8(n_0-1)}\binom{m-1}{n_0-1}.  \]
We may bound
\[\binom{m-1}{n_0-1} \leq \bigg(\frac{(m-1) e}{n_0-1}\bigg)^{n_0-1}, \]
and for $x > 0$, we have $(m / x)^x \leq e^{m/e}$.
Thus
\begin{align*} 
|\mc{A}(m, 0, p)| &\leq (\vbdconst)^m + (\vbdconst)^m e^{9m/e}\sum_{n_0 = 2}^m \frac{(2^{19} e^2)^{n_0}}{n_0-1}  \\
&\leq (\vbdconst)^m + (\vbdconst)^m e^{9m/e} (m-1) \frac{(2^{19} e^2)^m}{m-1} \\
&\leq (\vbdconst)^m (e^{9/e} 2^{19} e^4)^m. 
\end{align*}
To show the claim \eqref{eq:sum-of-vortices-bound-claim}, fix a tree $T$ on $[n_0]$. First, by Lemma \ref{lemma:vortex-combinatorial-bound}, the number of possible choices of $\plaqset_1$, such that $|\plaqset_1| = m_1$, $\plaqset_1 \ni p$ is at most $(\vbdconst)^{m_1}$. Let $i_1, \ldots, i_{d_1}$ be the neighbors of $1$ in $T$. Having chosen $\plaqset_1$, each $\plaqset_{i_j}$, $1 \leq j \leq d_1$ must be such that $|\plaqset_{i_j}| = m_{i_j}$, and $J(\plaqset_1, \plaqset_{i_j}) = 1$. Applying Lemma \ref{lemma:connectivity-graph-edge}, we thus have that the number of choices of $\plaqset_{i_j}$ is at most $2^{19} m_1^4 m_{i_j}^4 (\vbdconst)^{m_{i_j}}$. Now continue in this manner until we have chosen $\plaqset_i$ for all $1 \leq i \leq n_0$. 
\end{proof}

We now proceed to upper bound $|\mc{A}(m, \scale, p)|$ for general $\scale$. Naturally, the proof will follow by induction $\scale$. The proof of the inductive step will have the same form as the proof of Lemma \ref{lemma:ell-equals-zero-bound}, which we just finished. Recall that the first step in this proof was to associate the vertices $G^0(P)$ with a collection of ordered tuples. The following lemma allows us to do the same thing for the vertices of $G^\scale(P)$, for general $\scale$.

\begin{lemma}\label{lemma:decompose-ell-plus-one-into-ell-components}
Let $m \geq 1$, $\scale \geq 0$. Let $P \in \mc{A}(m, \scale+1)$. Let the vertices of $G^{\scale+1}(P)$ be $\plaqset^{\scale+1}_1, \ldots, \plaqset^{\scale+1}_{n_{\scale+1}}$. Then for all $1 \leq i \leq n_{\scale+1}$, we have $\plaqset^{\scale+1}_i \in \mc{A}(|\plaqset^{\scale+1}_i|, \scale)$.
\end{lemma}
\begin{proof}
Without loss of generality, take $i = 1$, and let $m_1 := |\plaqset^{\scale+1}_1|$. To show $\plaqset^{\scale+1}_1 \in \mc{A}(m_1, \scale)$, we need to show that $G^\scale(\plaqset^{\scale+1}_1)$ is connected, and that for all $\scale' \leq \scale^*(\plaqset^{\scale+1}_1)$, we have that all vertices of $G^{\scale'}(\plaqset^{\scale+1}_1)$ are of size at least $2^{\scale'}$. To show both these conditions, it suffices to show the following claim: for all $\scale' \leq \scale+1$, we have that the vertices of $G^{\scale'}(\plaqset^{\scale+1}_1)$ are also vertices of $G^{\scale'}(P)$.

To show this claim, observe that due to the hierarchical nature of the construction of $G^\scale(P)$, for all $\scale' \leq \scale$, we have that $G^{\scale'}(P)$ is a disjoint union of the graphs $G^{\scale'}(\plaqset^{\scale+1}_1), \ldots, G^{\scale'}(\plaqset^{\scale+1}_{n_{\scale+1}})$. This shows the claim for $\scale' \leq \scale$. For $\scale' = \scale+1$, observe that the connected components of $G^\scale(P)$ are the connected components of $G^\scale(\plaqset^{\scale+1}_1) \ldots, G^\scale(\plaqset^{\scale+1}_{\scale+1})$, and thus as the vertices of $G^{\scale+1}(P)$ are unions of the connected components of $G^\scale(P)$, it follows that the vertices of $G^{\scale+1}(P)$ are the vertices of $G^{\scale+1}(\plaqset^{\scale+1}_1) \ldots, G^{\scale+1}(\plaqset^{\scale+1}_{n_{\scale+1}})$. This implies that $G^{\scale+1}(\plaqset^{\scale+1}_1)$ can only have a single vertex, which then must be $\plaqset^{\scale+1}_1$.
\end{proof}

We will also need the following analogue of Lemma \ref{lemma:connectivity-graph-edge}.

\begin{lemma}\label{lemma:connectivity-graph-edge-general-ell}
Let $\scale \geq 0$. Suppose there is a constant $C$, such that for all $m \geq 1$, $p \in \plaquettes$, we have $|\mc{A}(m, \scale, p)| \leq C^m$. Fix $m, m' \geq 1$, and $P \in \mc{A}(m, \scale)$. Then the number of $P' \in \mc{A}(m', \scale)$ such that $J(P, P') = 1$ is at most
\[ 2^{19} m^4 (m')^4 C^{m'}.\]
\end{lemma}
\begin{proof}
Note for any $P \in \mc{A}(m, \scale)$, the graph $G^{\scale+1}(P)$ has a single vertex, which then must be $P$. Thus by Lemma \ref{lemma:vertices-of-g-ell-contained-in-cube}, the cube $B(P)$ has side length at most $3|P|$. Given this, the argument is the exact same as for Lemma \ref{lemma:connectivity-graph-edge}.
\end{proof}

We now have all the ingredients needed to be able to bound $|\mc{A}(m, \scale, p)|$ for general $\scale$.

\begin{lemma}\label{lemma:bound-on-a-m-p-inductive-step}
Let $\scale \geq 0$. Suppose there is a constant $C$ such that for all $m \geq 1$, $p \in \plaquettes$, we have $|\mc{A}(m, \scale, p)| \leq C^m$. Then for all $m \geq 1$, $p \in \plaquettes$, we have
\[ |\mc{A}(m, \scale+1, p)| \leq \begin{cases} \big(C (2^{19} e^4)^{2^{-(\scale+1)}} e^{9/e}\big)^m & \scale=0 \\ \big(C (2^{19} e^4)^{2^{-(\scale+1)}} (2^{\scale+1})^{9\cdot 2^{-(\scale+1)}}\big)^m & \scale \geq 1\end{cases}. \]
\end{lemma}

Before we prove this lemma, we show its consequences. Upon combining Lemmas \ref{lemma:ell-equals-zero-bound} and \ref{lemma:bound-on-a-m-p-inductive-step}, we obtain by induction the following corollary.

\begin{cor}\label{cor:bound-on-a-m-p-general-ell}
For all $m \geq 1$, $0 \leq \scale \leq \floor{\log_2 m} - 1$, $p \in \plaquettes$, we have 
\[ |\mc{A}(m, \scale, p)| \leq (20 e^{18/e} 2^{38} e^9 2^{27/2})^m \leq (10^{24})^m. \]
\end{cor}

\begin{proof}[Proof of Lemma \ref{lemma:knot-combinatorial-bound}]
As previously detailed, every knot $K \in \mc{K}$ of size $m$ is in the set $\mc{A}(m, \floor{\log_2 m} - 1)$. Thus for fixed $m \geq 1$, the number of knots $K \in \mc{K}$ of size $m$ which contain $p$ is bounded by $|\mc{A}(m, \floor{\log_2 m} - 1, p)|$. Now apply Corollary \ref{cor:bound-on-a-m-p-general-ell}.
\end{proof}

\begin{proof}[Proof of Lemma \ref{lemma:bound-on-a-m-p-inductive-step}]
Fix $m \geq 1$. Recall that if $\scale +1 > \floor{\log_2 m} - 1$, then $\mc{A}(m, \scale +1) = \mc{A}(m, \floor{\log_2 m} - 1) = \mc{A}(m, \scale)$. Thus we may assume that $\scale + 1 \leq \floor{\log_2 m} - 1$. For $\plaqset_1, \ldots, \plaqset_n \sse \plaquettes$, let $H(\plaqset_1, \ldots, \plaqset_n)$ be the graph as defined in the proof of Lemma \ref{lemma:ell-equals-zero-bound}. By applying Lemma \ref{lemma:decompose-ell-plus-one-into-ell-components}, we may proceed as in the proof of Lemma \ref{lemma:ell-equals-zero-bound} to upper bound $|\mc{A}(m, \scale+1, p)|$ by
\[\begin{split}
C^m +
\sum_{n_{\scale+1} = 2}^{\floor{m2^{-(\scale+1)}}} \frac{1}{(n_{\scale+1})!} 
\sum_{k=1}^{n_{\scale+1}} \sum_{\substack{m_1 + \cdots +\\ m_{n_{\scale+1}} = m \\ m_1, \ldots, m_{n_{\scale+1}} \geq 1}} \sum_{\substack{\plaqset_i \in \mc{A}(m_i, \scale) \\1 \leq i \leq n_{\scale+1}\\ \plaqset_k \ni p}} \ind(H(\plaqset_1, \ldots, \plaqset_{n_{\scale+1}}) \text{ is connected}). 
\end{split}\]
Here we have $n_{\scale+1} \leq m2^{-(\scale+1)}$ because for $P \in \mc{A}(m, \scale+1, p)$, we have that for $\scale' \leq \scale^*(P)$, all vertices of $G^{\scale'}(P)$ are of size at least $2^{\scale'}$. As $\scale+1 \leq \floor{\log_2 m} - 1$, it then follows that all vertices of $G^{\scale+1}(P)$ are of size at least $2^{\scale+1}$, and thus $G^{\scale+1}(P)$ can have at most $m 2^{-(\scale+1)}$ vertices. Note also that in the sum over $m_1 + \cdots + m_{n_{\scale+1}} = m$, we can further restrict $m_i \geq 2^{\scale+1}$, but we can afford to be loose and ignore this constraint. Now by summing over trees as in the proof of Lemma \ref{lemma:ell-equals-zero-bound}, and using Lemma \ref{lemma:connectivity-graph-edge-general-ell} in place of Lemma \ref{lemma:connectivity-graph-edge} to obtain the analogue of the inequality \eqref{eq:sum-of-vortices-bound-claim}, we obtain
\beq\label{eq:a-m-ell-plus-one-p-bound} \abs{\mc{A}(m, \scale+1, p)} \leq C^m + C^m D, \eeq
where
\[ D := \sum_{n_{\scale+1}=2}^{\floor{m 2^{-(\scale+1)}}} \frac{(2^{19} e)^{n_{\scale+1}}}{n_{\scale+1} - 1} \bigg(\frac{m}{n_{\scale+1}}\bigg)^{8(n_{\scale+1} - 1)} \binom{m-1}{n_{\scale+1} - 1}.\]
Now observe
\[ \sup_{1 \leq x \leq \floor{m 2^{-(\scale+1)}}} \bigg(\frac{m}{x}\bigg)^x \leq \begin{cases} e^{m/e} & \scale = 0 \\ (2^{\scale+1})^{m2^{-(\scale+1)}} & \scale \geq 1 \end{cases}.\]
Suppose $\scale \geq 1$. The case $\scale = 0$ will follow similarly. As
\[ \binom{m-1}{n_{\scale+1} - 1} \leq \bigg(\frac{(m-1)e}{n_{\scale+1} - 1}\bigg)^{n_{\scale+1} - 1}, \]
we obtain
\[ D \leq (2^{19} e^2)^{m2^{-(\scale+1)}} (2^{\scale+1})^{9m 2^{-(\scale+1)}} m2^{-(\scale+1)}.\]
Combining this with \eqref{eq:a-m-ell-plus-one-p-bound}, the desired result follows.
\end{proof}

\appendix 

\numberwithin{theorem}{section}

\section{}

We prove the representation theory statements which were made in Section \ref{section:intro}.

\begin{lemma}\label{lemma:exist-unitary-rep}
Let $G$ be a finite group of order at least 3. There exists a faithful unitary representation $\rho$ such that $\norm{A}_{op} < 1$.
\end{lemma}
\begin{proof}
Let $n := |G| \geq 3$. First, take $\tilde{\rho}$ to be the regular representation of $G$, with character $\tilde{\chi}$. This faithfully represents $G$ as a set of permutation matrices on the vector space $\C^n$, and moreover, we have $\tilde{\chi}(\groupid) = n$, and $\tilde{\chi}(g) = 0$ for all $g \neq \groupid$, so that $G_0 = G \backslash \{1\}$. Enumerate $G = \{g_1, \ldots, g_n\}$, with $g_1 = \groupid$, and let $\pi_1, \ldots, \pi_n$ be the permutations which give $\rho(g_1), \ldots, \rho(g_n)$. Observe that the permutations $\pi_2, \ldots, \pi_n$ (i.e. the permutations which correspond to non-identity elements), have no fixed points. Moreover, for any $2 \leq i \neq i' \leq n$, and any $1 \leq j \leq n$, we have $\pi_i(j) \neq \pi_{i'}(j)$. This implies that
\[ A = \frac{1}{n-1} \sum_{g \neq 1} \rho(g) = \frac{1}{n-1} (\mbf{1} - I),\]
where $\mbf{1}$ is the $n\times n$ matrix with every entry equal to 1, and $I$ is the $n \times n$ identity matrix. Now define the subspace 
\[ V := \bigg\{x \in \C^n : \sum_{i=1}^n x_i = 0\bigg\}. \]
Observe that on $V$, the matrix $A$ acts as $-\frac{1}{n-1}$ times the identity. As $n \geq 3$, we have $\frac{1}{n-1} \leq \frac{1}{2}$. Thus the desired representation $\rho$ may be obtained by taking the subrepresentation of $\tilde{\rho}$ given by the invariant subspace $V$.
\end{proof}

\begin{lemma}\label{lemma:faithful-unitary-irrep}
Let $G$ be a finite group of order at least 3, and let $\rho$ be a faithful unitary representation of $G$. If $\rho$ is irreducible, then $\norm{A}_{op} < 1$.
\end{lemma}
\begin{proof}
Observe that for any $h \in G$, we have $h G_0 h^{-1} = G_0$. This implies that for any $h \in G$, we have
\[ \rho(h) A \rho(h^{-1}) = A,\]
i.e. $A$ is an intertwiner. Since $\rho$ is irreducible, we may apply Schur's lemma to obtain that $A$ must be a multiple $\lambda$ of the identity. Since $A$ is Hermitian, and $\norm{A}_{op} \leq 1$, we have $\lambda \in [-1, 1]$. If $|\lambda| = 1$, then we must have that for all $g \in G_0$, $\rho(g)$ is $\lambda$ times the identity. Because $\rho$ is faithful, we have $\lambda \neq 1$. Now if $\lambda = -1$, then by the definition of $G_0$, we must have $G_0 = G \backslash \{1\}$, which by assumption is of size at least 2. This then contradicts the assumption that $\rho$ is faithful. Thus $|\lambda| \neq 1$, i.e. $\lambda \in (-1, 1)$.
\end{proof}

\section{}

We prove the topological statements which were needed in Section \ref{section:non-abelian-case}.

\begin{proof}[Proof of Lemma \ref{lemma:minimal-vortex-cell-complexes-simply-connected}]
Suppose without loss of generality that $e$ is the edge between $(0, 0, 0, 0)$ and $(1, 0, 0, 0)$. We start with $S_e$. Observe that $S_e$ is defined by including all plaquettes which are contained in a 3-cell which contains $e$. Consequently, we may attach all such 3-cells to $S_e$, without changing its fundamental group. Call the resulting space $\tilde{S}_e$. After doing this, we may then attach all 4-cells whose boundary 3-cells are in $\tilde{S}_e$, without changing the fundamental group. The resulting space is the union of the rectangular prisms $[0, 1] \times [-1, 1]^2 \times \{0\}$, $[0, 1] \times [-1, 1] \times \{0\} \times [-1, 1]$, and $[0, 1] \times \{0\} \times [-1, 1]^2$. The fundamental group of any of these prisms is trivial, and the intersection of these three prisms is the line segment $[0, 1] \times \{0\}^3$, which is path connected. Thus by the Seifert-Van Kampen theorem, we obtain that the fundamental group of the union of these spaces is also trivial, and thus also $\pi_1(S_e)$ is trivial.

Now onto $\partial S_e$. As noted right before Lemma \ref{lemma:minimal-vortex-cell-complexes-simply-connected}, recall that $\partial S_e$ is the union of the boundaries of three rectangular prisms. Denote the three boundaries by $B_1, B_2, B_3$. Each $B_i$, $1 \leq i \leq 3$ is simply connected. Recall also that $B_1, B_2$ intersect on the boundary of a rectangle, which is path connected. Thus by Seifert-Van Kampen, $B_1 \cup B_2$ is also simply connected. Now observe that $B_1 \cup B_2$ and $B_3$ intersect on the union of the boundaries of two rectangles, and moreover, these two boundaries have nonempty intersection. Thus the intersection of $B_1 \cup B_2$ and $B_3$ is path connected, and thus again by Seifert-Van Kampen, $\partial S_e = B_1 \cup B_2 \cup B_3$ is simply connected.

Finally, we look at $S_e^c$. Observe $S_e \cap S_e^c = \partial S_e$, while $S_e \cup S_e^c = \twoskel$. Thus by Seifert-Van Kampen, we have that $\pi_1(\twoskel) = \pi_1(S_e) * \pi_1(S_e^c)$, where $*$ denotes free product of groups. As both $\pi_1(\twoskel) = \pi_1(S_e) = \{\groupid\}$, we thus must also have $\pi_1(S_e^c) = \groupid$.
\end{proof}

\begin{proof}[Proof of Lemma \ref{lemma:rectangle-cell-complexes-simply-connected}]
First, to see why $S_2(B)$ is simply connected, we may attach all 3-cells whose boundary 2-cells are all contained in $S_2(B)$, without changing the fundamental group. Call the resulting 3-complex $S_3(B)$. We may then attach all 4-cells whose boundary 3-cells are all contained in $S_3(B)$, without changing its fundamental group. The resulting space is a four dimensional cube in $\R^4$, which is simply connected.

Now onto $\partial S_2(B)$. If $B$ is contained in the interior of $\lbox$, then by attaching 3-cells and 4-cells to $\partial S_2(B)$ as before, we obtain the boundary of a four dimesional cube in $\R^4$, which is simply connected. If instead $B$ intersects the boundary of $\lbox$, then upon attaching 3-cells and 4-cells, we obtain a space which is homeomorphic to a three dimensional unit ball, which is simply connected. Here we have used the assumption that all side lengths of $B$ are strictly less than the side length of $\lbox$.

Finally, we look at $S_2^c(B)$. Observe that $S_2(B) \cap S_2^c(B) = \partial S_2(B)$, which is simply connected. Also observe $S_2(B) \cup S_2^c(B) = \twoskel$, which is simply connected. Thus by Seifert-Van Kampen, we have $\pi_1(\twoskel) = \pi_1(S_2(B)) * \pi_1(S_2^c(B))$. As both $\pi_1(\twoskel) = \pi_1(S_2(B)) = \{\groupid\}$, we must also have $\pi_1(S_2^c(B)) = \{\groupid\}$.
\end{proof}

\begin{lemma}\label{lemma:plaquette-sets-too-small-means-simply-connected}
Let $\plaqset \sse \plaquettes$ be such that for every plaquette $p \in \plaqset$, every 3-cell which contains $p$ is completely contained in $\lbox$. If $|\plaqset| \leq 5$, then $\pi_1(\twoskelminus{\plaqset}) = \{1\}$. If $|\plaqset| = 6$, then $\pi_1(\twoskelminus{\plaqset}) \neq \{\groupid\}$ if and only if $\plaqset = P(e) \sse \plaquettes.$
\end{lemma}
\begin{proof}
Fix a vertex $x_0 \in \vertices$, and a spanning tree $T$ of $\oneskel$. We have the presentation
\[ \pi_1(\twoskelminus{\plaqset}, x_0) = \langle a_e, e \in \oneskel \backslash T ~|~ C_p, p \in \plaquettes \backslash \plaqset\rangle. \]
If we can show that for all $p \in \plaqset$, we have $C_p = \groupid$, then we would have
\[\pi_1(\twoskelminus{\plaqset}, x_0) = \langle a_e, e \in \oneskel \backslash T ~|~ C_p, p \in \plaquettes \rangle = \pi_1(\twoskelminus{\varnothing}, x_0).\] 
But $\twoskelminus{\varnothing} = \twoskel$ is the 2-skeleton of $\lbox$, which is simply connected. Thus our goal is to show that $C_p = \groupid$, for all $p \in \plaqset$.

This follows from the following observation. For $p \in \plaqset$, if there is a 3-cell $c$ which contains $p$, and such that for every other plaquette $p'$ of $c$, we have $C_{p'} = \groupid$, then also $C_p = \groupid$. Here we've used the assumption that since $c$ contains a plaquette of $\plaqset$, $c$ is contained in $\lbox$. Now if $|\plaqset| \leq 5$, then there always exists a plaquette $p \in \plaqset$ for which there is such a 3-cell $c$. The same holds if $|\plaqset| = 6$ but $\plaqset$ is not a minimal vortex.
\end{proof}

The proof of the following lemma is due to Ciprian Manolescu. 

\begin{lemma}\label{lemma:rank-bound}
For any plaquette set $\plaqset \sse \plaquettes$, we have $\mathrm{rk}(\pi_1(\twoskelminus{\plaqset})) \leq |\plaqset|$.
\end{lemma}
\begin{proof}
Suppose $\lbox = ([a_1, b_1] \times \cdots \times [a_4, b_4]) \cap \Z^4$. By going to the dual lattice, we have that the fundamental group of $\twoskelminus{\plaqset}$ is the same as the fundamental group of the complement of a 2-complex $M$ made of $|\plaqset|$ plaquettes in $B := [-(a_1+ 1/2), b_1+1/2] \times \cdots \times [-(a_4 + 1/2), b_4 + 1/2]$. Pick a point $x_0 \in B - M$, and take non intersecting paths in $B$ from $x_0$ to the center of each plaquette in $M$, such that the paths don't intersect $M$ anywhere else. This is possible by the transversality theorem \cite{GP1974} (and since $1 + 2 < 4$). These paths form a ``star" $S$. Let $N$ be a small neighborhood of $S$, such that $N$ intersects the plaquettes of $M$ in small disks around the center points. Let $Q := B \backslash N$, and observe that $Q$ is homeomorphic to $S^3 \times [0, 1]$. Now decompose $B \backslash M = (N \backslash M) \cup (Q \backslash M)$. 

Observe that the 2-complex $M$ with the center point of each plaquette removed deformation retracts to its 1-skeleton. Thus $Q \backslash M$ is homotopy equivalent to $S^3 \times [0, 1]$ minus a 1-dimensional space, and thus by the transversality theorem \cite{GP1974} (and since $1 + 2 < 4$), $Q \backslash M$ is simply connected. Moving to $N \backslash M$, observe that this space is homeomorphic to a four dimensional cube with $|\plaqset|$ parallel 2-dimensional hyperplanes removed, and thus $\pi_1(N \backslash M)$ is the free group on $|\plaqset|$ generators. Now finish by the Seifert-Van Kampen theorem.
\end{proof}




\section*{Index of notation}

There are some quantities which are defined in both the Abelian case and the general case. This was done to emphasize the analogies between certain aspects of the proofs in the Abelian and general cases. The result is that some entries of the index point to multiple pages; the earlier page contains the definition in the Abelian case, while the later page contains the definition in the general case. 

\begin{longtable}{lp{0.5in}p{0.610\textwidth}}
    \centering
         Notation & Page & Description  \\
         \hline
         $G$ & \pageref{notation:G} & Gauge group \\
         $\groupid$ & \pageref{notation:group-id} & Identity of $G$ \\
         $\rho$ & \pageref{notation:rho} & Unitary representation $G$ \\
         $\chi$ & \pageref{notation:chi} & Character of $\rho$ \\
         $\lbox$ & \pageref{notation:Lambda} & Finite 4d cube \\
         $p$ & \pageref{notation:plaquette} & Plaquette \\
         $\beta$ & \pageref{notation:beta} & Inverse coupling constant \\
         $\mu_{\lbox, \beta}$ & \pageref{notation:mu-lbox-beta} & Lattice gauge theory \\
         $\wloop$ & \pageref{notation:wloop} & Closed loop \\
         $W_\wloop$ & \pageref{notation:W-wloop} & Wilson loop variable \\
         $\Delta_G$ & \pageref{notation:Delta-G} & Strictly positive real number defined in terms of $G$ \\
         $\varphi_\beta$ & \pageref{notation:varphi-beta} & Function on $G$\\
         $r_\beta$ & \pageref{notation:r-beta} & Sum of $\varphi_\beta(g)$ over $g \in G, g \neq \groupid$\\
         $\elemmatrix$ & \pageref{notation:elemmatrix} & Weighted average of $\rho(g)$ over $g \in G, g \neq 1$, with weight $\varphi_\beta(g)$\\
         $\ell$ & \pageref{notation:ell} & Length of $\wloop$ \\
         $df$ & \pageref{notation:df} & Exterior derivative of $f$ \\
         $\delta f$ & \pageref{notation:delta-f} & Coderivative of $f$ \\
         $\langle f, g \rangle$ & \pageref{notation:inner-product-f-g} & Inner product of differential forms\\
         $N$ & \pageref{notation:N} & Side length of $\lbox$ \\
         $B_\wloop$ & \pageref{notation:B-wloop} & Cube of side length $\ell$ which contains $\wloop$ \\
         $L$ & \pageref{notation:L} & $\ell^\infty$ distance between the boundaries of $B_\wloop$ and $\lbox$ \\
         $\plaqset$ & \pageref{notation:plaqset} & Plaquette set \\
         $\vortex$ & \pageref{notation:vortex} & Vortex \\
         $\Phi$ & \pageref{notation:Phi-abelian}, \pageref{notation:Phi-nonabelian} & Function on plaquette sets \\
         $\Sigma$ & \pageref{notation:Sigma} & Random edge configuration with law $\mu_{\lbox, \beta}$ \\
         $\plaqset(\Sigma)$ & \pageref{notation:P-Sigma} & Random plaquette set, support of $d\Sigma$ \\
         $E$ & \pageref{notation:E-abelian}, \pageref{notation:E-nonabelian} &  Event which captures the typical behavior\\
         $N_\wloop$ & \pageref{notation:N-wloop-abelian}, \pageref{notation:E-nonabelian} & Count of weakly dependent rare events\\
         $\alpha_\beta$ & \pageref{notation:alpha-beta-abelian}, \pageref{notation:alpha-beta-nonabelian} & Upper bound on $\Phi$\\
         $\mc{P}_\lbox$ & \pageref{notation:P-lbox} & Set of nonempty vortices contained in $\lbox$ \\
         $\Xi_{\mc{P}}$ & \pageref{notation:Xi-P-abelian}, \pageref{notation:Xi-P-nonabelian} & Roughly, a restricted partition function \\
         $N_\lbox(V_1, \ldots, V_n)$ & \pageref{notation:N-lbox} & Vortices which are incompatible with at least one of $V_1, \ldots, V_n$ \\
         $\rho_{\mc{P}}$ & \pageref{notation:rho-P-abelian}, \pageref{notation:rho-P-nonabelian} & Reduced correlations \\
         $\oneskel$ & \pageref{notation:oneskel} & 1-skeleton of $\lbox$ \\
         $\twoskel$ & \pageref{notation:twoskel} & 2-skeleton of $\lbox$ \\
         $x_0$ & \pageref{notation:x-0} & Base point of $\oneskel$ \\
         $T$ & \pageref{notation:T} & Spanning tree of $\oneskel$ \\
         $a_e$ & \pageref{notation:a-e} & A closed loop corresponding to $e$ \\
         $\homsym_T^{x_0}(\sigma)$ & \pageref{notation:homsym} & Homomorphism induced by $\sigma$ \\
         $C_p$ & \pageref{notation:C-p} & Closed loop which winds around the boundary of $p$ \\
         $\supp(\sigma)$ & \pageref{notation:supp-sigma} & The set of plaquettes $p$ for which $\sigma_p \neq \groupid$ \\
         $GF(T)$ & \pageref{notation:GF-T} & Set of edge configurations which are identity on $T$ \\ 
         $S_e$ & \pageref{notation:S-e} & A certain 2-complex corresponding to $e$ \\
         $S_2(B)$ & \pageref{notation:S-2-B} & 2-skeleton of $B$ \\
         $\partial S_2(B)$ & \pageref{notation:partial-S-2-B} & Roughly, the 2-skeleton of the boundary of $B$ \\
         $S_2^c(B)$ & \pageref{notation:S-2-c-B} & Roughly, the 2-skeleton of the complement of $B$\\
         $K$ & \pageref{notation:K} & Knot \\
         $\mc{K}$ & \pageref{notation:script-K} & Collection of all knots in $\lbox$\\
         $\Gamma$ & \pageref{notation:Gamma} & Collection of vortices, i.e. a subset of $\mc{P}_\lbox$ \\
         $P(\Gamma)$ & \pageref{notation:P-Gamma} & Plaquette set which is the union of elements of $\Gamma$ \\
         $I(\Gamma)$ & \pageref{notation:I-Gamma} & Indicator which tracks whether the elements of $\Gamma$ are compatible \\
         $\randomv$ & \pageref{notation:randomv} & Random collection of vortices \\
         $\twoskelminus{\plaqset}$ & \pageref{notation:twoskelminus} & 2-complex obtained by deleting the plaquettes of $P$ from $\twoskel$ \\
         $G^\scale(P)$ & \pageref{notation:G-scale-P} & A certain graph induced by $P$ \\
         $\mc{A}(m, \scale)$ & \pageref{notation:A-m-scale} & Plaquette sets for which $G^s(P)$ is connected \\
         $\mc{A}(m, \scale, p)$ & $\pageref{notation:A-m-scale-p}$ & Plaquette sets containing $p$, for which $G^s(P)$ is connected
    \label{tab:notation-index}
\end{longtable}

\section*{Acknowledgements}

I thank my Ph.D. advisor Sourav Chatterjee for suggesting that I begin in this area, as well as for helpful conversations, encouragement, and valuable comments. I thank Persi Diaconis for helpful conversations about group theory and representation theory. I thank Ciprian Manolescu and Hongbin Sun for helping me with algebraic topology; I am particularly indebted to Ciprian Manolescu for providing a proof of Lemma \ref{lemma:rank-bound}. Finally, I would like to thank the anonymous referee for the many valuable suggestions and comments.

\end{document}